\documentclass[reqno]{amsart}

\usepackage[OT1]{fontenc}     

\usepackage{amssymb}          
\usepackage{mathrsfs}         
\usepackage{bbm}              
\usepackage{braket}
\usepackage{stmaryrd}
\usepackage{upref}

\usepackage{graphicx}
\usepackage{graphbox}         
\usepackage{tikz}
\usetikzlibrary{cd}           
\usetikzlibrary{knots}        
\usetikzlibrary{arrows.meta}
\usepackage{tikz-3dplot}      
\usepackage{pdfpages}

\usepackage[shortlabels]{enumitem}
\usepackage{setspace}
\usepackage{relsize}
\usepackage{comment}

\usepackage[margin=1in]{geometry}

\usepackage{natbib}
\usepackage{hyperref}
\usepackage[capitalise,noabbrev,nameinlink]{cleveref}

\definecolor{cubeblue}{RGB}{60,110,220}          
\definecolor{cubeorange}{RGB}{220,170,60}        
\definecolor{skyblue}{RGB}{100,230,250}          
\definecolor{brightpink}{RGB}{252,111,248}       
\definecolor{stalkgreen}{RGB}{110,190,140}       
\definecolor{shadowred}{RGB}{160,45,45}
\definecolor{shadowredbright}{RGB}{200,65,65}
\definecolor{buildingred}{RGB}{204,50,50}        

\newcommand{\chamberScale}{0.69}   
\newcommand{\resolvedScale}{0.72}  
\newcommand{\localScale}{0.66}      
\newcommand{\preimageScale}{0.525} 
\newcommand{\stalkScale}{0.71}     
\newcommand{\rayScale}{0.78}       
\newcommand{\buildingScale}{0.525} 

\tikzset{
	axono/.style={
		x={(2.6cm,0cm)},
		y={(1.05cm,0.65cm)},
		z={(0cm,2.6cm)},
		line join=round,
	},
}

\tikzset{
	cubepic/.style={axono, scale=\chamberScale},
	cubepic2/.style={axono, scale=\resolvedScale}, 
	sheet/.style   = {cubeblue, fill opacity=1},    
	chamber/.style = {cubeblue, fill opacity=0.27}, 
	sheetorange/.style   = {cubeorange, fill opacity=1},
	chamberorange/.style = {cubeorange, fill opacity=0.27},
	edge/.style     = {thick},
	hidden/.style   = {thick, dashed}, 
	occluded/.style = {thick, dotted}, 
	glued/.style    = {line width=1.5pt},
}

\tikzset{
	localpic/.style={axono, scale=\localScale, line cap=round},
	face/.style        = {skyblue, fill opacity=1},
	loosedotted/.style = {thick, loosely dotted},
}

\tikzset{
	preimagepic/.style={axono, scale=\preimageScale, line cap=round},
	sheetpink/.style   = {brightpink, fill opacity=1},
	chamberpink/.style = {brightpink, fill opacity=0.27},
	normalvec/.style   = {-{Triangle[length=3mm,width=2.5mm]},
		line width=1.6pt, green!60!black}, 
}

\tikzset{
	stalkproj/.style     = {axono, line cap=round},                      
	stalkprojskew/.style = {axono, line cap=round, z={(-0.45cm,2.6cm)}}, 
	stalkedge/.style     = {stalkgreen, semithick},
	stalkface/.style     = {stalkgreen, fill opacity=0.27},
	hookstart/.style     = {{Hooks[left]}-},                  
	bigtip/.style        = {-{>[length=5mm,width=3.5mm]}},    
	bigtip3/.style       = {-{>[length=4mm,width=3mm]}},      
	segvis/.style        = {line width=1.2pt, black},         
	segocc/.style        = {line width=1.2pt, black!45},      
	segvis3/.style       = {line width=1.1pt, black},         
	segocc3/.style       = {line width=1.1pt, black!45},      
	stalknormal/.style   = {-{Stealth[length=1.8mm]}, line width=1.2pt}, 
}

\tikzset{
	wavyplane/.style = {fill=blue!25, fill opacity=0.27, draw=blue!55!black, thin}, 
	outward/.style   = {green!45!black},   
	inward/.style    = {green!45!black},   
	rayvis/.style    = {line width=1.834pt, black},    
	rayocc/.style    = {line width=1.834pt, black!65}, 
	raytip/.style    = {-{Stealth[length=10.700mm,width=7.643mm]}},
	normaltip/.style = {-{Stealth[length=4.891mm,width=3.516mm]}, line width=3.057pt}, 
}

\tikzset{
	buildingpic/.style  = {axono, scale=\buildingScale}, 
	buildingpic1/.style = {axono, scale=\buildingScale}, 
	chamberred/.style   = {buildingred, fill opacity=0.27}, 
	sheetred/.style     = {buildingred, fill opacity=1},    
	darkredfill/.style  = {buildingred, fill opacity=1},    
}

\newcommand{\cubefaces}[2]{%
	\fill[#2] (A#1) -- (B#1) -- (E#1) -- (C#1) -- cycle;
	\fill[#2] (A#1) -- (C#1) -- (G#1) -- (D#1) -- cycle;
	\fill[#2] (C#1) -- (E#1) -- (H#1) -- (G#1) -- cycle;
	\fill[#2] (A#1) -- (B#1) -- (F#1) -- (D#1) -- cycle;
	\fill[#2] (B#1) -- (E#1) -- (H#1) -- (F#1) -- cycle;
	\fill[#2] (D#1) -- (F#1) -- (H#1) -- (G#1) -- cycle;
}

\newcommand{\cubecoords}[2]{%
	\begin{scope}[shift={#2}]
		\coordinate (A#1) at (0,0,0);
		\coordinate (B#1) at (1,0,0);
		\coordinate (C#1) at (0,1,0);
		\coordinate (D#1) at (0,0,1);
		\coordinate (E#1) at (1,1,0);
		\coordinate (F#1) at (1,0,1);
		\coordinate (G#1) at (0,1,1);
		\coordinate (H#1) at (1,1,1);
	\end{scope}%
}

\newcommand{\markpoint}[3]{%
	\fill[black] (#1) circle (3.6pt);
	\node [label={[label distance=#3]#2:p}] at (#1) {};%
}

\newcommand{\markpointas}[4]{%
	\fill[black] (#1) circle (3.6pt);
	\node [label={[label distance=#3]#2:#4}] at (#1) {};%
}

\newcommand{\markpointc}[4]{%
	\fill[black] (#1) circle (3.6pt);
	\node[anchor=#4] at ([shift={(canvas polar cs:angle=#2,radius=#3)}]#1) {p};%
}

\newcommand{\boldstamp}[4][]{%
	\foreach \dx/\dy in {0pt/0pt, 0.3pt/0pt, 0pt/0.3pt, 0.3pt/0.3pt} {
		\node[font=\Large\bfseries, #4, #1, xshift=\dx, yshift=\dy] at (#2) {#3};
	}%
}

\theoremstyle{plain}
\newtheorem{thm}{Theorem}[section]
\newtheorem{lem}[thm]{Lemma}

\newtheorem{prop}[thm]{Proposition}

\newtheorem{cor}[thm]{Corollary}

\theoremstyle{definition}
\newtheorem{defn}[thm]{Definition}

\theoremstyle{remark}

\theoremstyle{plain}

\newtheorem*{maintheorem}{Main Theorem}
\newtheorem*{thm*}{Theorem}

\theoremstyle{definition}

\theoremstyle{remark}

\theoremstyle{plain}

\crefname{thm}{Theorem}{Theorems}
\crefname{lem}{Lemma}{Lemmas}
\crefname{conj}{Conjecture}{Conjectures}
\crefname{prop}{Proposition}{Propositions}
\crefname{spec}{Speculation}{Speculations}
\crefname{obs}{Observation}{Observations}
\crefname{cor}{Corollary}{Corollaries}
\crefname{defn}{Definition}{Definitions}
\crefname{excs}{Exercise}{Exercises}
\crefname{rmk}{Remark}{Remarks}
\crefname{maintheorem}{Main Theorem}{Main Theorems}
\crefname{theorem}{Theorem}{Theorems}
\crefname{lemma}{Lemma}{Lemmas}
\crefname{conjecture}{Conjecture}{Conjectures}
\crefname{proposition}{Proposition}{Propositions}
\crefname{speculation}{Speculation}{Speculations}
\crefname{observation}{Observation}{Observations}
\crefname{definition}{Definition}{Definitions}
\crefname{exercise}{Exercise}{Exercises}
\crefname{remark}{Remark}{Remarks}

\makeatletter
\newcommand{\crefsharedcounter}[2]{
	\AddToHook{env/#1/begin}{\@namedef{cref@#2@alias}{#1}}}
\makeatother
\crefsharedcounter{lem}{thm}
\crefsharedcounter{conj}{thm}
\crefsharedcounter{prop}{thm}
\crefsharedcounter{spec}{thm}
\crefsharedcounter{obs}{thm}
\crefsharedcounter{cor}{thm}
\crefsharedcounter{defn}{thm}
\crefsharedcounter{excs}{thm}
\crefsharedcounter{rmk}{thm}
\crefsharedcounter{lemma}{theorem}
\crefsharedcounter{conjecture}{theorem}
\crefsharedcounter{proposition}{theorem}
\crefsharedcounter{speculation}{theorem}
\crefsharedcounter{observation}{theorem}
\crefsharedcounter{definition}{theorem}
\crefsharedcounter{exercise}{theorem}
\crefsharedcounter{remark}{theorem}

\newcommand{\sph}{{\mathbb{S}}}
\newcommand{\rl}{{\mathbb{R}}}

\DeclareMathOperator{\Imm}{Imm}

\newcommand{\Pprod}[1]{\ensuremath{\mathop{\vcenter{\hbox{\scalebox{2.5}{$\prod$}}}}\limits_{#1}}}
\newcommand{\DDelta}{\ensuremath{\vcenter{\hbox{\scalebox{2.75}{$\Delta$}}}}}
\newcommand{\dDelta}{\ensuremath{\vcenter{\hbox{\scalebox{1.5}{$\Delta$}}}}}
\newcommand{\ccdot}{\ensuremath{\vcenter{\hbox{\scalebox{2}{$\cdot$}}}}}

\newcommand{\abs}[1]{\left\vert#1\right\vert}

\newcommand\restr[2]{{%
		\left.\kern-\nulldelimiterspace
		#1
		\vphantom{\big\vert}
		\right\vert_{#2}
}}

\newcommand{\ol}{\overline}

\newcommand{\wh}{\widehat}

\title{Which Immersions into Euclidean Space are Boundary Immersions?}
\author{Austin Konkel}
\address{Department of Mathematics, Stony Brook University, Stony Brook, NY 11794, USA}
\email{austin.konkel@stonybrook.edu}
\date{October 2026}

\begin{document}

\begin{abstract}

We study the problem of whether a given codimension one immersion \(f\) of a closed and connected manifold into Euclidean space can be extended to an equidimensional immersion \(F\) of some compact manifold \(M\) whose restriction to the boundary parameterizes the image \(S\) of \(f\). An extension \(F\) exists if and only if (i) simple conditions on orientations and wrapping numbers are satisfied and (ii) a distinguished finite subset of the first cohomology of \(S\) with coefficients in a certain sheaf of finite groups contains the preferred point. 

\end{abstract}

\maketitle

 \section{Introduction}

Suppose you have a connected manifold that is almost parallelizable. Remove a ball. Then you can immerse what remains into Euclidean space of the same dimension. You wonder what this immersed boundary sphere looks like because its complement inside space can be filled in by the rest of the manifold. We analyze this situation by building a sheaf of groups whose first cohomology contains information about this immersed sphere. Alongside some wrapping number conditions that are easy to define, this first cohomology set characterizes extensions up to a natural equivalence in all dimensions. When the manifold less a ball \(M\) is contractible, it is a celebrated open question in dimension 4 whether \(M\) is combinatorially a ball, and a fairly recent theorem in dimension 3 that \(M\) is combinatorially a ball. In all other dimensions, it is known that \(M\) is combinatorially a ball. This construction of the nature of the immersion might give information about these missing dimensions.

\begin{figure}[h]
	\centering
	\includegraphics[scale=0.22]{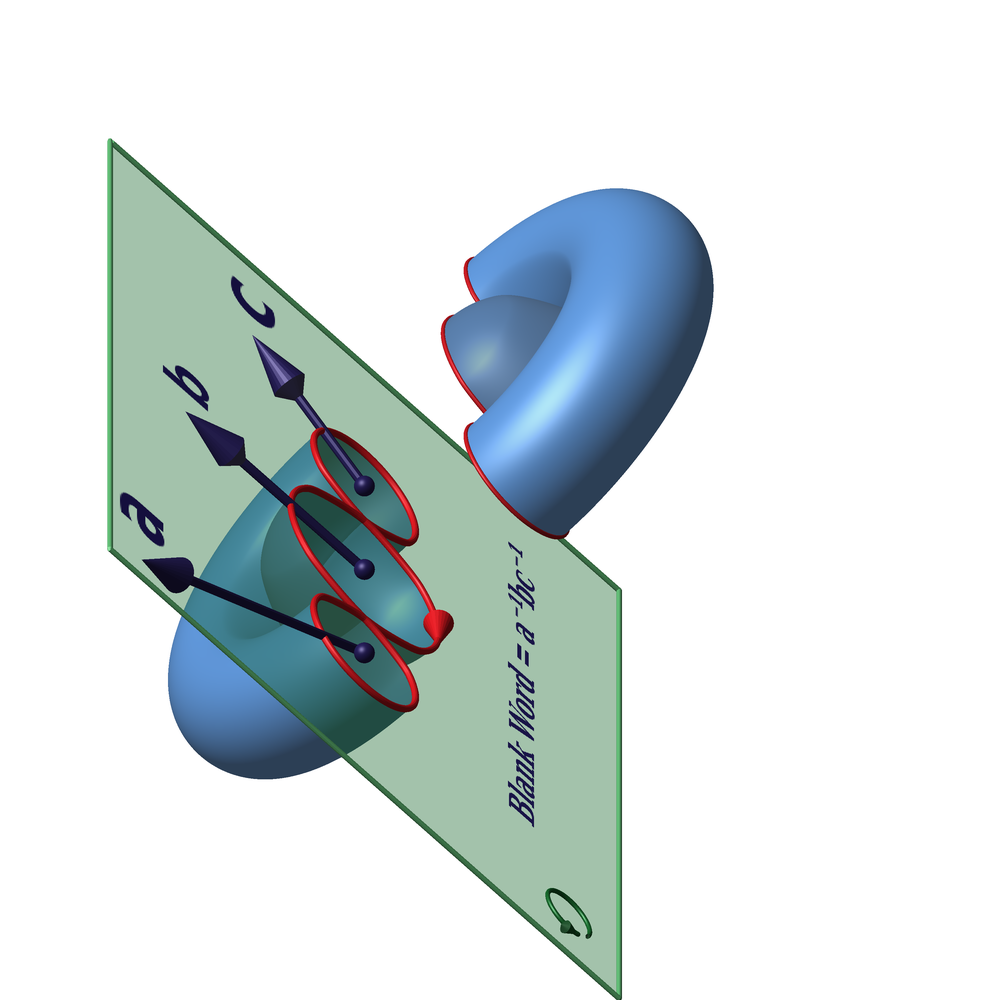}
	\caption{A portrayal of a dimension-reduction argument given by \cref{prop:transversal slice} wherein one transversely slices the image \(S\). Here an algorithm of Sandy \cite{blank1967} is applied to obstruct an extension by obstructing one inside a transversely intersecting plane.}
	\label{fig:transversal slice}
\end{figure}



\section{History and Motivation}

Equidimensional immersions are powerful: one can transport any local structure on the target back to the domain to tautologically obtain a local isomorphism in the category of structure being considered.   Once Tony \cite{Phillips1967} proved the submersion theorem for open manifolds,\footnote{David \cite{spring2005} refers to the years from 1959 to 1973 as the early `golden' period for immersion theory.} the existence of many equidimensional immersions was guaranteed.  For example, Eduard \cite{Stiefel1935} had already shown that all orientable 3-manifolds are parallelizable,\footnote{Refer to \cite{BaisZuddas2023} for a more modern account of \citeauthor{Stiefel1935}'s parallelizability.} so the image of some generic immersion of the 2-sphere into 3-space that bounds an immersed copy of your favorite such 3-manifold could then always be built and placed on your coffee table at home.\footnote{One month before his death, J. H. C. \cite{Whitehead1961ImmersionOpen3Manifold} was writing about immersing open 3-manifolds into Euclidean 3-space.} Down a dimension, constructing immersions of Riemann surfaces with boundary into the plane is classical. Being \emph{almost parallelizable}\footnote{We say a smooth manifold is \emph{almost parallelizable} if that manifold is connected and parallelizable after removing a point.} is the clarifying concept which drives applications of \citeauthor{Phillips1967}' submersion theorem. Further examples of almost parallelizable manifolds include all connected spin 4-manifolds and all smooth \(d\)-manifolds homotopy equivalent to the \(d\)-sphere. Another interesting classical corollary of this submersion theorem is that every open and oriented 4-manifold admits a complex structure since each can be immersed into the complex projective plane less a point, see Dennis \cite{Sullivan2023}.

One knows that a closed, oriented, and smooth manifold is an oriented boundary if and only if all its Stiefel--Whitney and Pontryagin numbers vanish. Stephen \cite{smale1959} and Morris \cite{hirsch1959} determine when one \(d\)-manifold immerses into another by working handle by handle, up to the top handles. Their argument reduces the question from the title in codimension two or greater to homotopy theory. Extending in codimension one is inherently a codimension zero problem because our equidimensional extension problem is posed relative to the boundary.  We describe a functor from the category of codimension one immersions in general position inside Euclidean space, with morphisms being smooth isotopies, to the category of sheaves of finite groups, with morphisms being isomorphisms. Any first cohomology set over the image \(S\) obtained using this functor characterizes the extensions of \(S\) with a compact domain up to a natural equivalence, revealing in the process new information and structure concerning this pervasive top-cell problem.

 Andr\'e \cite{Gramain1970a,Gramain1970b} studies specifically extensions of codimension one immersions of spheres into spheres and into Euclidean spaces. 
\citeauthor{Gramain1970a} remarks that the immersion theory of \citeauthor{smale1959} and  \citeauthor{hirsch1959} produces a bijection between the connected components \(\pi_0\left(\Imm{\left(\sph^d,\rl^{d+1}\right)}\right)\) of our space of immersions and a homotopy group \(\pi_d\left(SO\left(d+1\right)\right)\) of a special orthogonal group.  \citeauthor{Gramain1970b} uses the well-known \(J\)-homomorphism \(J : \pi_k\left(SO\left(d\right)\right) \to \pi_{k+d}\left(\sph^d\right)\) to prove that if an immersion of a sphere extends, then after applying the remarked bijection, the equivalence class of that immersion in \(\pi_0\left(\Imm{\left(\sph^d,\rl^{d+1}\right)}\right)\) is in the kernel of the \(J\)-homomorphism. The image of the \(J\)-homomorphism was famously studied by Frank \cite{AdamsJGroups}.

 All points on the image of \emph{a generic codimension one immersion} are modeled locally on finitely many hyperplanes meeting in general position inside Euclidean space. Christian \cite{Pappas1996} uses Morse theory to produce all extensions level by level according to a Morse function on the range. \citeauthor{Pappas1996} identifies geometric obstructions to continuing extensions past the critical points, and he identifies which critical points have choices. \citeauthor{Pappas1996} asked whether the extending domains of the same codimension one immersion need be homeomorphic. Our sheaf structure below gives an algebraic obstruction theory approach to this question.
 
  The level sets of some immersed spheres are portrayed in \cref{fig:sphere,fig:dimple,fig:tube,fig:s1xs2,fig:giller,fig:ferry,fig:milnorpush}.  \cref{fig:dimple,fig:milnorpush}  portray simple and controlled homotopies through immersions that destroy the property of admitting an extension. We call an element in this special class of simple and controlled homotopies \emph{a finger move}. For example, the immersed sphere in \cref{fig:dimple} does not admit an extension, but is exactly one finger move away from the embedding in \cref{fig:sphere}. By contrast, the immersed sphere in \cref{fig:tube} admits an extension and is four finger moves away from the embedding. \cref{fig:milnorpush} portrays an immersed sphere that does not admit an extension, yet is an easy finger move away from an immersion that extends in exactly two inequivalent ways.

The earliest references we found that study immersions combinatorially are due to Charles \cite{Titus1960, Titus1961}, who studies immersed disks into the plane. Studying immersions of the disk is motivated by complex analysis.  Around that time, Robert Gunning was also looking into finding the possible holomorphic maps of a disk into the complex plane taking the boundary of the disk to a given immersed circle. His study arose in examining flat structures on Riemann surfaces \citep{GunningNarasimhan1967}. This approach towards a characterization works because one may transport any complex structure on the plane back to the disk under an equidimensional immersion to tautologically obtain a holomorphic map. The thesis of Sandy \cite{blank1967} makes progress towards Gunning's question in the generic case without branching. \cref{fig:transversal slice} portrays  \citeauthor{blank1967}'s combinatorial if and only if algorithm, which decides exactly which generic immersions of the circle into the plane extend to an immersion of the disk.  Valentin \cite{Poenaru1968} provides a Bourbaki account of \citeauthor{blank1967}'s work. Morris \cite{Marx1965, Marx1968, Marx1974} makes progress in the generic case with branching.  \cite{MarxVerhey1970} study extending generic immersions of the circle into the plane with polynomial regularity and obtain a sufficient condition for existence. George \cite{Francis1969, Francis1970, Francis1973, Francis1978} makes several contributions, and Keith \cite{Bailey1975} studies a higher-genus situation. Margaret \cite{McIntyre1997} studies extensions of generic immersions of the circle into any closed and orientable surface, and provides bounds on analogues of the wrapping numbers.  Dennis \cite{Frisch2010Sphere} studies extending generic immersions of the circle into the 2-sphere by disks. A unified treatment of the ideas by all authors in this line of work would be valuable. Since compact, orientable, and connected 2-manifolds with exactly one boundary component are determined up to homeomorphism by their Euler characteristic, the topology of any extension is known a priori via the theorem of Gauss--Bonnet. How does the topology of a manifold constrain its immersed boundaries?
 
 Here is another question. Immerse a closed, connected, and almost parallelizable \(d\)-manifold less a ball into flat Euclidean \(d\)-space. Transport the metric back to the domain less that ball. By elementary obstruction theory, one can extend the Riemannian metric over the removed ball to concentrate all nonzero curvature into that small open ball.  How does the  curvature of a Riemannian manifold constrain its immersed boundaries?

  Recall that \citeauthor{Phillips1967} showed us that generic immersions of the 2-sphere into Euclidean 3-space which extend to an immersion of some compact 3-manifold whose immersed 2-sphere boundary parameterizes the given immersed 2-sphere form a rich space. \cite{MasudaNagase2001}, for the 2-sphere, and later Bojun \cite{zhao2019}, for surfaces, describe finite algorithms for finding extensions of a generic immersion into 3-space. The finite approach presented in what follows has the advantage of being formulated within the well-known framework of sheaf theory and cohomology. Furthermore, our methods apply in any dimension.
 
  The explicit combinatorics here describing how finitely many compact covering spaces are adhered together along their boundaries  yield the gluing data of a triangulation of any extending 3-manifold. Given a triangulated 3-manifold, there is the problem of recognizing whether that manifold is the 3-ball. The now-proven Poincar\'e conjecture in dimension 3 implies that being compact and contractible is sufficient. Independently of this big theorem, the 3-ball recognition problem has been solved. There exist finite-time algorithms going back to \cite{Rubinstein1995}, \cite{Thompson1994}, and others,\footnote{Refer to the article by Joel \cite{Hass1998Algorithms} for a survey on algorithms for solving problems in 3-dimensional topology.} that decide whether a triangulated 3-manifold  is the 3-ball. Combined with these algorithms, our results answer Problem 3.19 of Robion \cite{Kirby1997}, which asks which immersed 2-spheres in 3-space bound immersed 3-balls. 

 Rob \cite{kusner2005conformal} similarly asked whether there is a computable way to detect whether an immersed surface in 3-space is \emph{Alexandrov immersed}, meaning that the immersion extends to an immersion of a 3-manifold it bounds. The obstruction theory developed here answers \citeauthor{kusner2005conformal}'s question in the compact case, and the noncompact case seems within reach. \citeauthor{kusner2005conformal} ruled out any extension of the noncompact Enneper surface by an argument that is essentially \cref{prop:transversal slice}. It would be valuable to understand the computational complexity of these finite obstruction-theoretic calculations. In low dimensions, this problem arises in computer graphics, where 3D models collide in practice and intersect generically with probability one. This has been studied by \citet*{LiBarbic2018} and \citet*{Sassen2024}. 

\citet*{Buchbinder2026MinimalHomotopies3D} study extensions of generic immersions of 2-spheres into 3-space to Lipschitz 3-balls and show that the minimal volume is at least a sum of the volumes of the components of the complement of the immersed sphere, each weighted by its wrapping number. The bound is attained exactly when there is an immersed ball. They also give a linear-time algorithm for computing the wrapping numbers. The Lipschitz regularity of 4-manifolds raises important open questions.

From \citet*{carraracartersaito2001}, one knows that any smooth embedding of the 2-sphere into 4-space admits some smooth projection down into 3-space that is a generic immersion. Given any such 2-sphere, an algorithm from \cite{CarterSaito1997Seifert} produces an embedded 3-manifold that it bounds. Flexibility established in what follows might be used to study other extension problems, such as ones involving these higher-dimensional Seifert surfaces. To do this, one simply does for the components of the complement of the image with negative wrapping number exactly what is done for the components with positive wrapping number, just reversing arrows and treating them as positive. Other mysterious examples to try and apply this structure to are the knotted 3-balls inside the 4-sphere of \cite{BudneyGabai2021}. Take two of their examples bounded by the same great 2-sphere of the 4-sphere, then glue them together along that sphere. This produces a map of the 3-sphere into the 4-sphere that is an immersion away from a fold along a great equatorial 2-sphere. What do these once folded immersed 3-spheres and the 4-balls they bound look like?




As portrayed in \cref{fig:local pictures}, the image of any generic immersion of the 2-sphere into 3-space is locally homeomorphic to some 2-dimensional real algebraic variety, so the work of \cite{akbulutking1992} proves that the image is in fact homeomorphic to a real algebraic variety. Real algebraic varieties admit a stratified structure \`a la \cite{Whitney1965Tangents}. Hence, so too does any generic immersed image via transport of structure. One may read about the related concept of an equisingular stratification in the work of Ren{\'e} \cite{Thom1969} and John \cite{Mather2012}. We can see this stratification directly from the local models: the image is equisingularly stratified by its single, double, and triple points. An important theorem in algebraic geometry is one of resolution of singularities in characteristic zero proven by Heisuke \cite{Hironaka1964}, but this problem remains open in characteristic \(p\). Our singularities here have very simple local pictures, and resolution techniques have been used to great effect in this theory before by Thomas \cite{banchoff1974a} and Tobias \cite{Ekholm2001}.

\emph{The shape space} of images of generic codimension one immersions of the \(d\)-sphere into \(\left(d+1\right)\)-space can be defined in the following way. Let \(\Imm{\left(\sph^d,\rl^{d+1}\right)}\) be the Fr\'echet manifold of immersions of the \(d\)-sphere into \(\left(d+1\right)\)-space and \(\Imm_{\text{gen}}\left(\sph^d,\rl^{d+1}\right)\) its open and dense submanifold of generic immersions. Let \(\text{Diff}\left(\sph^d\right)\) be the diffeomorphism group of the \(d\)-sphere, which acts on \(\Imm_{\text{gen}}\left(\sph^d,\rl^{d+1}\right)\) from the right. The image of any generic immersion of the \(d\)-sphere into \(\left(d+1\right)\)-space is then some point in the quotient space  \(\Imm_{\text{gen}}\left(\sph^d,\rl^{d+1}\right)/\text{Diff}\left(\sph^d\right)\).  A corollary of \citet*{CerveraMascaroMichor1991} is that the shape space of generic immersions of the \(d\)-sphere into \(\left(d+1\right)\)-space has the structure of a Fr\'echet manifold that is the base space of a smooth principal \(\text{Diff}\left(\sph^d\right)\)-bundle. The following discussion concerns generic immersions of the 2-sphere into 3-space, but the coming definitions, arguments, and patterns generalize readily.


\section{Definitions and Discussion}

 \emph{An immersion} is a smooth map between smooth manifolds whose derivative is injective everywhere, so locally an embedding. A smooth map into Euclidean space is \emph{transverse} if, wherever two points have the same image, the images of their tangent spaces under the derivative together span that Euclidean space. Choose a smooth structure on the 2-sphere and on 3-space. A \emph{generic} immersion of the 2-sphere into 3-space is a transverse immersion such that each point in the image can be locally modeled\footnote{A precise definition of \emph{an atlas of local models} for a generically immersed 2-sphere in 3-space is provided in \(\S 6\).} by one of the varieties portrayed in \cref{fig:local pictures}. \emph{A generically immersed 2-sphere in 3-space} is the image of a generic immersion of the 2-sphere into 3-space, endowed with the subspace topology. Two generically immersed 2-spheres in 3-space are \emph{equivalent} if there exists an ambient smooth isotopy of 3-space that brings one subset to the other. A generically immersed 2-sphere \(S\) in 3-space is \emph{fillable} if there exists a compact 3-manifold \(M\) immersed inside 3-space such that \(S\) is the immersed 2-sphere boundary of \(M\). In other words, there exists an immersion \(F\) of \(M\) into 3-space such that the restriction of \(F\) to the boundary of \(M\) parameterizes \(S\). Any such extension \(F\) is called \emph{a filling of \(S\) by \(M\)}. Two fillings of \(S\) by \(M\) and by \(M'\) are \emph{equivalent} if there exists a diffeomorphism from \(M\) to \(M'\) that forms a commutative triangle over their common images inside 3-space. The property of being fillable is preserved under the equivalence of ambient smooth isotopy, and the domain of any filling is orientable because any orientation on 3-space may be transported back to the domain via the filling.

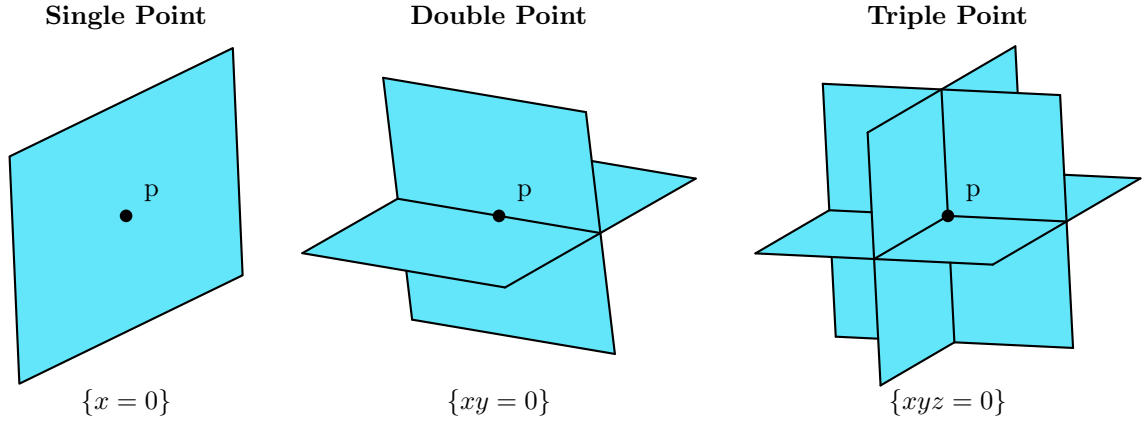
\begin{figure}[h!]
	{\centering
		\begin{minipage}{\textwidth}
			\centering
			
			\begin{tabular}{@{}c@{\hspace{0.3in}}c@{\hspace{0.3in}}c@{}}
				\textbf{Single Point} & \textbf{Double Point} & \textbf{Triple Point}
				\\[1.5mm]
				
				\(\vcenter{\hbox{\begin{tikzpicture}[localpic]
							
							\coordinate (ymzm) at (-0.611616,-0.522736,-1.163045);
							\coordinate (ypzm) at (0.445524,1.121011,-0.738134);
							\coordinate (ypzp) at (0.611616,0.522736,1.163045);
							\coordinate (ymzp) at (-0.445524,-1.121011,0.738134);
							
							\fill[face] (ymzm) -- (ypzm) -- (ypzp) -- (ymzp) -- cycle;
							\draw[edge] (ymzm) -- (ypzm) -- (ypzp) -- (ymzp) -- cycle;
							
							\markpointc{0,0}{30}{2mm}{south west}
				\end{tikzpicture}}}\)
				&
				\(\vcenter{\hbox{\begin{tikzpicture}[localpic]
							
							\coordinate (xmym) at (-1.296752,-0.543584,-0.151492);
							\coordinate (xmzm) at (-0.896071,0.559946,-0.939956);
							\coordinate (xm) at (-0.929532,0.366501,0.040584);
							\coordinate (xmzp) at (-0.962993,0.173057,1.021125);
							\coordinate (xmyp) at (-0.562312,1.276587,0.232661);
							\coordinate (xpym) at (0.562312,-1.276587,-0.232661);
							\coordinate (xpzm) at (0.962993,-0.173057,-1.021125);
							\coordinate (xp) at (0.929532,-0.366501,-0.040584);
							\coordinate (xpzp) at (0.896071,-0.559946,0.939956);
							\coordinate (xpyp) at (1.296752,0.543584,0.151492);
							
							\coordinate (yzaux1) at (-0.909605,0.481703,-0.543354);
							\coordinate (yzaux2) at (0.343810,0.919316,0.193099);
							
							\fill[face] (xm) -- (xmzp) -- (xpzp) -- (xp) -- cycle; 
							\fill[face] (xm) -- (xmzm) -- (xpzm) -- (xp) -- cycle; 
							\fill[face] (xm) -- (xmyp) -- (xpyp) -- (xp) -- cycle; 
							\fill[face] (xm) -- (xmym) -- (xpym) -- (xp) -- cycle; 
							
							\draw[edge] (xm) -- (xmzp);
							\draw[edge] (xpzp) -- (xmzp);
							\draw[edge] (xp) -- (xpzp);
							\draw[edge] (xp) -- (xm);
							\draw[edge] (yzaux1) -- (xmzm);
							\draw[edge] (xpzm) -- (xmzm);
							\draw[edge] (xp) -- (xpzm);
							\draw[edge] (xpyp) -- (yzaux2);
							\draw[edge] (xp) -- (xpyp);
							\draw[edge] (xm) -- (xmym);
							\draw[edge] (xpym) -- (xmym);
							\draw[edge] (xp) -- (xpym);
							
							\markpointc{0,0}{30}{2mm}{south west}
				\end{tikzpicture}}}\)
				&
				\(\vcenter{\hbox{\begin{tikzpicture}[localpic]
							
							\coordinate (XmYm) at (-1.157795,-0.807412,-0.087156);
							\coordinate (XmZm) at (-0.969673,0.259480,-0.996195);
							\coordinate (Xm) at (-0.984808,0.173648,0.000000);
							\coordinate (XmZp) at (-0.999942,0.087817,0.996195);
							\coordinate (XmYp) at (-0.811820,1.154708,0.087156);
							\coordinate (YmZm) at (-0.157853,-0.895229,-1.083350);
							\coordinate (Ym) at (-0.172987,-0.981060,-0.087156);
							\coordinate (YmZp) at (-0.188122,-1.066892,0.909039);
							\coordinate (Zm) at (0.015134,0.085832,-0.996195);
							\coordinate (O) at (0.000000,0.000000,0.000000);
							\coordinate (Zp) at (-0.015134,-0.085832,0.996195);
							\coordinate (YpZm) at (0.188122,1.066892,-0.909039);
							\coordinate (Yp) at (0.172987,0.981060,0.087156);
							\coordinate (YpZp) at (0.157853,0.895229,1.083350);
							\coordinate (XpYm) at (0.811820,-1.154708,-0.087156);
							\coordinate (XpZm) at (0.999942,-0.087817,-0.996195);
							\coordinate (Xp) at (0.984808,-0.173648,0.000000);
							\coordinate (XpZp) at (0.969673,-0.259480,0.996195);
							\coordinate (XpYp) at (1.157795,0.807412,0.087156);
							
							\coordinate (aux1) at (0.163447,0.926956,0.715111);
							\coordinate (aux2) at (0.181696,1.030449,-0.486070);
							\coordinate (aux3) at (0.005594,0.031727,-0.368241);
							\coordinate (aux4) at (-0.632643,0.111552,0.000000);
							\coordinate (aux5) at (-0.979213,0.205376,-0.368241);
							\coordinate (aux6) at (-0.617508,0.197384,-0.996195);
							\coordinate (aux7) at (0.525154,0.918964,0.087156);
							\coordinate (aux8) at (-0.888513,0.719763,0.048516);
							
							\fill[face] (O) -- (Yp) -- (YpZp) -- (Zp) -- cycle; 
							\fill[face] (O) -- (Yp) -- (YpZm) -- (Zm) -- cycle; 
							\fill[face] (O) -- (Ym) -- (YmZp) -- (Zp) -- cycle; 
							\fill[face] (O) -- (Ym) -- (YmZm) -- (Zm) -- cycle; 
							\fill[face] (O) -- (Xp) -- (XpZp) -- (Zp) -- cycle; 
							\fill[face] (O) -- (Xp) -- (XpZm) -- (Zm) -- cycle; 
							\fill[face] (O) -- (Xm) -- (XmZp) -- (Zp) -- cycle; 
							\fill[face] (O) -- (Xm) -- (XmZm) -- (Zm) -- cycle; 
							\fill[face] (O) -- (Xp) -- (XpYp) -- (Yp) -- cycle; 
							\fill[face] (O) -- (Xp) -- (XpYm) -- (Ym) -- cycle; 
							\fill[face] (O) -- (Xm) -- (XmYp) -- (Yp) -- cycle; 
							\fill[face] (O) -- (Xm) -- (XmYm) -- (Ym) -- cycle; 
							
							\draw[edge] (aux1) -- (YpZp);
							\draw[edge] (Zp) -- (YpZp);
							\draw[edge] (O) -- (Zp);
							\draw[edge] (Zm) -- (aux3);
							\draw[edge] (O) -- (Ym);
							\draw[edge] (Ym) -- (YmZp);
							\draw[edge] (Zp) -- (YmZp);
							\draw[edge] (YmZm) -- (Ym);
							\draw[edge] (Zm) -- (YmZm);
							\draw[edge] (Xp) -- (O);
							\draw[edge] (Xp) -- (XpZp);
							\draw[edge] (XpZp) -- (Zp);
							\draw[edge] (Xp) -- (XpZm);
							\draw[edge] (Zm) -- (XpZm);
							\draw[edge] (aux4) -- (Xm);
							\draw[edge] (Xm) -- (XmZp);
							\draw[edge] (XmZp) -- (Zp);
							\draw[edge] (aux5) -- (XmZm);
							\draw[edge] (aux6) -- (XmZm);
							\draw[edge] (Xp) -- (XpYp);
							\draw[edge] (XpYp) -- (aux7);
							\draw[edge] (Xp) -- (XpYm);
							\draw[edge] (XpYm) -- (Ym);
							\draw[edge] (Xm) -- (XmYm);
							\draw[edge] (XmYm) -- (Ym);
							
							\markpointc{0,0}{30}{2mm}{south west}
				\end{tikzpicture}}}\)
				\\[1.5mm]
				\(\{x=0\}\) & \(\{xy=0\}\) & \(\{xyz=0\}\)
			\end{tabular}
			
		\end{minipage}
		\par}
	\caption{All local models around points \(p\) on any generically immersed 2-sphere \(S\) in 3-space.}
	\label{fig:local pictures}
\end{figure}

 The complement of \(S\) in 3-space is a disconnected set by a classical argument of Henri \cite{poincare1895}, wherein the contrary one constructs a loop crossing \(S\) exactly once.   Define \emph{a room} to be a component \(R\) of the complement and  \emph{a chamber} to be the closure \(\ol{R}\) of a room.  Each point in any chamber can be locally modeled by one of the models portrayed in \cref{fig:local chamber pictures}. There is exactly one unbounded room since \(S\) is compact. The distinction between closed chambers and open rooms is important for a few reasons, some being that a chamber need not be homeomorphic to a 3-manifold while a room always is, or that a room may be simply connected yet its chamber need not be. Define the set of \emph{mould lines} of a chamber to be the set of points where that chamber fails to be a manifold, that is, the points whose local models in \cref{fig:local chamber pictures} are labeled II or III.  \emph{A wrapping number \(w\left(R\right)\)} of a room \(R\) equals a signed intersection number with \(S\) of a ray that leaves \(R\), meets \(S\) transversely, then runs off to infinity.\footnote{This notion of a wrapping number is well defined because each ray is considered as a relative cycle in the 2-sphere cross the real line. Adapting the argument from \cref{lem:unbounded room} proves that this definition does not depend on a choice of room point.} The two wrapping numbers of a room differ only by a sign, so the absolute value \(\abs{w\left(R\right)}\) is well defined.

\begin{figure}[h]
	{\centering
		\begin{minipage}{\textwidth}
			\centering
			
			
			\begin{tabular}{@{}c@{\hspace{0.3in}}c@{\hspace{0.3in}}c@{\hspace{0.3in}}c@{}}
				\textbf{Room Point} & \textbf{Wall Point} & \textbf{Edge Point I} & \textbf{Edge Point II}
				\\[1.5mm]
				
				\(\vcenter{\hbox{\begin{tikzpicture}[cubepic]
							\cubecoords{}{(0,0,0)}
							
							\fill[chamber] (A) -- (B) -- (E) -- (C) -- cycle; 
							\fill[chamber] (A) -- (C) -- (G) -- (D) -- cycle; 
							\fill[chamber] (C) -- (E) -- (H) -- (G) -- cycle; 
							\fill[chamber] (A) -- (B) -- (F) -- (D) -- cycle; 
							\fill[chamber] (B) -- (E) -- (H) -- (F) -- cycle; 
							\fill[chamber] (D) -- (F) -- (H) -- (G) -- cycle; 
							
							\draw[edge] (A) -- (B) -- (E) -- (H) -- (G) -- (D) -- (A); 
							\draw[edge] (B) -- (F) -- (D);                             
							\draw[edge] (F) -- (H);
							\draw[hidden] (A) -- (C);                                  
							\draw[hidden] (C) -- (E);
							\draw[hidden] (C) -- (G);
							
							\markpointc{0.5,0.5,0.5}{30}{2mm}{south west}
				\end{tikzpicture}}}\)
				&
				\(\vcenter{\hbox{\begin{tikzpicture}[cubepic]
							\cubecoords{}{(0,0,0)}
							
							\fill[chamber] (A) -- (B) -- (E) -- (C) -- cycle; 
							\fill[chamber] (A) -- (C) -- (G) -- (D) -- cycle; 
							\fill[chamber] (C) -- (E) -- (H) -- (G) -- cycle; 
							\fill[chamber] (A) -- (B) -- (F) -- (D) -- cycle; 
							\draw[hidden] (C) -- (E);                         
							\fill[sheet]   (B) -- (E) -- (H) -- (F) -- cycle; 
							\fill[chamber] (D) -- (F) -- (H) -- (G) -- cycle; 
							
							\draw[edge] (A) -- (B) -- (E) -- (H) -- (G) -- (D) -- (A);
							\draw[edge] (B) -- (F) -- (D);
							\draw[edge] (F) -- (H);
							\draw[hidden] (A) -- (C);
							\draw[hidden] (C) -- (G);
							
							\markpointc{1,0.5,0.5}{90}{2mm}{south}
				\end{tikzpicture}}}\)
				&
				\(\vcenter{\hbox{\begin{tikzpicture}[cubepic]
							\cubecoords{}{(0,0,0)}
							
							\fill[sheet]   (A) -- (B) -- (E) -- (C) -- cycle; 
							\fill[chamber] (A) -- (C) -- (G) -- (D) -- cycle; 
							\fill[chamber] (C) -- (E) -- (H) -- (G) -- cycle; 
							\fill[chamber] (A) -- (B) -- (F) -- (D) -- cycle; 
							\draw[hidden] (C) -- (E);
							\fill[sheet]   (B) -- (E) -- (H) -- (F) -- cycle; 
							\fill[chamber] (D) -- (F) -- (H) -- (G) -- cycle; 
							
							\draw[edge] (A) -- (B) -- (E) -- (H) -- (G) -- (D) -- (A);
							\draw[edge] (B) -- (F) -- (D);
							\draw[edge] (F) -- (H);
							\draw[hidden] (A) -- (C);
							\draw[hidden] (C) -- (G);
							
							\markpointc{1,0.5,0}{-30}{2mm}{north west}
				\end{tikzpicture}}}\)
				&
				\(\vcenter{\hbox{\begin{tikzpicture}[cubepic, line cap=round]
							\cubecoords{}{(0,0,0)}    
							\cubecoords{2}{(-1,0,1)}  
							
							\fill[chamber] (A) -- (B) -- (E) -- (C) -- cycle;
							\fill[sheet]   (A) -- (C) -- (G) -- (D) -- cycle;
							\fill[chamber] (C) -- (E) -- (H) -- (G) -- cycle;
							\fill[chamber] (A) -- (B) -- (F) -- (D) -- cycle;
							\fill[chamber] (B) -- (E) -- (H) -- (F) -- cycle;
							\draw[hidden] (A) -- (C);
							\draw[hidden] (C) -- (E);
							\draw[hidden] (C) -- (G);
							\fill[sheet]   (D) -- (F) -- (H) -- (G) -- cycle;
							
							\fill[sheet]   (A2) -- (B2) -- (E2) -- (C2) -- cycle;
							\fill[chamber] (A2) -- (C2) -- (G2) -- (D2) -- cycle;
							\fill[chamber] (C2) -- (E2) -- (H2) -- (G2) -- cycle;
							\fill[chamber] (A2) -- (B2) -- (F2) -- (D2) -- cycle;
							\draw[hidden] (A2) -- (C2);
							\draw[hidden] (C2) -- (E2);
							\draw[hidden] (C2) -- (G2);
							\fill[sheet]   (B2) -- (E2) -- (H2) -- (F2) -- cycle;
							\fill[chamber] (D2) -- (F2) -- (H2) -- (G2) -- cycle;
							
							\draw[edge] (A) -- (B) -- (E) -- (H) -- (G);
							\draw[edge] (D) -- (A);
							\draw[edge] (B) -- (F) -- (D);
							\draw[edge] (F) -- (H);
							
							\draw[edge] (A2) -- (B2);
							\draw[edge] (E2) -- (H2) -- (G2) -- (D2) -- (A2);
							\draw[edge] (B2) -- (F2) -- (D2);
							\draw[edge] (F2) -- (H2);
							
							\draw[glued] (B2) -- (E2);
							
							\markpointc{0,0.5,1}{90}{2mm}{south}
				\end{tikzpicture}}}\)
			\end{tabular}
			
			\vspace{3ex}
			
			
			\begin{tabular}{@{}c@{\hspace{0.3in}}c@{\hspace{0.3in}}c@{}}
				\textbf{Corner Point I} & \textbf{Corner Point II} & \textbf{Corner Point III}
				\\[1.5mm]
				
				\(\vcenter{\hbox{\begin{tikzpicture}[cubepic, line cap=round]
							\cubecoords{}{(0,0,0)}
							
							\draw[hidden] (C) -- (E);
							\draw[hidden] (C) -- (G);
							
							\fill[chamber] (A) -- (B) -- (E) -- (C) -- cycle; 
							\fill[chamber] (A) -- (C) -- (G) -- (D) -- cycle; 
							\fill[sheet]   (C) -- (E) -- (H) -- (G) -- cycle; 
							\fill[chamber] (A) -- (B) -- (F) -- (D) -- cycle; 
							\fill[sheet]   (B) -- (E) -- (H) -- (F) -- cycle; 
							\fill[sheet]   (D) -- (F) -- (H) -- (G) -- cycle; 
							
							\draw[edge] (D) -- (A) -- (B) -- (E) -- (H) -- (G) -- (D);
							\draw[edge] (B) -- (F) -- (D);
							\draw[edge] (F) -- (H);
							\draw[hidden] (A) -- (C);
							
							\markpointc{1,1,1}{30}{2mm}{south west}
				\end{tikzpicture}}}\)
				&
				\(\vcenter{\hbox{\begin{tikzpicture}[cubepic, line cap=round]
							\cubecoords{2}{(-1,0,1)}  
							\cubecoords{3}{(-1,-1,0)} 
							
							\draw[hidden] (C2) -- (G2);
							\fill[sheet]   (A2) -- (B2) -- (E2) -- (C2) -- cycle;
							\fill[chamber] (A2) -- (C2) -- (G2) -- (D2) -- cycle;
							\fill[chamber] (C2) -- (E2) -- (H2) -- (G2) -- cycle;
							\fill[sheet]   (A2) -- (B2) -- (F2) -- (D2) -- cycle;
							\fill[sheet]   (B2) -- (E2) -- (H2) -- (F2) -- cycle;
							\fill[chamber] (D2) -- (F2) -- (H2) -- (G2) -- cycle;
							
							\draw[hidden] (C3) -- (E3);
							\draw[hidden] (C3) -- (G3);
							\fill[chamber] (A3) -- (B3) -- (E3) -- (C3) -- cycle;
							\fill[chamber] (A3) -- (C3) -- (G3) -- (D3) -- cycle;
							\fill[sheet]   (C3) -- (E3) -- (H3) -- (G3) -- cycle;
							\fill[chamber] (A3) -- (B3) -- (F3) -- (D3) -- cycle;
							\fill[sheet]   (B3) -- (E3) -- (H3) -- (F3) -- cycle;
							\fill[sheet]   (D3) -- (F3) -- (H3) -- (G3) -- cycle;
							
							\draw[edge] (B2) -- (E2) -- (H2) -- (G2) -- (D2) -- (A2);
							\draw[edge] (B2) -- (F2) -- (D2);
							\draw[edge] (F2) -- (H2);
							\draw[occluded] (A2) -- (C2);
							\draw[occluded] (C2) -- (E2);
							
							\draw[edge] (E3) -- (H3);
							\draw[edge] (D3) -- (A3) -- (B3) -- (E3);
							\draw[edge] (G3) -- (D3);
							\draw[edge] (B3) -- (F3) -- (D3);
							\draw[edge] (F3) -- (H3);
							\draw[hidden] (A3) -- (C3);
							
							\draw[glued] (A2) -- (B2);
							
							\markpointc{0,0,1}{-30}{2mm}{north west}
				\end{tikzpicture}}}\)
				&
				\(\vcenter{\hbox{\begin{tikzpicture}[cubepic, line cap=round]
							\cubecoords{}{(0,0,0)}    
							\cubecoords{2}{(-1,0,1)}  
							\cubecoords{3}{(-1,-1,0)} 
							
							\draw[hidden] (C) -- (E);
							\fill[chamber] (A) -- (B) -- (E) -- (C) -- cycle;
							\fill[sheet]   (A) -- (C) -- (G) -- (D) -- cycle;
							\fill[chamber] (C) -- (E) -- (H) -- (G) -- cycle;
							\fill[sheet]   (A) -- (B) -- (F) -- (D) -- cycle;
							\fill[chamber] (B) -- (E) -- (H) -- (F) -- cycle;
							\fill[sheet]   (D) -- (F) -- (H) -- (G) -- cycle;
							
							\draw[hidden] (C2) -- (G2);
							\fill[sheet]   (A2) -- (B2) -- (E2) -- (C2) -- cycle;
							\fill[chamber] (A2) -- (C2) -- (G2) -- (D2) -- cycle;
							\fill[chamber] (C2) -- (E2) -- (H2) -- (G2) -- cycle;
							\fill[sheet]   (A2) -- (B2) -- (F2) -- (D2) -- cycle;
							\fill[sheet]   (B2) -- (E2) -- (H2) -- (F2) -- cycle;
							\fill[chamber] (D2) -- (F2) -- (H2) -- (G2) -- cycle;
							
							\draw[hidden] (C3) -- (E3);
							\draw[hidden] (C3) -- (G3);
							\fill[chamber] (A3) -- (B3) -- (E3) -- (C3) -- cycle;
							\fill[chamber] (A3) -- (C3) -- (G3) -- (D3) -- cycle;
							\fill[sheet]   (C3) -- (E3) -- (H3) -- (G3) -- cycle;
							\fill[chamber] (A3) -- (B3) -- (F3) -- (D3) -- cycle;
							\fill[sheet]   (B3) -- (E3) -- (H3) -- (F3) -- cycle;
							\fill[sheet]   (D3) -- (F3) -- (H3) -- (G3) -- cycle;
							
							\draw[edge] (A) -- (B) -- (E) -- (H) -- (G);
							\draw[edge] (B) -- (F) -- (D);
							\draw[edge] (F) -- (H);
							\draw[occluded] (A) -- (C);
							\draw[occluded] (C) -- (G);
							
							\draw[edge] (E2) -- (H2) -- (G2) -- (D2) -- (A2);
							\draw[edge] (B2) -- (F2) -- (D2);
							\draw[edge] (F2) -- (H2);
							\draw[occluded] (A2) -- (C2);
							\draw[occluded] (C2) -- (E2);
							
							\draw[edge] (D3) -- (A3) -- (B3) -- (E3);
							\draw[edge] (G3) -- (D3);
							\draw[edge] (B3) -- (F3) -- (D3);
							\draw[edge] (F3) -- (H3);
							\draw[hidden] (A3) -- (C3);
							
							\draw[glued] (B2) -- (E2); 
							\draw[glued] (E3) -- (H3); 
							\draw[glued] (A2) -- (B2); 
							
							\markpointc{0,0,1}{-45}{2mm}{north west}
				\end{tikzpicture}}}\)
			\end{tabular}
			
		\end{minipage}
		\par}
	\caption{All local models around points \(p\) in any chamber \(\ol{R}\) of \(S\) as well as around points \(p\) in the closure of the inverse image \(\ol{F^{-1}\left(R\right)}\) of any room  \(R\) under any filling \(F\) of \(S\). The darker blue portrays the walls, while the lighter blue portrays the room.}
	\label{fig:local chamber pictures}
\end{figure}
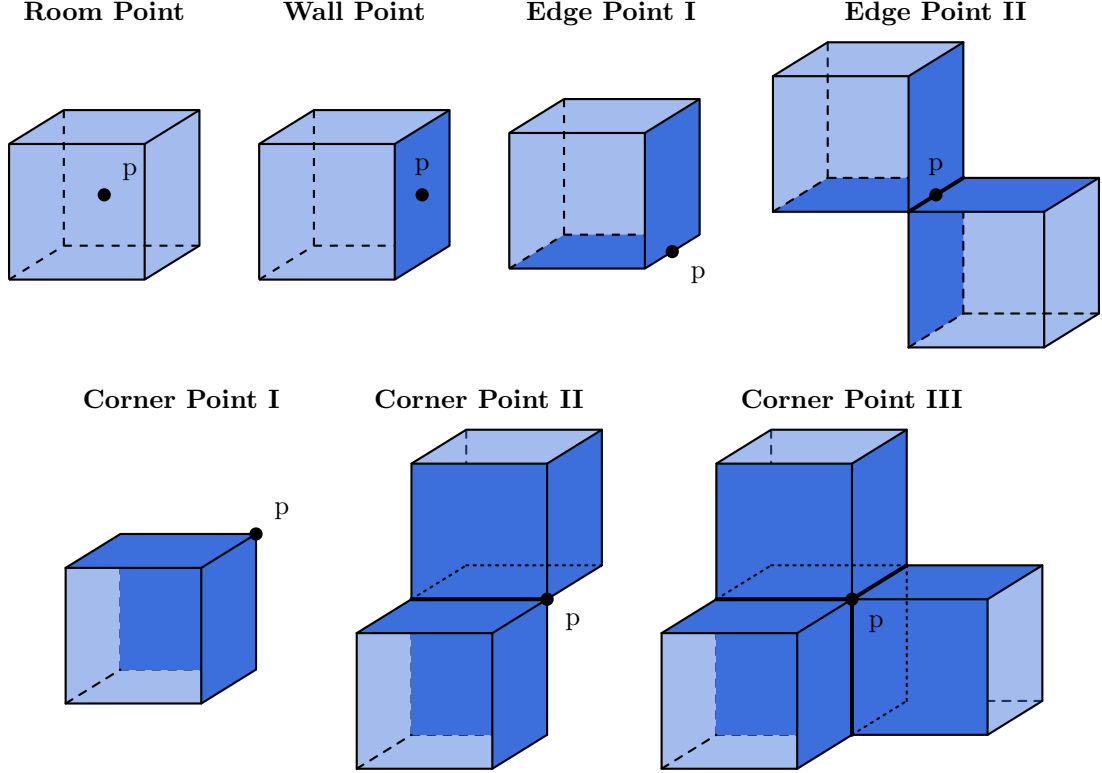


For any generically immersed 2-sphere \(S\) in 3-space, define \emph{the hardware store of \(S\)} to be the set \(H_S\) that is the union, over all rooms \(R\) of \(S\), of all isomorphism classes of degree \(\abs{w\left(R\right)}\) covering spaces over \(R\). Since each room has finitely generated fundamental group, standard covering space theory implies that \(H_S\) is finite.  Let \(\mathfrak{R}_S^{\neq 0}\) denote the set of all rooms of \(S\) with nonzero wrapping number. By definition, \(H_S\) comes equipped with a surjection \(\pi_S\) onto \(\mathfrak{R}_S^{\neq 0}\), which we call \emph{the hardware store bundle of} \(S\). Define \emph{an aisle of} \(S\) to be a fiber of \(\pi_S\). The hardware stores of equivalent immersed spheres are in canonical bijection.

\emph{A kit of \(S\)} is a space \(K\)  representing the isomorphism class of a global section of \(\pi_S\), so \(K\) comes equipped with a covering map over the complement of \(S\) in 3-space. We call a connected component of \(K\) \emph{a piece}, and the union of pieces coming from a chosen covering space at an aisle \emph{a part}. \emph{A selection from the hardware store of \(S\)} is an object consisting of a kit \(K\) and a global section \(\left[K'\right]\) of \(\pi_S\) such that \(K\) represents \(\left[K'\right]\). \emph{An equivalence of selections} is a morphism given by any isomorphism of covering spaces between the underlying kits.  Connected components of this category are canonically in bijection with the set of global sections of \(\pi_S\), and hence form a finite set called \emph{the set of inequivalent selections of kit from the hardware store of \(S\)}. 

\emph{A building of \(S\)} is a component of the closure of a union over all rooms in \(\mathfrak{R}_S^{\neq 0}\) whose wrapping numbers have a given sign.  The way wrapping numbers change by exactly one across each wall implies that the intersection of any positive building with any negative building is empty or one-dimensional, and that all buildings are compact and smooth 3-manifolds with corners. We call a generically immersed 2-sphere \(S\) in 3-space \emph{positive} if \(S\) has exactly one building, which we call \emph{the building of \(S\)} and denote by \(B_S\). Each point in \(B_S\) can be locally modeled by one of the models portrayed in \cref{fig:building}.  If \(S\) is positive, then there exist fixed orientations so that the wrapping numbers of \(S\) are nonnegative.

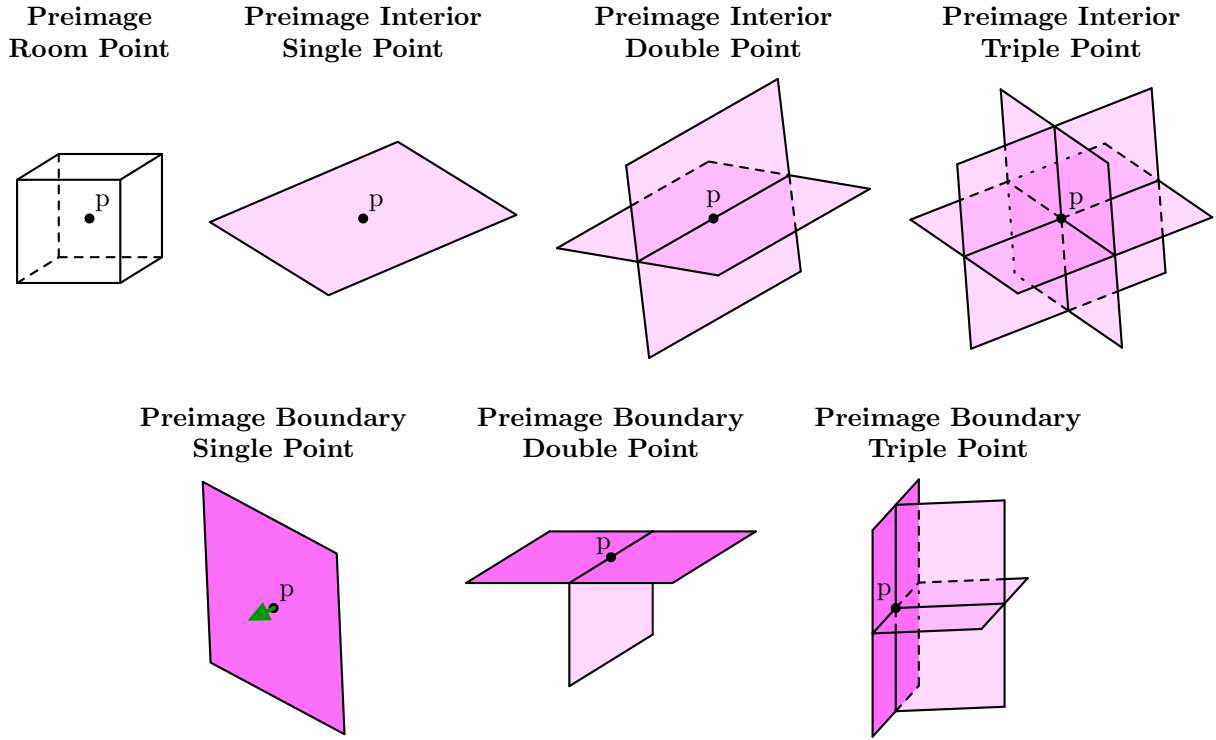
\begin{figure}[h!]
	{\centering
		\begin{minipage}{\textwidth}
			\centering
			
			\begin{tabular}{@{}c@{\hspace{0.2in}}c@{\hspace{0.2in}}c@{\hspace{0.2in}}c@{}}
				\shortstack{\textbf{Preimage}\\\textbf{Room Point}} & \shortstack{\textbf{Preimage Interior}\\\textbf{Single Point}} & \shortstack{\textbf{Preimage Interior}\\\textbf{Double Point}} & \shortstack{\textbf{Preimage Interior}\\\textbf{Triple Point}}
				\\[1.5mm]
				
				\(\vcenter{\hbox{\begin{tikzpicture}[preimagepic]
							\cubecoords{}{(0,0,0)}
							
							\draw[edge] (A) -- (B) -- (E) -- (H) -- (G) -- (D) -- (A);
							\draw[edge] (B) -- (F) -- (D);
							\draw[edge] (F) -- (H);
							\draw[hidden] (A) -- (C);
							\draw[hidden] (C) -- (E);
							\draw[hidden] (C) -- (G);
							
							\markpoint{0.5,0.5,0.5}{30}{-2mm}
				\end{tikzpicture}}}\)
				&
				\(\vcenter{\hbox{\begin{tikzpicture}[preimagepic]
							
							\coordinate (xmym) at (-1.396875,-0.219848,0.020206);
							\coordinate (xpym) at (0.204703,-1.338233,-0.408935);
							\coordinate (xpyp) at (1.396875,0.219848,-0.020206);
							\coordinate (xmyp) at (-0.204703,1.338233,0.408935);
							
							\fill[chamberpink] (xmym) -- (xpym) -- (xpyp) -- (xmyp) -- cycle;
							\draw[edge] (xmym) -- (xpym) -- (xpyp) -- (xmyp) -- cycle;
							
							\markpoint{0,0}{30}{-2mm}
				\end{tikzpicture}}}\)
				&
				\(\vcenter{\hbox{\begin{tikzpicture}[preimagepic]
							
							\coordinate (xmym) at (-1.296752,-0.543584,-0.151492);
							\coordinate (xmyp) at (-0.562312,1.276587,0.232661);
							\coordinate (ymzm) at (-0.333759,-0.716641,-1.172617);
							\coordinate (ym) at (-0.367220,-0.910086,-0.192077);
							\coordinate (ymzp) at (-0.400681,-1.103530,0.788464);
							\coordinate (ypzm) at (0.400681,1.103530,-0.788464);
							\coordinate (yp) at (0.367220,0.910086,0.192077);
							\coordinate (ypzp) at (0.333759,0.716641,1.172617);
							\coordinate (xpym) at (0.562312,-1.276587,-0.232661);
							\coordinate (xpyp) at (1.296752,0.543584,0.151492);
							
							\coordinate (xzaux1) at (0.386658,1.022457,-0.377516);
							\coordinate (xzaux2) at (-0.938554,0.344142,0.035865);
							
							\fill[chamberpink] (ym) -- (ymzp) -- (ypzp) -- (yp) -- cycle; 
							\fill[chamberpink] (ym) -- (ymzm) -- (ypzm) -- (yp) -- cycle; 
							\fill[chamberpink] (ym) -- (xpym) -- (xpyp) -- (yp) -- cycle; 
							\fill[chamberpink] (ym) -- (xmym) -- (xmyp) -- (yp) -- cycle; 
							
							\draw[edge] (ym) -- (ymzp);
							\draw[edge] (ymzp) -- (ypzp);
							\draw[edge] (yp) -- (ypzp);
							\draw[edge] (ym) -- (yp);
							\draw[edge] (ymzm) -- (ym);
							\draw[edge] (ymzm) -- (ypzm);
							\draw[edge] (ypzm) -- (xzaux1);
							\draw[edge] (xpym) -- (ym);
							\draw[edge] (xpym) -- (xpyp);
							\draw[edge] (xpyp) -- (yp);
							\draw[edge] (xmym) -- (ym);
							\draw[edge] (xmym) -- (xzaux2);
							
							\draw[hidden] (xzaux1) -- (yp);
							\draw[hidden] (xzaux2) -- (xmyp);
							\draw[hidden] (xmyp) -- (yp);
							
							\markpoint{0,0}{90}{-1mm}
				\end{tikzpicture}}}\)
				&
				\(\vcenter{\hbox{\begin{tikzpicture}[preimagepic]
							
							\coordinate (XmYm) at (-1.407398,-0.136747,0.023026);
							\coordinate (XmZm) at (-0.808417,0.880104,-0.756227);
							\coordinate (Xm) at (-0.766320,0.608761,0.205335);
							\coordinate (XmZp) at (-0.724224,0.337419,1.166897);
							\coordinate (XmYp) at (-0.125243,1.354270,0.387644);
							\coordinate (YmZm) at (-0.683174,-0.474165,-1.143871);
							\coordinate (Ym) at (-0.641078,-0.745508,-0.182309);
							\coordinate (YmZp) at (-0.598982,-1.016851,0.779253);
							\coordinate (Zm) at (-0.042096,0.271343,-0.961562);
							\coordinate (O) at (0.000000,0.000000,0.000000);
							\coordinate (Zp) at (0.042096,-0.271343,0.961562);
							\coordinate (YpZm) at (0.598982,1.016851,-0.779253);
							\coordinate (Yp) at (0.641078,0.745508,0.182309);
							\coordinate (YpZp) at (0.683174,0.474165,1.143871);
							\coordinate (XpYm) at (0.125243,-1.354270,-0.387644);
							\coordinate (XpZm) at (0.724224,-0.337419,-1.166897);
							\coordinate (Xp) at (0.766320,-0.608761,-0.205335);
							\coordinate (XpZp) at (0.808417,-0.880104,0.756227);
							\coordinate (XpYp) at (1.407398,0.136747,-0.023026);
							
							\coordinate (aux1) at (0.326147,0.379275,0.092749);
							\coordinate (aux2) at (0.615403,0.911003,-0.404158);
							\coordinate (aux3) at (0.284050,0.650618,-0.868813);
							\coordinate (aux4) at (-0.025675,0.165495,-0.586467);
							\coordinate (aux5) at (-0.749899,0.502913,0.580431);
							\coordinate (aux6) at (-0.532222,0.660696,-0.830233);
							\coordinate (aux7) at (-0.231576,1.230615,0.357405);
							\coordinate (aux8) at (-1.081253,0.242527,0.115775);
							
							\fill[chamberpink] (O) -- (Yp) -- (YpZp) -- (Zp) -- cycle; 
							\fill[chamberpink] (O) -- (Yp) -- (YpZm) -- (Zm) -- cycle; 
							\fill[chamberpink] (O) -- (Ym) -- (YmZp) -- (Zp) -- cycle; 
							\fill[chamberpink] (O) -- (Ym) -- (YmZm) -- (Zm) -- cycle; 
							\fill[chamberpink] (O) -- (Xp) -- (XpZp) -- (Zp) -- cycle; 
							\fill[chamberpink] (O) -- (Xp) -- (XpZm) -- (Zm) -- cycle; 
							\fill[chamberpink] (O) -- (Xm) -- (XmZp) -- (Zp) -- cycle; 
							\fill[chamberpink] (O) -- (Xm) -- (XmZm) -- (Zm) -- cycle; 
							\fill[chamberpink] (O) -- (Xp) -- (XpYp) -- (Yp) -- cycle; 
							\fill[chamberpink] (O) -- (Xp) -- (XpYm) -- (Ym) -- cycle; 
							\fill[chamberpink] (O) -- (Xm) -- (XmYp) -- (Yp) -- cycle; 
							\fill[chamberpink] (O) -- (Xm) -- (XmYm) -- (Ym) -- cycle; 
							
							\draw[edge] (aux1) -- (Yp);
							\draw[edge] (Yp) -- (YpZp);
							\draw[edge] (Zp) -- (YpZp);
							\draw[edge] (O) -- (Zp);
							\draw[edge] (YpZm) -- (aux2);
							\draw[edge] (aux3) -- (YpZm);
							\draw[edge] (Zm) -- (aux4);
							\draw[edge] (O) -- (Ym);
							\draw[edge] (Ym) -- (YmZp);
							\draw[edge] (Zp) -- (YmZp);
							\draw[edge] (YmZm) -- (Ym);
							\draw[edge] (Zm) -- (YmZm);
							\draw[edge] (Xp) -- (O);
							\draw[edge] (Xp) -- (XpZp);
							\draw[edge] (XpZp) -- (Zp);
							\draw[edge] (Xp) -- (XpZm);
							\draw[edge] (Zm) -- (XpZm);
							\draw[edge] (aux5) -- (XmZp);
							\draw[edge] (XmZp) -- (Zp);
							\draw[edge] (Xp) -- (XpYp);
							\draw[edge] (XpYp) -- (Yp);
							\draw[edge] (Xp) -- (XpYm);
							\draw[edge] (XpYm) -- (Ym);
							\draw[edge] (aux8) -- (XmYm);
							\draw[edge] (XmYm) -- (Ym);
							
							\draw[hidden] (O) -- (aux1);
							\draw[hidden] (aux2) -- (Yp);
							\draw[hidden] (Zm) -- (aux3);
							\draw[hidden] (aux4) -- (O);
							\draw[hidden] (O) -- (Xm);
							\draw[hidden] (Xm) -- (aux5);
							\draw[hidden] (Zm) -- (aux6);
							\draw[hidden] (aux7) -- (XmYp);
							\draw[hidden] (XmYp) -- (Yp);
							\draw[hidden] (Xm) -- (aux8);
							
							\draw[loosedotted] (Xm) -- (XmZm);
							\draw[loosedotted] (aux6) -- (XmZm);
							\draw[loosedotted] (Xm) -- (aux7);
							
							\markpoint{0,0}{10}{-1.5mm}
				\end{tikzpicture}}}\)
			\end{tabular}
			
			\vspace{4ex}
			
			\begin{tabular}{@{}c@{\hspace{0.3in}}c@{\hspace{0.3in}}c@{}}
				\shortstack{\textbf{Preimage Boundary}\\\textbf{Single Point}} & \shortstack{\textbf{Preimage Boundary}\\\textbf{Double Point}} & \shortstack{\textbf{Preimage Boundary}\\\textbf{Triple Point}}
				\\[1.5mm]
				
				\(\vcenter{\hbox{\begin{tikzpicture}[preimagepic]
							
							\coordinate (xmzm) at (-0.927863,0.783947,-0.724221);
							\coordinate (xpzm) at (0.761772,-0.185672,-1.176958);
							\coordinate (xpzp) at (0.927863,-0.783947,0.724221);
							\coordinate (xmzp) at (-0.761772,0.185672,1.176958);
							
							\fill[sheetpink] (xmzm) -- (xpzm) -- (xpzp) -- (xmzp) -- cycle;
							\draw[edge] (xmzm) -- (xpzm) -- (xpzp) -- (xmzp) -- cycle;
							
							\markpoint{0,0}{30}{-2mm}
							\draw[normalvec] (0,0,0) -- (-0.158571,-0.246562,-0.063737); 
				\end{tikzpicture}}}\)
				&
				\(\vcenter{\hbox{\begin{tikzpicture}[preimagepic]
							
							\coordinate (xmym) at (-1,-1,.5);
							\coordinate (xmyp) at (-1,1,.5);
							\coordinate (ymzm) at (0,-1,-.5);
							\coordinate (ym) at (0,-1,.5);
							\coordinate (ypzm) at (0,1,-.5);
							\coordinate (yp) at (0,1,.5);
							\coordinate (xpym) at (1,-1,.5);
							\coordinate (xpyp) at (1,1,.5);
							
							\coordinate (aux1) at (0.807692,1.000000,0.500000);
							
							\draw[edge] (ypzm) -- (yp); 
							\fill[chamberpink] (ym) -- (ymzm) -- (ypzm) -- (yp) -- cycle; 
							\fill[sheetpink]   (ym) -- (xpym) -- (xpyp) -- (yp) -- cycle; 
							\fill[sheetpink]   (ym) -- (xmym) -- (xmyp) -- (yp) -- cycle; 
							
							\draw[edge] (ym) -- (ymzm);
							\draw[edge] (ymzm) -- (ypzm);
							\draw[edge] (ym) -- (xpym);
							\draw[edge] (xpym) -- (xpyp);
							\draw[edge] (xpyp) -- (aux1);
							\draw[edge] (ym) -- (xmym);
							\draw[edge] (xmym) -- (xmyp);
							\draw[edge] (xmyp) -- (yp);
							\draw[edge] (ym) -- (yp);
							\draw[edge] (aux1) -- (yp);
							
							\markpoint{0,0,.5}{145}{-3mm}
				\end{tikzpicture}}}\)
				&
				\(\vcenter{\hbox{\begin{tikzpicture}[preimagepic]
							
							\coordinate (YmZm) at (0.173648,-0.984808,-1.000000);
							\coordinate (Ym) at (0.173648,-0.984808,0.000000);
							\coordinate (YmZp) at (0.173648,-0.984808,1.000000);
							\coordinate (Zm) at (0.000000,0.000000,-1.000000);
							\coordinate (O) at (0.000000,0.000000,0.000000);
							\coordinate (Zp) at (0.000000,0.000000,1.000000);
							\coordinate (YpZm) at (-0.173648,0.984808,-1.000000);
							\coordinate (Yp) at (-0.173648,0.984808,0.000000);
							\coordinate (YpZp) at (-0.173648,0.984808,1.000000);
							\coordinate (XpYm) at (1.158456,-0.811160,0.000000);
							\coordinate (XpZm) at (0.984808,0.173648,-1.000000);
							\coordinate (Xp) at (0.984808,0.173648,0.000000);
							\coordinate (XpZp) at (0.984808,0.173648,1.000000);
							\coordinate (XpYp) at (0.811160,1.158456,0.000000);
							
							\coordinate (aux1) at (-0.173648,0.984808,0.763020);
							\coordinate (aux2) at (-0.173648,0.984808,-0.473963);
							\coordinate (aux3) at (0.000000,0.000000,-0.236982);
							\coordinate (aux4) at (0.601993,1.121574,0.000000);
							
							\fill[sheetpink]   (O) -- (Yp) -- (YpZp) -- (Zp) -- cycle; 
							\fill[sheetpink]   (O) -- (Yp) -- (YpZm) -- (Zm) -- cycle; 
							\fill[sheetpink]   (O) -- (Ym) -- (YmZp) -- (Zp) -- cycle; 
							\fill[sheetpink]   (O) -- (Ym) -- (YmZm) -- (Zm) -- cycle; 
							\fill[chamberpink] (O) -- (Xp) -- (XpZp) -- (Zp) -- cycle; 
							\fill[chamberpink] (O) -- (Xp) -- (XpZm) -- (Zm) -- cycle; 
							\fill[chamberpink] (O) -- (Xp) -- (XpYp) -- (Yp) -- cycle; 
							\fill[chamberpink] (O) -- (Xp) -- (XpYm) -- (Ym) -- cycle; 
							
							\draw[edge] (aux1) -- (YpZp);
							\draw[edge] (Zp) -- (YpZp);
							\draw[edge] (O) -- (Zp);
							\draw[edge] (Zm) -- (aux3);
							\draw[edge] (O) -- (Ym);
							\draw[edge] (Ym) -- (YmZp);
							\draw[edge] (Zp) -- (YmZp);
							\draw[edge] (YmZm) -- (Ym);
							\draw[edge] (Zm) -- (YmZm);
							\draw[edge] (Xp) -- (O);
							\draw[edge] (Xp) -- (XpZp);
							\draw[edge] (XpZp) -- (Zp);
							\draw[edge] (Xp) -- (XpZm);
							\draw[edge] (Zm) -- (XpZm);
							\draw[edge] (Xp) -- (XpYp);
							\draw[edge] (XpYp) -- (aux4);
							\draw[edge] (Xp) -- (XpYm);
							\draw[edge] (XpYm) -- (Ym);
							
							\draw[hidden] (O) -- (Yp);
							\draw[hidden] (Yp) -- (aux1);
							\draw[hidden] (YpZm) -- (aux2);
							\draw[hidden] (Zm) -- (YpZm);
							\draw[hidden] (aux3) -- (O);
							\draw[hidden] (aux4) -- (Yp);
							
							\draw[loosedotted] (aux2) -- (Yp);
							
							\markpoint{0,0}{165}{-2mm}
				\end{tikzpicture}}}\)
			\end{tabular}
			
		\end{minipage}
		\par}
	\caption{All local models around points \(p\) in the domain \(M\) of any filling \(F\) of a generically immersed 2-sphere \(S\) in 3-space. The darker pink portrays the boundary of \(M\), while the lighter pink portrays its complement inside the inverse image \( F^{-1}\left(S\right)\) of \(S\).}
	\label{fig:local pictures in preimage}
	
\end{figure}

\begin{figure}[h!]
	{\centering
		\begin{minipage}{\textwidth}
			\centering
			
			
			{\footnotesize
				\begin{tabular}{@{}c@{\hspace{0.15in}}c@{\hspace{0.15in}}c@{\hspace{0.15in}}c@{}}
					\shortstack{\textbf{Building}\\\textbf{Room Point}} & \shortstack{\textbf{Building Interior}\\\textbf{Single Point}} & \shortstack{\textbf{Building Interior}\\\textbf{Double Point}} & \shortstack{\textbf{Building Interior}\\\textbf{Triple Point}}
					\\[1.5mm]
					
					\(\vcenter{\hbox{\begin{tikzpicture}[buildingpic1]
								\cubecoords{}{(0,0,0)}
								
								\fill[chamberred] (A) -- (B) -- (E) -- (C) -- cycle; 
								\fill[chamberred] (A) -- (C) -- (G) -- (D) -- cycle; 
								\fill[chamberred] (C) -- (E) -- (H) -- (G) -- cycle; 
								\fill[chamberred] (A) -- (B) -- (F) -- (D) -- cycle; 
								\fill[chamberred] (B) -- (E) -- (H) -- (F) -- cycle; 
								\fill[chamberred] (D) -- (F) -- (H) -- (G) -- cycle; 
								
								\draw[edge] (A) -- (B) -- (E) -- (H) -- (G) -- (D) -- (A);
								\draw[edge] (B) -- (F) -- (D);
								\draw[edge] (F) -- (H);
								\draw[hidden] (A) -- (C);
								\draw[hidden] (C) -- (E);
								\draw[hidden] (C) -- (G);
								
								\markpointc{0.5,0.5,0.5}{30}{2mm}{south west}
					\end{tikzpicture}}}\)
					&
					\(\vcenter{\hbox{\begin{tikzpicture}[buildingpic1]
								\cubecoords{A}{(0,0,0)}   
								\cubecoords{D}{(-1,0,0)}  
								\cubefaces{A}{chamberred}
								\cubefaces{D}{chamberred}
								
								\fill[darkredfill] (AA) -- (CA) -- (GA) -- (DA) -- cycle;
								
								\draw[edge] (0,0,0) -- (1,0,0); 
								\draw[edge] (0,0,0) -- (0,1,0); 
								\draw[edge] (0,0,0) -- (0,0,0.25) -- (0,0,1); 
								\draw[edge] (1,0,0) -- (1,1,0); 
								\draw[edge] (1,0,0) -- (1,0,0.25) -- (1,0,1); 
								\draw[hidden] (0,1,0) -- (0.596154,1,0) -- (1,1,0); 
								\draw[edge] (0,1,0) -- (0,1,0.75) -- (0,1,1); 
								\draw[edge] (0,0,1) -- (0.403846,0,1) -- (1,0,1); 
								\draw[edge] (0,0,1) -- (0,1,1); 
								\draw[edge] (1,1,0) -- (1,1,1); 
								\draw[edge] (1,0,1) -- (1,1,1); 
								\draw[edge] (0,1,1) -- (1,1,1); 
								\draw[edge] (-1,0,0) -- (0,0,0); 
								\draw[hidden] (-1,0,0) -- (-1,1,0); 
								\draw[edge] (-1,0,0) -- (-1,0,1); 
								\draw[hidden] (-1,1,0) -- (-0.403846,1,0); 
								\draw[hidden] (-1,1,0) -- (-1,1,0.75) -- (-1,1,1); 
								\draw[edge] (-1,0,1) -- (-0.596154,0,1) -- (0,0,1); 
								\draw[edge] (-1,0,1) -- (-1,1,1); 
								\draw[edge] (-1,1,1) -- (0,1,1); 
								\markpointc{0,0.5,0.5}{90}{2mm}{south}
					\end{tikzpicture}}}\)
					&
					\(\vcenter{\hbox{\begin{tikzpicture}[buildingpic1]
								\cubecoords{A}{(0,0,0)}    
								\cubecoords{B}{(0,0,-1)}   
								\cubecoords{D}{(-1,0,0)}   
								\cubecoords{F}{(-1,0,-1)}  
								\cubefaces{A}{chamberred}
								\cubefaces{B}{chamberred}
								\cubefaces{D}{chamberred}
								\cubefaces{F}{chamberred}
								
								\fill[darkredfill] (AA) -- (BA) -- (EA) -- (CA) -- cycle; 
								\fill[darkredfill] (AA) -- (CA) -- (GA) -- (DA) -- cycle; 
								\fill[darkredfill] (AB) -- (CB) -- (GB) -- (DB) -- cycle; 
								\fill[darkredfill] (AD) -- (BD) -- (ED) -- (CD) -- cycle; 
								
								\draw[edge] (0,0,0) -- (0.403846,0,0) -- (1,0,0); 
								\draw[edge] (0,0,0) -- (0,1,0); 
								\draw[edge] (0,0,0) -- (0,0,0.25) -- (0,0,1); 
								\draw[edge] (1,0,0) -- (1,1,0); 
								\draw[edge] (1,0,0) -- (1,0,0.25) -- (1,0,1); 
								\draw[edge] (0,1,0) -- (0.596154,1,0) -- (1,1,0); 
								\draw[edge] (0,1,0) -- (0,1,0.75) -- (0,1,1); 
								\draw[edge] (0,0,1) -- (0.403846,0,1) -- (1,0,1); 
								\draw[edge] (0,0,1) -- (0,1,1); 
								\draw[edge] (1,1,0) -- (1,1,1); 
								\draw[edge] (1,0,1) -- (1,1,1); 
								\draw[edge] (0,1,1) -- (1,1,1); 
								\draw[edge] (0,0,-1) -- (1,0,-1); 
								\draw[edge] (0,0,-1) -- (0,1,-1); 
								\draw[edge] (0,0,-1) -- (0,0,-0.75) -- (0,0,-1.11022e-16); 
								\draw[edge] (1,0,-1) -- (1,1,-1); 
								\draw[edge] (1,0,-1) -- (1,0,-0.75) -- (1,0,0); 
								\draw[hidden] (0,1,-1) -- (0.596154,1,-1) -- (1,1,-1); 
								\draw[edge] (0,1,-1) -- (0,1,-0.25); 
								\draw[occluded] (0,1,-0.25) -- (0,1,0); 
								\draw[edge] (1,1,-1) -- (1,1,0); 
								\draw[edge] (-1,0,0) -- (-0.596154,0,0) -- (0,0,0); 
								\draw[edge] (-1,0,0) -- (-1,1,0); 
								\draw[edge] (-1,0,0) -- (-1,0,1); 
								\draw[edge] (-1,1,0) -- (-0.403846,1,0); 
								\draw[occluded] (-0.403846,1,0) -- (0,1,0); 
								\draw[hidden] (-1,1,0) -- (-1,1,0.75) -- (-1,1,1); 
								\draw[edge] (-1,0,1) -- (-0.596154,0,1) -- (0,0,1); 
								\draw[edge] (-1,0,1) -- (-1,1,1); 
								\draw[edge] (-1,1,1) -- (0,1,1); 
								\draw[edge] (-1,0,-1) -- (0,0,-1); 
								\draw[hidden] (-1,0,-1) -- (-1,1,-1); 
								\draw[edge] (-1,0,-1) -- (-1,0,-1.11022e-16); 
								\draw[hidden] (-1,1,-1) -- (-0.403846,1,-1); 
								\draw[hidden] (-1,1,-1) -- (-1,1,-0.25); 
								\markpointc{0,0.5,0}{90}{4mm}{south}
					\end{tikzpicture}}}\)
					&
					\(\vcenter{\hbox{\begin{tikzpicture}[buildingpic1]
								\cubecoords{A}{(0,0,0)}     
								\cubecoords{B}{(0,0,-1)}    
								\cubecoords{C}{(0,-1,0)}    
								\cubecoords{D}{(-1,0,0)}    
								\cubecoords{E}{(0,-1,-1)}   
								\cubecoords{F}{(-1,0,-1)}   
								\cubecoords{G}{(-1,-1,0)}   
								\cubecoords{H}{(-1,-1,-1)}  
								\cubefaces{A}{chamberred}
								\cubefaces{B}{chamberred}
								\cubefaces{C}{chamberred}
								\cubefaces{D}{chamberred}
								\cubefaces{E}{chamberred}
								\cubefaces{F}{chamberred}
								\cubefaces{G}{chamberred}
								\cubefaces{H}{chamberred}
								
								\fill[darkredfill] (AA) -- (BA) -- (EA) -- (CA) -- cycle; 
								\fill[darkredfill] (AA) -- (BA) -- (FA) -- (DA) -- cycle; 
								\fill[darkredfill] (AA) -- (CA) -- (GA) -- (DA) -- cycle; 
								\fill[darkredfill] (AB) -- (BB) -- (FB) -- (DB) -- cycle; 
								\fill[darkredfill] (AB) -- (CB) -- (GB) -- (DB) -- cycle; 
								\fill[darkredfill] (AC) -- (BC) -- (EC) -- (CC) -- cycle; 
								\fill[darkredfill] (AC) -- (CC) -- (GC) -- (DC) -- cycle; 
								\fill[darkredfill] (AD) -- (BD) -- (ED) -- (CD) -- cycle; 
								\fill[darkredfill] (AD) -- (BD) -- (FD) -- (DD) -- cycle; 
								\fill[darkredfill] (AE) -- (CE) -- (GE) -- (DE) -- cycle; 
								\fill[darkredfill] (AF) -- (BF) -- (FF) -- (DF) -- cycle; 
								\fill[darkredfill] (AG) -- (BG) -- (EG) -- (CG) -- cycle; 
								
								\draw[edge] (0,0,0) -- (0.403846,0,0) -- (0.596154,0,0) -- (1,0,0); 
								\draw[occluded] (0,0,0) -- (0,1,0); 
								\draw[edge] (0,0,0) -- (0,0,0.25) -- (0,0,0.75) -- (0,0,1); 
								\draw[edge] (1,0,0) -- (1,1,0); 
								\draw[edge] (1,0,0) -- (1,0,0.25) -- (1,0,1); 
								\draw[occluded] (0,1,0) -- (0.192308,1,0) -- (0.596154,1,0); 
								\draw[edge] (0.596154,1,0) -- (1,1,0); 
								\draw[occluded] (0,1,0) -- (0,1,0.5) -- (0,1,0.75); 
								\draw[edge] (0,1,0.75) -- (0,1,1); 
								\draw[edge] (0,0,1) -- (0.403846,0,1) -- (1,0,1); 
								\draw[edge] (0,0,1) -- (0,1,1); 
								\draw[edge] (1,1,0) -- (1,1,1); 
								\draw[edge] (1,0,1) -- (1,1,1); 
								\draw[edge] (0,1,1) -- (1,1,1); 
								\draw[edge] (0,0,-1) -- (0.596154,0,-1) -- (1,0,-1); 
								\draw[occluded] (0,0,-1) -- (0,1,-1); 
								\draw[edge] (0,0,-1) -- (0,0,-0.75) -- (0,0,-0.25); 
								\draw[occluded] (0,0,-0.25) -- (0,0,-1.11022e-16); 
								\draw[edge] (1,0,-1) -- (1,1,-1); 
								\draw[edge] (1,0,-1) -- (1,0,-0.75) -- (1,0,0); 
								\draw[hidden] (0.596154,1,-1) -- (1,1,-1); 
								\draw[occluded] (0,1,-1) -- (0,1,-0.5) -- (0,1,-0.25) -- (0,1,0); 
								\draw[edge] (1,1,-1) -- (1,1,0); 
								\draw[edge] (0,-1,0) -- (0.403846,-1,0) -- (0.807692,-1,0) -- (1,-1,0); 
								\draw[edge] (0,-1,0) -- (0,0,0); 
								\draw[edge] (0,-1,0) -- (0,-1,0.25) -- (0,-1,0.5) -- (0,-1,1); 
								\draw[edge] (1,-1,0) -- (1,0,0); 
								\draw[edge] (1,-1,0) -- (1,-1,0.25) -- (1,-1,0.5) -- (1,-1,1); 
								\draw[edge] (0,-1,1) -- (0.403846,-1,1) -- (0.807692,-1,1) -- (1,-1,1); 
								\draw[edge] (0,-1,1) -- (0,0,1); 
								\draw[edge] (1,-1,1) -- (1,0,1); 
								\draw[edge] (-1,0,0) -- (-0.596154,0,0) -- (-0.403846,0,0); 
								\draw[occluded] (-0.403846,0,0) -- (0,0,0); 
								\draw[occluded] (-1,0,0) -- (-1,1,0); 
								\draw[edge] (-1,0,0) -- (-1,0,0.75) -- (-1,0,1); 
								\draw[occluded] (-1,1,0) -- (-0.807692,1,0) -- (-0.403846,1,0) -- (0,1,0); 
								\draw[hidden] (-1,1,0.75) -- (-1,1,1); 
								\draw[edge] (-1,0,1) -- (-0.596154,0,1) -- (0,0,1); 
								\draw[edge] (-1,0,1) -- (-1,1,1); 
								\draw[edge] (-1,1,1) -- (0,1,1); 
								\draw[edge] (0,-1,-1) -- (1,-1,-1); 
								\draw[edge] (0,-1,-1) -- (0,0,-1); 
								\draw[edge] (0,-1,-1) -- (0,-1,-0.75) -- (0,-1,-0.5) -- (0,-1,-1.11022e-16); 
								\draw[edge] (1,-1,-1) -- (1,0,-1); 
								\draw[edge] (1,-1,-1) -- (1,-1,-0.75) -- (1,-1,-0.5) -- (1,-1,0); 
								\draw[edge] (-1,0,-1) -- (-0.403846,0,-1); 
								\draw[occluded] (-0.403846,0,-1) -- (0,0,-1); 
								\draw[edge] (-1,0,-1) -- (-1,0,-0.25); 
								\draw[occluded] (-1,0,-0.25) -- (-1,0,-1.11022e-16); 
								\draw[edge] (-1,-1,0) -- (-0.596154,-1,0) -- (-0.192308,-1,0) -- (0,-1,0); 
								\draw[edge] (-1,-1,0) -- (-1,0,0); 
								\draw[edge] (-1,-1,0) -- (-1,-1,1); 
								\draw[edge] (-1,-1,1) -- (-0.596154,-1,1) -- (-0.192308,-1,1) -- (0,-1,1); 
								\draw[edge] (-1,-1,1) -- (-1,0,1); 
								\draw[edge] (-1,-1,-1) -- (0,-1,-1); 
								\draw[hidden] (-1,-1,-1) -- (-1,0,-1); 
								\draw[edge] (-1,-1,-1) -- (-1,-1,0); 
								\markpointc{0,0,0}{160}{2.5mm}{south east}
					\end{tikzpicture}}}\)
				\end{tabular}
			}
			
			\vspace{3ex}
			
			
			{\footnotesize
				\begin{tabular}{@{}c@{\hspace{0.35in}}c@{\hspace{0.35in}}c@{}}
					\shortstack{\textbf{Building Boundary}\\\textbf{Single Point}} & \shortstack{\textbf{Building Boundary}\\\textbf{Double Point}} & \shortstack{\textbf{Building Boundary}\\\textbf{Triple Point}}
					\\[1.5mm]
					
					\(\vcenter{\hbox{\begin{tikzpicture}[buildingpic]
								\cubecoords{}{(0,0,0)}
								
								\fill[chamberred] (A) -- (B) -- (E) -- (C) -- cycle; 
								\fill[chamberred] (A) -- (C) -- (G) -- (D) -- cycle; 
								\fill[chamberred] (C) -- (E) -- (H) -- (G) -- cycle; 
								\fill[chamberred] (A) -- (B) -- (F) -- (D) -- cycle; 
								\draw[hidden] (C) -- (E);                            
								\fill[sheetred]   (B) -- (E) -- (H) -- (F) -- cycle;  
								\fill[chamberred] (D) -- (F) -- (H) -- (G) -- cycle; 
								
								\draw[edge] (A) -- (B) -- (E) -- (H) -- (G) -- (D) -- (A);
								\draw[edge] (B) -- (F) -- (D);
								\draw[edge] (F) -- (H);
								\draw[hidden] (A) -- (C);
								\draw[hidden] (C) -- (G);
								
								\markpointc{1,0.5,0.5}{90}{2mm}{south}
					\end{tikzpicture}}}\)
					&
					\(\vcenter{\hbox{\begin{tikzpicture}[buildingpic]
								\cubecoords{A}{(0,0,0)}   
								\cubecoords{B}{(0,0,-1)}  
								\cubecoords{D}{(-1,0,0)}  
								\cubefaces{A}{chamberred}
								\cubefaces{B}{chamberred}
								\cubefaces{D}{chamberred}
								
								\fill[darkredfill] (AA) -- (BA) -- (EA) -- (CA) -- cycle; 
								\fill[darkredfill] (AA) -- (CA) -- (GA) -- (DA) -- cycle; 
								\fill[darkredfill] (AB) -- (CB) -- (GB) -- (DB) -- cycle; 
								\fill[darkredfill] (AD) -- (BD) -- (ED) -- (CD) -- cycle; 
								
								\draw[edge] (0,0,0) -- (0.403846,0,0) -- (1,0,0); 
								\draw[edge] (0,0,0) -- (0,1,0); 
								\draw[edge] (0,0,0) -- (0,0,0.25) -- (0,0,1); 
								\draw[edge] (1,0,0) -- (1,1,0); 
								\draw[edge] (1,0,0) -- (1,0,0.25) -- (1,0,1); 
								\draw[edge] (0,1,0) -- (0.596154,1,0) -- (1,1,0); 
								\draw[edge] (0,1,0) -- (0,1,0.75) -- (0,1,1); 
								\draw[edge] (0,0,1) -- (0.403846,0,1) -- (1,0,1); 
								\draw[edge] (0,0,1) -- (0,1,1); 
								\draw[edge] (1,1,0) -- (1,1,1); 
								\draw[edge] (1,0,1) -- (1,1,1); 
								\draw[edge] (0,1,1) -- (1,1,1); 
								\draw[edge] (0,0,-1) -- (1,0,-1); 
								\draw[edge] (0,0,-1) -- (0,1,-1); 
								\draw[edge] (0,0,-1) -- (0,0,-1.11022e-16); 
								\draw[edge] (1,0,-1) -- (1,1,-1); 
								\draw[edge] (1,0,-1) -- (1,0,-0.75) -- (1,0,0); 
								\draw[hidden] (0,1,-1) -- (0.596154,1,-1) -- (1,1,-1); 
								\draw[edge] (0,1,-1) -- (0,1,-0.25); 
								\draw[occluded] (0,1,-0.25) -- (0,1,0); 
								\draw[edge] (1,1,-1) -- (1,1,0); 
								\draw[edge] (-1,0,0) -- (0,0,0); 
								\draw[edge] (-1,0,0) -- (-1,1,0); 
								\draw[edge] (-1,0,0) -- (-1,0,1); 
								\draw[edge] (-1,1,0) -- (-0.403846,1,0); 
								\draw[occluded] (-0.403846,1,0) -- (0,1,0); 
								\draw[hidden] (-1,1,0) -- (-1,1,0.75) -- (-1,1,1); 
								\draw[edge] (-1,0,1) -- (-0.596154,0,1) -- (0,0,1); 
								\draw[edge] (-1,0,1) -- (-1,1,1); 
								\draw[edge] (-1,1,1) -- (0,1,1); 
								\markpointc{0,0.5,0}{90}{4mm}{south}
					\end{tikzpicture}}}\)
					&
					\(\vcenter{\hbox{\begin{tikzpicture}[buildingpic]
								\cubecoords{A}{(0,0,0)}     
								\cubecoords{B}{(0,0,-1)}    
								\cubecoords{C}{(0,-1,0)}    
								\cubecoords{D}{(-1,0,0)}    
								\cubecoords{F}{(-1,0,-1)}   
								\cubecoords{G}{(-1,-1,0)}   
								\cubecoords{H}{(-1,-1,-1)}  
								\cubefaces{A}{chamberred}
								\cubefaces{B}{chamberred}
								\cubefaces{C}{chamberred}
								\cubefaces{D}{chamberred}
								\cubefaces{F}{chamberred}
								\cubefaces{G}{chamberred}
								\cubefaces{H}{chamberred}
								
								\fill[darkredfill] (AA) -- (BA) -- (EA) -- (CA) -- cycle; 
								\fill[darkredfill] (AA) -- (BA) -- (FA) -- (DA) -- cycle; 
								\fill[darkredfill] (AA) -- (CA) -- (GA) -- (DA) -- cycle; 
								\fill[darkredfill] (AB) -- (BB) -- (FB) -- (DB) -- cycle; 
								\fill[darkredfill] (AB) -- (CB) -- (GB) -- (DB) -- cycle; 
								\fill[darkredfill] (AC) -- (BC) -- (EC) -- (CC) -- cycle; 
								\fill[darkredfill] (AC) -- (CC) -- (GC) -- (DC) -- cycle; 
								\fill[darkredfill] (AD) -- (BD) -- (ED) -- (CD) -- cycle; 
								\fill[darkredfill] (AD) -- (BD) -- (FD) -- (DD) -- cycle; 
								\fill[darkredfill] (BH) -- (EH) -- (HH) -- (FH) -- cycle; 
								\fill[darkredfill] (AF) -- (BF) -- (FF) -- (DF) -- cycle; 
								\fill[darkredfill] (AG) -- (BG) -- (EG) -- (CG) -- cycle; 
								
								\draw[edge] (0,0,0) -- (0.403846,0,0) -- (0.596154,0,0) -- (1,0,0); 
								\draw[occluded] (0,0,0) -- (0,1,0); 
								\draw[edge] (0,0,0) -- (0,0,0.25) -- (0,0,0.75) -- (0,0,1); 
								\draw[edge] (1,0,0) -- (1,1,0); 
								\draw[edge] (1,0,0) -- (1,0,0.25) -- (1,0,1); 
								\draw[occluded] (0,1,0) -- (0.192308,1,0) -- (0.596154,1,0); 
								\draw[edge] (0.596154,1,0) -- (1,1,0); 
								\draw[occluded] (0,1,0) -- (0,1,0.5) -- (0,1,0.75); 
								\draw[edge] (0,1,0.75) -- (0,1,1); 
								\draw[edge] (0,0,1) -- (0.403846,0,1) -- (1,0,1); 
								\draw[edge] (0,0,1) -- (0,1,1); 
								\draw[edge] (1,1,0) -- (1,1,1); 
								\draw[edge] (1,0,1) -- (1,1,1); 
								\draw[edge] (0,1,1) -- (1,1,1); 
								\draw[edge] (0,0,-1) -- (1,0,-1); 
								\draw[occluded] (0,0,-1) -- (0,1,-1); 
								\draw[edge] (0,0,-1) -- (0,0,-0.75) -- (0,0,-0.25); 
								\draw[occluded] (0,0,-0.25) -- (0,0,-1.11022e-16); 
								\draw[edge] (1,0,-1) -- (1,1,-1); 
								\draw[edge] (1,0,-1) -- (1,0,-0.75) -- (1,0,0); 
								\draw[hidden] (0.596154,1,-1) -- (1,1,-1); 
								\draw[occluded] (0,1,-1) -- (0,1,-0.5) -- (0,1,-0.25) -- (0,1,0); 
								\draw[edge] (1,1,-1) -- (1,1,0); 
								\draw[edge] (0,-1,0) -- (0.403846,-1,0) -- (0.807692,-1,0) -- (1,-1,0); 
								\draw[edge] (0,-1,0) -- (0,0,0); 
								\draw[edge] (0,-1,0) -- (0,-1,0.25) -- (0,-1,0.5) -- (0,-1,1); 
								\draw[edge] (1,-1,0) -- (1,0,0); 
								\draw[edge] (1,-1,0) -- (1,-1,0.25) -- (1,-1,0.5) -- (1,-1,1); 
								\draw[edge] (0,-1,1) -- (0.403846,-1,1) -- (0.807692,-1,1) -- (1,-1,1); 
								\draw[edge] (0,-1,1) -- (0,0,1); 
								\draw[edge] (1,-1,1) -- (1,0,1); 
								\draw[edge] (-1,0,0) -- (-0.596154,0,0) -- (-0.403846,0,0); 
								\draw[occluded] (-0.403846,0,0) -- (0,0,0); 
								\draw[occluded] (-1,0,0) -- (-1,1,0); 
								\draw[edge] (-1,0,0) -- (-1,0,0.75) -- (-1,0,1); 
								\draw[occluded] (-1,1,0) -- (-0.807692,1,0) -- (-0.403846,1,0) -- (0,1,0); 
								\draw[hidden] (-1,1,0.75) -- (-1,1,1); 
								\draw[edge] (-1,0,1) -- (-0.596154,0,1) -- (0,0,1); 
								\draw[edge] (-1,0,1) -- (-1,1,1); 
								\draw[edge] (-1,1,1) -- (0,1,1); 
								\draw[edge] (-1,0,-1) -- (-0.403846,0,-1); 
								\draw[occluded] (-0.403846,0,-1) -- (0,0,-1); 
								\draw[edge] (-1,0,-1) -- (-1,0,-0.25); 
								\draw[occluded] (-1,0,-0.25) -- (-1,0,-1.11022e-16); 
								\draw[edge] (-1,-1,0) -- (-0.596154,-1,0) -- (-0.192308,-1,0) -- (0,-1,0); 
								\draw[edge] (-1,-1,0) -- (-1,0,0); 
								\draw[edge] (-1,-1,0) -- (-1,-1,1); 
								\draw[edge] (-1,-1,1) -- (-0.596154,-1,1) -- (-0.192308,-1,1) -- (0,-1,1); 
								\draw[edge] (-1,-1,1) -- (-1,0,1); 
								\draw[edge] (-1,-1,-1) -- (0,-1,-1); 
								\draw[hidden] (-1,-1,-1) -- (-1,0,-1); 
								\draw[edge] (-1,-1,-1) -- (-1,-1,0); 
								\draw[edge] (0,-1,-1) -- (0,0,-1); 
								\draw[edge] (0,-1,-1) -- (0,-1,-0.75) -- (0,-1,-0.5) -- (0,-1,-1.11022e-16); 
								\markpointc{0,0,0}{160}{2.5mm}{south east}
					\end{tikzpicture}}}\)
				\end{tabular}
			}

		\end{minipage}
		\par}
	\caption{All local models around points \(p\) in the building \(B_S\) of any positive \(S\). The darker red portrays \(S\), while the lighter red portrays the rooms with positive wrapping number. See that \(B_S\) admits the structure of a compact and smooth 3-manifold with corners.}
	\label{fig:building}
\end{figure}
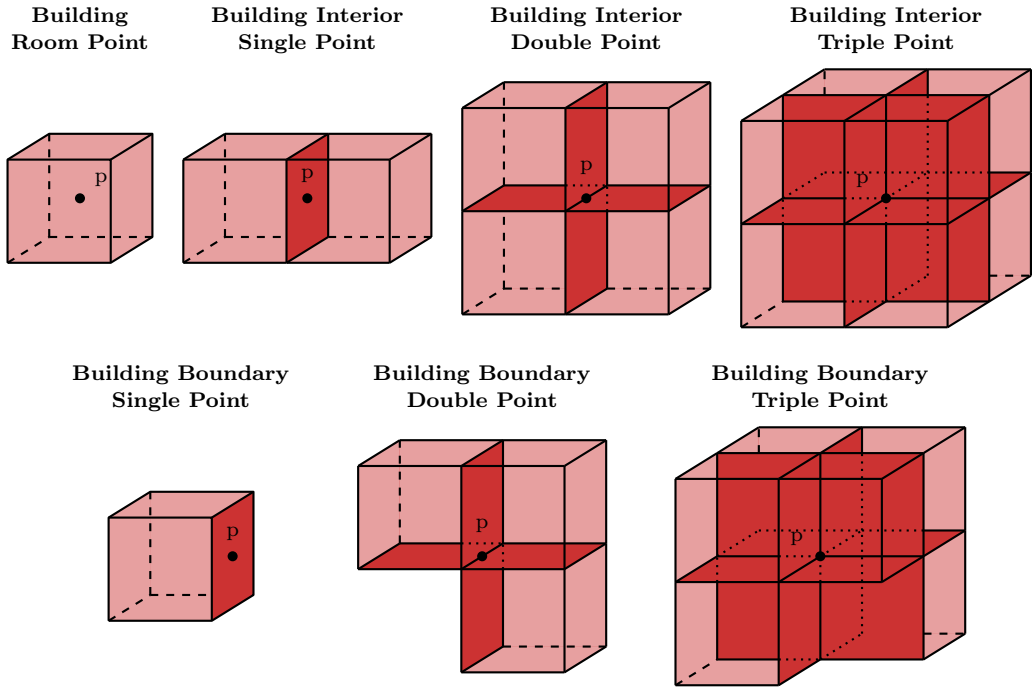

 \cref{def:allpushouts,def:instruction manual topology} construct, for each kit \(K\) of any positive \(S\), a compact space \(E_K\), which we call \emph{the instruction manual of \(K\)}. Each \(E_K\) is equipped with a surjective local homeomorphism\footnote{\emph{A local homeomorphism} is a continuous map that is locally a homeomorphism.} \(\phi_K\) onto \(S\), which we call \emph{the instructions for} \(K\). Any filling yields a kit by \cref{cor:filling to kit}, and the instructions for this kit admit a global section by \cref{thm:push-out}.  The fiber \(\phi_K^{-1}\left(p\right)\) above a point \(p\) on \(S\) is defined to be the set of all diagrams with a certain universal property in a finite category \(\mathscr{G}_p\left(K\right)\) from \cref{def:glue category}.  We call a point in \(\phi_K^{-1}\left(p\right)\) \emph{an adhesion of \(K\) at \(p\)}.  An instruction manual is Hausdorff if and only if \(S\) is embedded. \cref{lem:stalknumbers} computes the cardinality of  \(\phi_K^{-1}\left(p\right)\) to be equal to a product of factorials of the wrapping numbers of the chambers containing \(p\), with exponents in a pattern resembling the binomial coefficients.  
 
\emph{The gluing category \(\mathscr{G}_p\left(K\right)\)} is defined for any kit \(K\) of any \(S\), and amalgamates, above a point \(p\) on \(S\), all ways to adhere together the parts of a certain compactification \(\wh{K}\) of \(K\), which we call \emph{the complete kit of \(K\)}. In \cref{lem:assemblage}, we compactify any kit by providing a geometric interpretation in \cref{prop:cut and glue} of a functor as a way to canonically cut and then glue back together a geometrically divided space. This functor was studied by Ralph \cite{Fox1957} and is portrayed in \cref{fig:resolved chamber pictures}.  If \(S\) is fillable, \cref{thm:push-out} establishes the picture  in \cref{fig:instruction manual picture}, in which a diagram links the singular strata of \(S\) in a way that characterizes how the compact covering spaces that make up \(\wh{K}\) are adhered together along their boundaries so as to build the domain of any filling as a quotient space. \cref{thm:globalsection} uses this picture to show that \(S\) is fillable if and only if at least one \(\phi_K\) admits a global section. 

One knows from  Roger \cite{Godement1958} that under the \'etal\'e-space formulation of sheaf theory, any sheaf of sets is equivalent to some local homeomorphism in the topological category. As Morris Hirsch once said (D. Sullivan, personal communication), ``a cross section of the sheaf of local objects is a global version of said object.'' This is exactly the situation for each kit \(K\) of \(S\) with the sheaf of sections of  its instructions \(\phi_K\). \cref{def:assemblysheaf} constructs a natural sheaf of finite groups \(\mathscr{A}_S\), which we call \emph{the assembly sheaf  of \(S\)}.  Recall that a sheaf of groups over a base always has a first cohomology set that contains a preferred point.  For each surjective local homeomorphism \(\phi_K\), its sheaf of sections is a torsor\footnote{A set \(X\) is a torsor for a group \(G\) if there is an action of \(G\) on \(X\) that is free and has exactly one orbit.} for \(\mathscr{A}_S\) by \cref{prop:assemblysheaftorsor}.  Distinguished for each \(K\) is a degree one cohomology class of \(S\) with coefficients in \(\mathscr{A}_S\) that equals the preferred point if and only if there exists a global section of \(\phi_K\). We call this cohomology class associated with \(K\) \emph{the assemblage class of \(K\)}. Equivalent selections of kit produce the same assemblage class by \cref{lem:isomorphic torsors}. Hence, there are only finitely many assemblage classes of \(S\). It turns out that this first cohomology set is always a finite set.


Any equidimensional immersion is a local diffeomorphism, and any local diffeomorphism must produce isomorphisms between local pictures in the smooth category.   Hence smoothness of a model, or lack thereof, is preserved under any filling.  The domain \(M\) of each filling \(F\) of \(S\) is some undetermined compact and smooth manifold with boundary. Geometrically, \(M\) has no corners: only smooth interior points and smooth boundary points. Generated is the picture that any filling produces isomorphic inclusions of the local models of \cref{fig:local pictures in preimage} into the local models of \cref{fig:building}. Assume here that the dimension \(d\) is greater than zero.

 \begin{maintheorem}
 	A generically immersed \(d\)-sphere \(S\) inside Euclidean \(\left(d+1\right)\)-space is fillable if and only if \(S\) is positive and some assemblage class of \(S\) equals the preferred point in the first cohomology of \(S\) with coefficients in its assembly sheaf. Furthermore, the number of inequivalent fillings of \(S\) equals the number of orbits of the deck group of a kit acting freely on the set of global sections of its instructions, summed over one selection of kit for each element in the set of inequivalent selections of kit from the hardware store of \(S\). 
 	\label{thm:maintheorem}
 \end{maintheorem}


The following story is meant to convey to the reader the nature of the argument used to prove the main theorem, in an intuitive yet precisable way. One may notice this story differs from the formal exposition in that the story concerns chambers rather than rooms, but properties of \emph{the \citeauthor{Fox1957} completion} functor make these formulations equivalent by \cref{lem:assemblage}. Suppose you are given a generically immersed 2-sphere \(S\) in 3-space on your workbench in front of you, with the problem being to fill \( S\) with some compact and smooth 3-manifold. One knows of an a priori way to obstruct the existence of a filling via a simple measurement involving finitely many rays that intersect \(S\) finitely many times. In the case that the problem might still have a solution, there is now a nonnegative wrapping number \(w\left(C\right)\) assigned to each chamber \(C\). Break \(S\) down into the disjoint union of its chambers, then check which chambers are manifolds. For each chamber that is not a manifold, use your Fox completion X-Acto knife to resolve the chamber canonically into a manifold with corners by cutting the chamber along its mould lines, then pulling apart. Now go to the hardware store that has one aisle for each chamber of \(S\). For each chamber \(C\), aisle \(C\) contains exactly the isomorphism classes of all degree \(w\left(C\right)\) covering spaces of the resolved copy of \(C\), so including the disconnected covering spaces. If \(w\left(C\right) = 0\), then aisle \(C\) is empty. If \(w\left(C\right) = 1\), or the room of \(C\) is simply connected, then aisle \(C\) has exactly one isomorphism class of a part. For each chamber \(C\), select from aisle \(C\) exactly one part. Once exactly one part from each aisle has been selected, the complete kit now obtained comes automatically with its instruction manual and instructions for how to adhere together the parts of that kit.  Go back to your workbench with the selected kit, then try to follow its instructions so as to assemble all parts of that kit together into some smooth 3-manifold with boundary. As you assemble a filling from a kit, there might be choices which lead to assemblages that cannot be continued, while other choices might lead to assemblages that fill in inequivalent ways. If for some selection of kit no filling can be built, then one obtains an algebraic reason why no assemblage of parts will work for any equivalent selection of kit. If this happens, then go back to the hardware store of \(S\) and exchange at least one part to obtain an inequivalent kit, and hence new instructions. Go back to your workbench, then try to assemble a filling from this new complete kit, and so on, until all finitely many inequivalent selections of kit from the hardware store of \(S\) are exhausted. Suppose there exists some assemblage of parts from a kit that forms a filling of \(S\), and that this kit has two pieces from the same part related by a deck transformation. While you assemble those pieces of that part, you must choose which isomorphic piece to adhere where. Despite this choice, the two fillings obtained are equivalent because the gluing geometry for how the parts are adhered has not changed. The first time any part is adhered during the construction of a filling, the full deck group of that part presents symmetries that only result in mosaic recolorings of the domain that all yield equivalent fillings. 

\begin{figure}[h]
	\centering
	\includegraphics[width=0.8\textwidth]{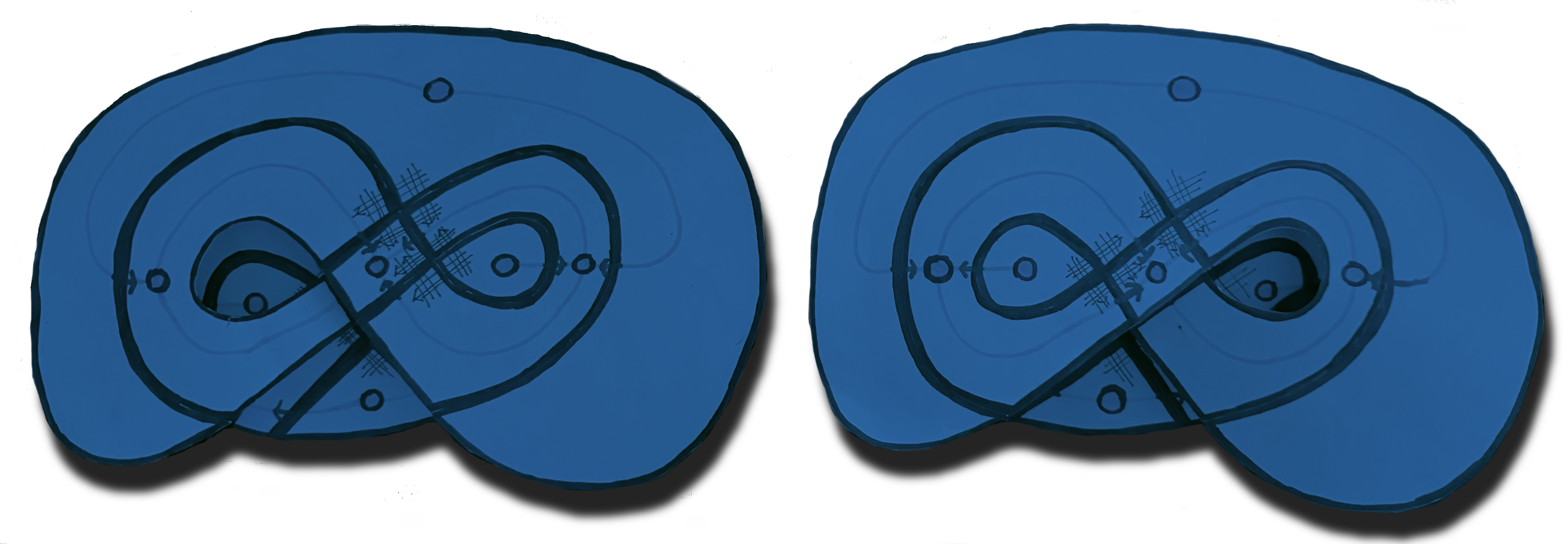}
	\caption{Two inequivalent fillings, by the disk, of an immersed circle in the plane. Both models are built from the same kit of paper pieces from the office supply store. The underlying generically immersed circle in the plane being filled is an example due to John Milnor.}
	\label{fig:paper immersions}
\end{figure}

\section{Three A Priori Obstructions to Bounding an Immersed Compact 3-Manifold in 3-Space}

Below are three a priori obstructions that provide useful necessary conditions for a generically immersed 2-sphere \(S\) in 3-space to be fillable by a compact 3-manifold. The first a priori obstruction concerns the existence of an `inward-pointing' direction field along \(S\). \emph{A normal direction field along \(S\)} is a section of the sphere bundle of a normal bundle of \(S\), and up to equivalence there are exactly two. A normal direction field records the direction in which a filling begins. The first a priori obstruction shows that no compact filling can throw a point into the unbounded room lest the extension fold back and crease itself outside the category of immersions. The second a priori obstruction gives the aforementioned simple conditions on orientations and wrapping numbers. This necessary condition imposed by this second obstruction is a nonnegativity condition on wrapping numbers, of which this inward-pointing direction field condition is a special case. The first a priori obstruction is nonetheless interesting because, when the condition imposed by this obstruction is satisfied, there exists a well-defined convention for consistently computing wrapping numbers in a way that does not involve choosing orientations. One simply endows \(S\) with \emph{the inward-pointing direction field along \(S\)} and adopts our natural convention for computing wrapping numbers that is portrayed in \cref{fig:ray intersection}. To compute, relative to any transverse ray running towards infinity, one simply subtracts the number of outward-pointing intersections from the number of inward-pointing intersections with \(S\). The technique created by the third a priori obstruction is portrayed in \cref{fig:transversal slice} and applied in \cref{fig:milnorpush} to prove that another condition beyond nonnegativity is needed to find an extension.

Orientations can be subtle, so this first a priori obstruction proposes a natural convention that liberates one from choosing orientations without any loss of information. Notice that only after choosing both a hand with which to orient 3-space and a normal direction at just one point of \(S\) is a unique compatible orientation for the 2-sphere determined. Each filling \(F\) of \(S\) by \(M\) uniquely determines a normal direction field. Simply look at the normal direction to \(S\) determined by the image of any inward-pointing vector at the boundary of \(M\) under the injective derivative of \(F\).  All fillings with compact domain determine equivalent normal direction fields because no compact filling can throw a point into the unbounded room. 

Call the intersection of \(S\) with its unbounded chamber \emph{the peripheral boundary of \(S\)}. To determine whether an inward-pointing normal direction field exists, choose any point \(p\) on the peripheral boundary of \(S\) and take the normal direction there that points inward away from the unbounded room. Any normal direction at a point uniquely distinguishes the equivalence class of normal direction fields which coincide with the direction chosen at that point. If any other point on the peripheral boundary has determined an outward-pointing direction because of the choice at \(p\), then \(S\) is not fillable with a compact domain. The domain of any filling \(F\) of \(S\) here is assumed compact.

\begin{lem}[First a priori obstruction] 	\label{lem:unbounded room}
	The intersection of the image of any filling \(F\) of \(S\) with the unbounded room of \(S\) is empty.
\end{lem}
\begin{proof}
	Define on the complement of \(S\) in 3-space a relation that relates two points if their fibers under \(F\) have equal cardinality. All these cardinalities are finite because the domain of \(F\) is assumed compact. Equality of integers is an equivalence relation and any equivalence relation partitions its underlying set into disjoint subsets. The restriction of  \(F\) to the inverse image of the complement of \(S\) in 3-space being a proper local diffeomorphism implies that the equivalence classes of this equivalence relation are open, so all points in the same room have equal fiber cardinality.  If the intersection of the image of \(F\) with the unbounded room were nonempty, then each point of the unbounded room would be in the image of \(F\). However, the continuous image of any compact space is compact and hence bounded in 3-space, producing a contradiction.
\end{proof}

If there exists an inward-pointing direction field along \(S\), then our natural convention computes all wrapping numbers in a consistent manner.  The nonexistence of an inward-pointing direction field implies that there must always exist a room with wrapping number negative one that has a wall on the peripheral boundary. The resulting geometric contradiction is that any ray leaving a negative room has more material going out than in, forcing some material into the unbounded room lest the extension fold back before then. 

\begin{lem}[Second a priori obstruction] 	\label{lem:wrapping numbers}
Let \(F\) be any filling of \(S\) and \(w\left(R\right)\) a wrapping number of any room \(R\). Then \(w\left(R\right)\) is nonnegative when computed using our natural convention relative to the inward-pointing direction field determined by the derivative of \(F\). If \(w\left(R\right)=0\), then the inverse image \(F^{-1}\left(R\right)\) is empty; otherwise, the restriction map \(F_R : F^{-1}\left(R\right) \to R\) is a degree \(w\left(R\right)\) covering map.
\end{lem}
\begin{proof}
	It is well known that any proper local homeomorphism with a Hausdorff domain is a covering map; therefore, if \(F^{-1}\left(R\right)\) is nonempty, then \(F_R\) is a covering map. Now choose any point \(p\) in \(R\) and let \(L\) be any ray with origin \(p\) that exits \(R\) and intersects \(S\) transversely on its way to infinity.  We prove that the degree \(\# F^{-1}\left(p\right)\) over \(R\) equals \(w\left(R\right)\) by first proving a wall-crossing formula, then inducting on the number of intersections of \(L\) with \(S\). Let \(I\) be any short line segment embedded in 3-space that intersects \(S\) transversely at exactly one point \(q\), which is a single point of \(S\). Suppose that \(I\) is oriented from its endpoint \(e'\) in \(R'\) towards its other nearby endpoint \(e\) in \(R\). Our wall-crossing formula is the equality of differences
\begin{align*} w\left(R\right) - w\left(R'\right) = \#F^{-1}\left(e\right) - \#F^{-1}\left(e'\right).
	\end{align*}
The difference in wrapping numbers equals positive one if the normal direction at \(q\) points into \(R\). The signed intersection number of \(I\) with \(S\) equals positive one if the normal direction at \(q\) points towards \(e\). Therefore, \(w\left(R\right) - w\left(R'\right)\) equals the signed intersection number of \(I\) with \(S\) by convention. The inverse image \(F^{-1}\left(I\right)\) is a disjoint union of finitely many intervals. Each interval is embedded, and all but one lie inside the interior, with the exception having on the boundary exactly an endpoint. In the difference of degrees, the endpoints of the interior intervals cancel. The last interval will contribute positive one in a manner consistent with the above because the direction in which to begin a filling of \(S\) coincides with the normal direction field along \(S\). 

The unbounded room is the unique room  for which  \(L\) can intersect \(S\) zero times. \cref{lem:unbounded room} proves the base case as \(F\) can throw no point into the unbounded room. The induction hypothesis says that if some ray \(L'\) whose origin is a room point \(p'\) intersects \(S\) at most \(n-1\) times before running off to infinity, then \(w\left(R'\right) = \#F^{-1}\left(p'\right)\), where \(R'\) always denotes the room containing \(p'\). Suppose \(L\) intersects \(S\) exactly \(n\) times, with the intersection points \(\left\{q_1,q_2,\dots,q_n\right\}\) ordered from \(p\) walking towards infinity. Consider a room point \(p'\) on \(L\) between \(q_1\) and \(q_2\). Apply the induction hypothesis to the subray of \(L\) with origin \(p'\) to deduce  that \(w\left(R'\right) = \#F^{-1}\left(p'\right)\). Intersect a small open neighborhood of \(q_1\) with \(L\) and call that intersection \(I\), then give \(I\) the orientation induced from \(L\). The wall-crossing formula implies that \(w\left(R'\right) - w\left(R\right) = \#F^{-1}\left(p'\right) - \#F^{-1}\left(p\right)\), where substituting \(w\left(R'\right) = \#F^{-1}\left(p'\right)\) yields the desired equation.
\end{proof}

\begin{cor}		\label{cor:filling to kit}
		Let \(F\) be any filling of \(S\). Then \(S\) is positive and the image of \(F\) equals the building \(B_S\) of \(S\). Furthermore, the inverse image \(F^{-1}\left(B_S \setminus S\right)\) of the complement of \(S\) inside \(B_S\), with its covering map given by restricting \(F\) to  \(F^{-1}\left(B_S \setminus S\right)\), is a kit of \(S\).
\end{cor}

Since all wrapping numbers are nonnegative and change by exactly one across each wall in a way portrayed by \cref{fig:ray intersection,fig:instruction manual picture}, the models portrayed in \cref{fig:local chamber pictures,fig:building} are indeed the local models. One reason is that locally only rooms with the same wrapping number might be contained in the same room globally. Furthermore, there always exists a unique chamber with minimum wrapping number as well as a unique chamber with maximum wrapping number among all chambers containing a given point in 3-space.  The local models of \cref{fig:local pictures in preimage} are justified simply because \(F\) is a local diffeomorphism of some 3-manifold with boundary into 3-space. We have all local models of \(S\) already given to us within \cref{fig:local pictures}, so we simply transport those local models back to the domain. 

Portrayed in \cref{fig:transversal slice}, the third a priori obstruction involves slicing \(S\) transversely with embedded planes and checking whether those intersections are fillable inside those intersecting planes. This third obstruction suggests performing all sorts of planar slicing tests to obstruct fillability. A difficulty in constructing fillings from this level-by-level point of view is transporting a choice of extension at one level consistently to all the other levels. Moreover, the phenomenon of inequivalent extensions adds another layer of difficulty there.

 \begin{figure}[h]
 	{\centering
 		\begin{minipage}{\textwidth}
 			\centering
 			
 			\tdplotsetmaincoords{60}{300}
 			\scalebox{\rayScale}{%
 				\begin{tikzpicture}[line join=round, >=Stealth, tdplot_main_coords]
 					
 					\draw[wavyplane]
 					(4.9676,-2.4674,-2.4182)
 					.. controls (4.9676,-0.9072,-2.6160) and (4.9676,0.8781,-2.8054) .. (4.9676,2.6402,-2.4614)
 					.. controls (4.9676,2.1816,-0.7433) and (4.9676,2.8965,0.6706) .. (4.9676,2.4492,2.4856)
 					.. controls (4.9676,0.6969,3.0353) and (4.9676,-0.9683,3.0052) .. (4.9676,-2.5902,2.4510)
 					.. controls (4.9676,-2.2285,0.8155) and (4.9676,-2.0394,-0.6366) .. (4.9676,-2.4674,-2.4182);
 					
 					\draw[wavyplane]
 					(10.4702,-2.5877,-2.5639)
 					.. controls (10.4702,-1.0006,-2.5759) and (10.4702,0.5273,-2.7169) .. (10.4702,2.2301,-2.4622)
 					.. controls (10.4702,2.8756,-0.6317) and (10.4702,2.9344,1.0079) .. (10.4702,2.4183,2.6026)
 					.. controls (10.4702,0.8779,2.4214) and (10.4702,-0.6599,2.3004) .. (10.4702,-2.4368,2.5099)
 					.. controls (10.4702,-2.8216,0.6880) and (10.4702,-2.4755,-0.8304) .. (10.4702,-2.5877,-2.5639);
 					
 					\draw[wavyplane]
 					(15.9729,-2.2356,-2.2863)
 					.. controls (15.9729,-0.6480,-1.9796) and (15.9729,0.6475,-2.6601) .. (15.9729,2.2166,-2.5787)
 					.. controls (15.9729,1.9740,-0.7269) and (15.9729,2.3694,0.8600) .. (15.9729,2.6338,2.4318)
 					.. controls (15.9729,1.1298,2.0601) and (15.9729,-0.4632,2.1861) .. (15.9729,-2.2253,2.3986)
 					.. controls (15.9729,-2.5923,0.8243) and (15.9729,-1.9332,-0.7130) .. (15.9729,-2.2356,-2.2863);
 					
 					\draw[wavyplane]
 					(21.4755,-2.5874,-2.3393)
 					.. controls (21.4755,-1.0742,-2.9007) and (21.4755,0.7719,-2.7371) .. (21.4755,2.2764,-2.3797)
 					.. controls (21.4755,2.0103,-0.6551) and (21.4755,2.2046,0.8645) .. (21.4755,2.2698,2.4093)
 					.. controls (21.4755,0.7840,2.0067) and (21.4755,-0.9097,2.9326) .. (21.4755,-2.5773,2.3401)
 					.. controls (21.4755,-2.9182,0.8551) and (21.4755,-2.4466,-0.6230) .. (21.4755,-2.5874,-2.3393);
 					
 					\draw[wavyplane]
 					(26.9781,-2.6336,-2.4076)
 					.. controls (26.9781,-1.1041,-2.5304) and (26.9781,0.5193,-2.8394) .. (26.9781,2.3278,-2.3992)
 					.. controls (26.9781,2.8291,-0.8807) and (26.9781,2.7516,0.7046) .. (26.9781,2.3867,2.4241)
 					.. controls (26.9781,0.9669,2.7304) and (26.9781,-0.8434,2.9508) .. (26.9781,-2.2351,2.4381)
 					.. controls (26.9781,-2.8203,0.8545) and (26.9781,-2.6455,-0.6466) .. (26.9781,-2.6336,-2.4076);
 					
 					\draw[rayvis] (0,0,0) -- (4.9812,0,0);
 					\draw[rayocc] (4.9812,0,0) -- (8.8230,0,0);
 					\draw[rayvis] (8.8230,0,0) -- (10.4834,0,0);
 					\draw[rayocc] (10.4834,0,0) -- (14.1949,0,0);
 					\draw[rayvis] (14.1949,0,0) -- (15.9856,0,0);
 					\draw[rayocc] (15.9856,0,0) -- (19.5505,0,0);
 					\draw[rayvis] (19.5505,0,0) -- (21.4877,0,0);
 					\draw[rayocc] (21.4877,0,0) -- (25.1829,0,0);
 					\draw[rayvis] (25.1829,0,0) -- (26.9898,0,0);
 					\draw[rayocc] (26.9898,0,0) -- (30.7013,0,0);
 					\draw[raytip, rayvis] (30.7013,0,0) -- (32.5571,0,0);
 					\fill[black] (0,0,0) circle (3.668pt);
 					
 					\node[anchor=north west, inner sep=3.057pt] (origin) at (0,0,0) {};
 					
 					\fill[outward] (4.9676,0,0) circle (2.446pt);
 					\fill[inward] (10.4702,0,0) circle (2.446pt);
 					\fill[outward] (15.9729,0,0) circle (2.446pt);
 					\fill[inward] (21.4755,0,0) circle (2.446pt);
 					\fill[inward] (26.9781,0,0) circle (2.446pt);
 					
 					\draw[normaltip, outward] (4.9676,0,0) -- (5.8847,0,0);
 					\draw[normaltip, inward] (10.4702,0,0) -- (9.5531,0,0);
 					\draw[normaltip, outward] (15.9729,0,0) -- (16.8900,0,0);
 					\draw[normaltip, inward] (21.4755,0,0) -- (20.5584,0,0);
 					\draw[normaltip, inward] (26.9781,0,0) -- (26.0610,0,0);

 				\end{tikzpicture}
 			}
 			
 		\end{minipage}
 		\par}
 	\caption{A ray exiting a room and intersecting an image \(S\) equipped with a normal direction field. There are exactly 5 intersections. There are 3 inward-pointing intersections and 2 outward-pointing intersections, so the wrapping number of that room is equal to 1.}
 	\label{fig:ray intersection}
 \end{figure}
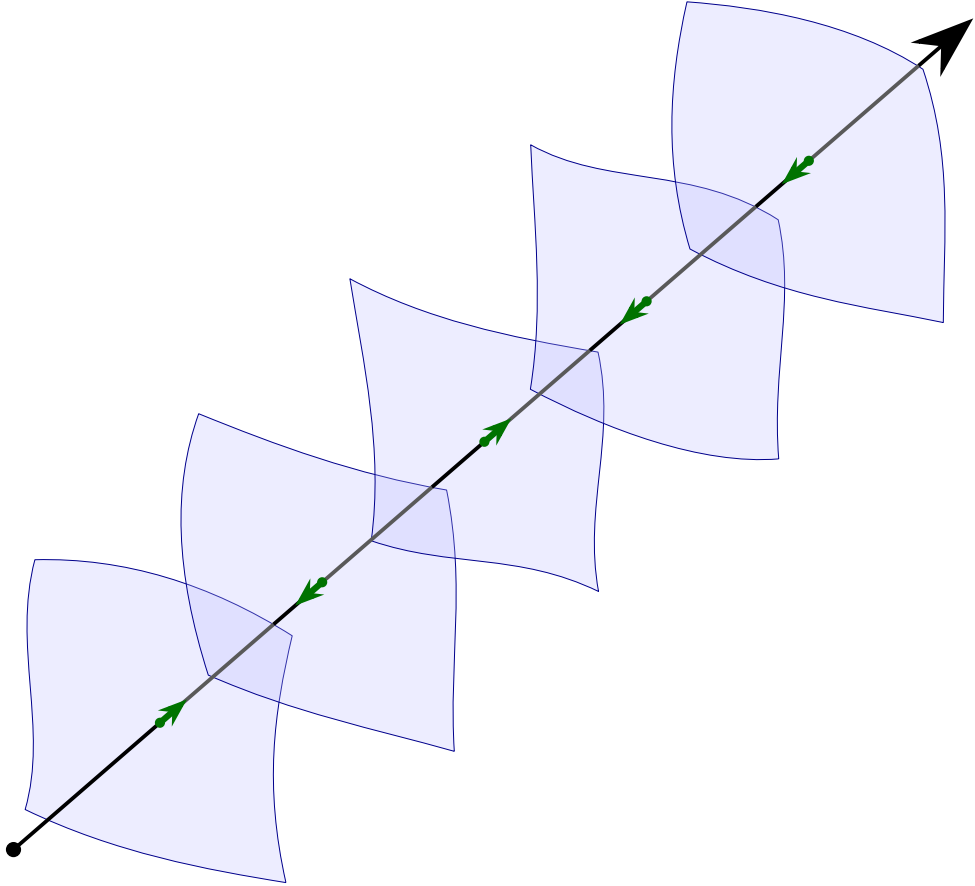
 
 \begin{prop}[Third a priori obstruction]  	\label{prop:transversal slice}
 If \(S\) is fillable, then the transverse intersection of \(S\) with any plane \(P\) inside 3-space is a fillable immersion of a  closed 1-manifold inside that intersecting plane.
 \end{prop}
 \begin{proof}
 	The intersection is an immersed 1-manifold  by transversality and closed because so is the 2-sphere. The inverse image \(F^{-1}\left(P\right)\) of \(P\) under any filling \(F\) of \(S\) is embedded because \(P\) is embedded and \(F\) is a local diffeomorphism of a manifold with boundary. Transversality makes \(F^{-1}\left(P\right)\) a compact 2-manifold with exactly its boundary embedded on the boundary of the domain of \(F\). Then, the restriction to  \(F^{-1}\left(P\right)\) yields a desired filling by a compact 2-manifold.
 \end{proof}

\section{An Assembly of a Kit}

We have now seen the finitistic nature of fillings: each filling of any generically immersed 2-sphere \(S\) in 3-space produces a covering space whose isomorphism class can be selected by a section of the hardware store bundle of \(S\). Let \(F\) be any filling of \(S\) by \(M\). \emph{The kit assembled by the filling \(F\)} is defined to be the subspace \textbf{\(K_F\)} of \(M\) equal to the inverse image \(F^{-1}\left(B_S \setminus S\right)\) of the complement of \(S\) inside its building \(B_S\). \cref{fig:local chamber pictures} portrays how a chamber \(\ol{R}\) can fail to be a manifold, and \cref{fig:resolved chamber pictures} portrays its canonical resolution into a compact, orientable, connected, and smooth 3-manifold \(\wh{R}\) with corners. As a space, this resolution is defined to be \emph{the 1957 \citeauthor{Fox1957} completion} of the room.\footnote{For the sake of concision, the phrase \emph{the Fox completion of a room} will be confounded to mean the domain of the Fox completion of the inclusion spread of a room into its chamber, where no confusion should arise.} By a universal property there is a quotient map from \(\wh{R}\) to \(\ol{R}\) that restricts to a diffeomorphism from the interior of \(\wh{R}\) onto the room \(R\) living as a dense subspace of \(\ol{R}\), making the identifications exactly along the mould lines of \(\ol{R}\). \emph{The complete kit  assembled by the filling \(F\)} is defined to be the domain \(\wh{K_F}\) of the Fox completion of the restriction map of \(F\) to \(K_F\) over \(B_S\). The complete kit assembled by \(F\) is unique up to unique isomorphism. \emph{The gluing of the complete kit assembled by the filling \(F\)} is defined to be the quotient map \(g_F : \wh{K_F} \to M\) obtained via functoriality of the Fox completion.  For a kit \(K\) of \(S\), define \emph{the complete kit  of \(K\)} to be the domain \(\wh{K}\) of the Fox completion of the covering map of \(K\) over the complement of \(S\) inside \(B_S\) that is given by virtue of our definition of a kit.
 
\begin{prop}[How to cut and then glue back together any triangulated space] 	\label{prop:cut and glue}
	Let \(M\) be any triangulated space of dimension \(d\) and \(\Delta^\circ\) the disjoint union of all open \(d\)-simplices of \(M\). If \(\iota : \Delta^\circ \hookrightarrow M\) is the inclusion spread,  then the Fox completion spread \(\wh{\iota} : \wh{\Delta^\circ} \to M\) is the quotient map from the disjoint union of closed \(d\)-simplices \(\wh{\Delta^\circ}\) that is \(\iota\) when restricted along a dense and locally connected inclusion \(\Delta^\circ \hookrightarrow \wh{\Delta^\circ}\).
\end{prop}

\begin{proof}
	The proof is an exercise in understanding the definitions provided by \cite{Fox1957}. The definitions used here are recalled below for the convenience of the reader.
\end{proof}

\begin{figure}[h!] 
	{\centering
		\begin{minipage}{\textwidth}
			\centering
			
			
			{\footnotesize
				\begin{tabular}{@{}c@{\hspace{0.3in}}c@{\hspace{0.3in}}c@{}}
					\shortstack{\textbf{Resolved}\\\textbf{Edge Point II}} & \shortstack{\textbf{Resolved}\\\textbf{Corner Point II}} & \shortstack{\textbf{Resolved}\\\textbf{Corner Point III}}
					\\[1.5mm]
					
					\(\vcenter{\hbox{\begin{tikzpicture}[cubepic2, line cap=round]
								\cubecoords{}{(0.070711,0,-0.070711)}    
								\cubecoords{2}{(-1.070711,0,1.070711)}   
								
								\fill[chamberorange] (A) -- (B) -- (E) -- (C) -- cycle;
								\fill[sheetorange]   (A) -- (C) -- (G) -- (D) -- cycle;
								\fill[chamberorange] (C) -- (E) -- (H) -- (G) -- cycle;
								\fill[chamberorange] (A) -- (B) -- (F) -- (D) -- cycle;
								\fill[chamberorange] (B) -- (E) -- (H) -- (F) -- cycle;
								\draw[hidden] (A) -- (C);
								\draw[hidden] (C) -- (E);
								\draw[hidden] (C) -- (G);
								\fill[sheetorange]   (D) -- (F) -- (H) -- (G) -- cycle;
								
								\fill[sheetorange]   (A2) -- (B2) -- (E2) -- (C2) -- cycle;
								\fill[chamberorange] (A2) -- (C2) -- (G2) -- (D2) -- cycle;
								\fill[chamberorange] (C2) -- (E2) -- (H2) -- (G2) -- cycle;
								\fill[chamberorange] (A2) -- (B2) -- (F2) -- (D2) -- cycle;
								\draw[hidden] (A2) -- (C2);
								\draw[hidden] (C2) -- (E2);
								\draw[hidden] (C2) -- (G2);
								\fill[sheetorange]   (B2) -- (E2) -- (H2) -- (F2) -- cycle;
								\fill[chamberorange] (D2) -- (F2) -- (H2) -- (G2) -- cycle;
								
								\draw[edge] (A) -- (B) -- (E) -- (H) -- (G);
								\draw[edge] (D) -- (A);
								\draw[edge] (B) -- (F) -- (D);
								\draw[edge] (F) -- (H);
								
								\draw[edge] (A2) -- (B2);
								\draw[edge] (E2) -- (H2) -- (G2) -- (D2) -- (A2);
								\draw[edge] (B2) -- (F2) -- (D2);
								\draw[edge] (F2) -- (H2);
								
								\draw[glued] (G) -- (D);   
								\draw[glued] (B2) -- (E2); 
								
								\markpointas{0.070711,0.5,0.929289}{-1}{0.25mm}{\(p''\)}
								\markpointas{-0.070711,0.5,1.070711}{90}{0.25mm}{\(p'\)}
					\end{tikzpicture}}}\)
					&
					\(\vcenter{\hbox{\begin{tikzpicture}[cubepic2, line cap=round]
								\cubecoords{2}{(-1,0.070711,1.070711)}   
								\cubecoords{3}{(-1,-1.070711,-0.070711)} 
								
								\draw[hidden] (C2) -- (G2);
								\fill[sheetorange]   (A2) -- (B2) -- (E2) -- (C2) -- cycle;
								\fill[chamberorange] (A2) -- (C2) -- (G2) -- (D2) -- cycle;
								\fill[chamberorange] (C2) -- (E2) -- (H2) -- (G2) -- cycle;
								\fill[sheetorange]   (A2) -- (B2) -- (F2) -- (D2) -- cycle;
								\fill[sheetorange]   (B2) -- (E2) -- (H2) -- (F2) -- cycle;
								\fill[chamberorange] (D2) -- (F2) -- (H2) -- (G2) -- cycle;
								
								\draw[hidden] (C3) -- (E3);
								\draw[hidden] (C3) -- (G3);
								\fill[chamberorange] (A3) -- (B3) -- (E3) -- (C3) -- cycle;
								\fill[chamberorange] (A3) -- (C3) -- (G3) -- (D3) -- cycle;
								\fill[sheetorange]   (C3) -- (E3) -- (H3) -- (G3) -- cycle;
								\fill[chamberorange] (A3) -- (B3) -- (F3) -- (D3) -- cycle;
								\fill[sheetorange]   (B3) -- (E3) -- (H3) -- (F3) -- cycle;
								\fill[sheetorange]   (D3) -- (F3) -- (H3) -- (G3) -- cycle;
								
								\draw[edge] (B2) -- (E2) -- (H2) -- (G2) -- (D2) -- (A2);
								\draw[edge] (B2) -- (F2) -- (D2);
								\draw[edge] (F2) -- (H2);
								\draw[occluded] (A2) -- (C2);
								\draw[occluded] (C2) -- (E2);
								
								\draw[edge] (E3) -- (H3);
								\draw[edge] (D3) -- (A3) -- (B3) -- (E3);
								\draw[edge] (G3) -- (D3);
								\draw[edge] (B3) -- (F3) -- (D3);
								\draw[edge] (F3) -- (H3);
								\draw[hidden] (A3) -- (C3);
								
								\draw[glued] (A2) -- (B2); 
								\draw[glued] (G3) -- (H3); 
								
								\markpointas{0,0.070711,1.070711}{135}{-0.9mm}{\(p'\)}
								\markpointas{0,-0.070711,0.929289}{267}{0.4mm}{\(p''\)}
					\end{tikzpicture}}}\)
					&
					\(\vcenter{\hbox{\begin{tikzpicture}[cubepic2, line cap=round]
								\cubecoords{}{(0.141421,0.070711,-0.070711)}    
								\cubecoords{2}{(-1.070711,0.070711,1.141421)}   
								\cubecoords{3}{(-1.070711,-1.141421,-0.070711)} 
								
								\draw[hidden] (C) -- (E);
								\fill[chamberorange] (A) -- (B) -- (E) -- (C) -- cycle;
								\fill[sheetorange]   (A) -- (C) -- (G) -- (D) -- cycle;
								\fill[chamberorange] (C) -- (E) -- (H) -- (G) -- cycle;
								\fill[sheetorange]   (A) -- (B) -- (F) -- (D) -- cycle;
								\fill[chamberorange] (B) -- (E) -- (H) -- (F) -- cycle;
								\fill[sheetorange]   (D) -- (F) -- (H) -- (G) -- cycle;
								
								\draw[hidden] (C2) -- (G2);
								\fill[sheetorange]   (A2) -- (B2) -- (E2) -- (C2) -- cycle;
								\fill[chamberorange] (A2) -- (C2) -- (G2) -- (D2) -- cycle;
								\fill[chamberorange] (C2) -- (E2) -- (H2) -- (G2) -- cycle;
								\fill[sheetorange]   (A2) -- (B2) -- (F2) -- (D2) -- cycle;
								\fill[sheetorange]   (B2) -- (E2) -- (H2) -- (F2) -- cycle;
								\fill[chamberorange] (D2) -- (F2) -- (H2) -- (G2) -- cycle;
								
								\draw[hidden] (C3) -- (E3);
								\draw[hidden] (C3) -- (G3);
								\fill[chamberorange] (A3) -- (B3) -- (E3) -- (C3) -- cycle;
								\fill[chamberorange] (A3) -- (C3) -- (G3) -- (D3) -- cycle;
								\fill[sheetorange]   (C3) -- (E3) -- (H3) -- (G3) -- cycle;
								\fill[chamberorange] (A3) -- (B3) -- (F3) -- (D3) -- cycle;
								\fill[sheetorange]   (B3) -- (E3) -- (H3) -- (F3) -- cycle;
								\fill[sheetorange]   (D3) -- (F3) -- (H3) -- (G3) -- cycle;
								
								\draw[edge] (A) -- (B) -- (E) -- (H) -- (G);
								\draw[edge] (B) -- (F) -- (D);
								\draw[edge] (F) -- (H);
								\draw[occluded] (A) -- (C);
								\draw[occluded] (C) -- (G);
								
								\draw[edge] (E2) -- (H2) -- (G2) -- (D2) -- (A2);
								\draw[edge] (B2) -- (F2) -- (D2);
								\draw[edge] (F2) -- (H2);
								\draw[occluded] (A2) -- (C2);
								\draw[occluded] (C2) -- (E2);
								
								\draw[edge] (D3) -- (A3) -- (B3) -- (E3);
								\draw[edge] (G3) -- (D3);
								\draw[edge] (B3) -- (F3) -- (D3);
								\draw[edge] (F3) -- (H3);
								\draw[hidden] (A3) -- (C3);
								
								\draw[glued] (G) -- (D);   
								\draw[glued] (D) -- (A);   
								\draw[glued] (B2) -- (E2); 
								\draw[glued] (A2) -- (B2); 
								\draw[glued] (E3) -- (H3); 
								\draw[glued] (G3) -- (H3); 
								
								\markpointas{0.141421,0.070711,0.929289}{-45}{0.1mm}{\(p'''\)}
								\markpointas{-0.070711,0.070711,1.141421}{135}{-0.9mm}{\(p'\)}
								\markpointas{-0.070711,-0.141421,0.929289}{267}{0.4mm}{\(p''\)}
					\end{tikzpicture}}}\)
				\end{tabular}
			}
			
		\end{minipage}
		\par}
	\caption{The resolution into a compact, orientable, connected, and smooth 3-manifold \(\wh{R}\) with corners obtained via the Fox completion of a room \(R\) of \(S\). The darker orange portrays the walls, while the lighter orange portrays the rooms.}
		\label{fig:resolved chamber pictures}
\end{figure}
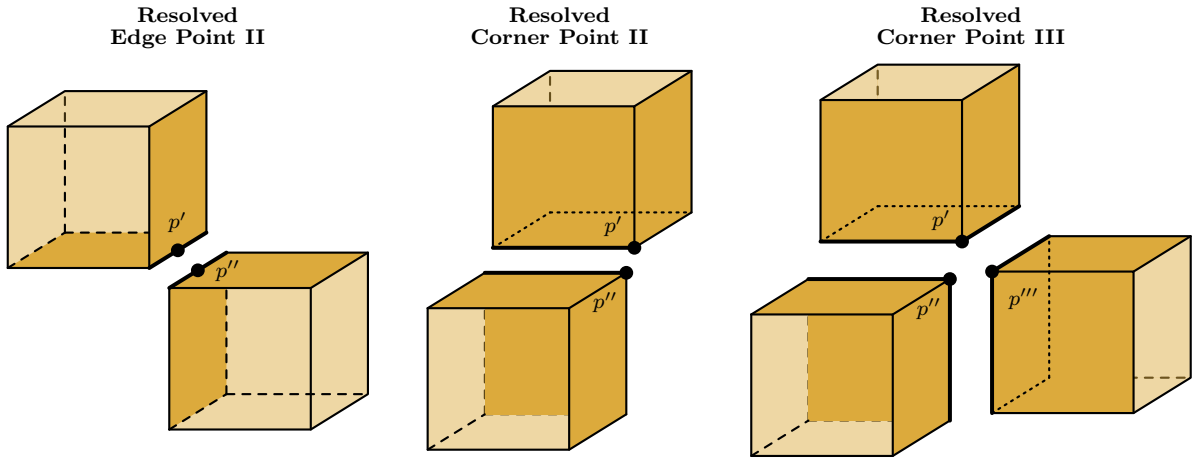


 Ralph \cite{Fox1957} studied the categories of spreads and complete spreads.\footnote{Consult the book by \cite{BungeFunk2006} for a more complete categorical understanding of the Fox completion.} \emph{A spread} \(g : X \to Y\) is a map of locally connected \(T_1\)-spaces such that the components of the inverse images of the open sets of \(Y\) form a basis of \(X\). A morphism of spreads is a commutative square. Fox proves that if \(X\) is spread over \(Y\) and \(Y\) is regular, then \(X\) is regular. Since \(X\) and \(Y\) are in particular Kolmogorov spaces by definition, \(Y\) being regular implies that both \(X\) and \(Y\) are Hausdorff. 

To a complete spread, \citeauthor{Fox1957} provides the following refined definition. A spread \(g: X \to Y\) will be said to be \emph{complete} if for every point \(y\) of \(Y\) the following condition is satisfied: If to every open neighborhood \(V\) of \(y\) there is selected a connected component \(U\) of \(g^{-1}\left(V\right)\) in such a way that \(U_1 \subset U_2\) whenever \(V_1 \subset V_2\), then \(\cap_V U\) is non-vacuous (and is therefore a point). A spread \(\wh{g} : \wh{X} \to Y\) is called \emph{a completion of the spread} \(g : X \to Y\) if \(\wh{g}\) is complete and there exists an embedding \(\iota\) of \(X\) into \(\wh{X}\) such that the image \(\iota\left(X\right)\) is dense and locally connected in \(\wh{X}\)\footnote{A subspace \(X\) is \emph{locally connected in \(Y\)} if there exists a basis \(\mathscr{B}\) of \(Y\) such that \(U \cap X\) is connected in \(X\) for every \(U\) in \(\mathscr{B}\).} and the resulting triangle diagram is commutative.

For any spread \(g : X \to Y\), \citeauthor{Fox1957}  constructs a completion \(\wh{g} : \wh{X} \to Y\). Let \(y\) be any point of \(Y\). A point \(x\) of \(\wh{X}\) in the fiber of \(\wh{g}\) above \(y\) is a function that assigns to each open neighborhood \(V\) of \(y\) a connected component \(x\left(V\right)\) of \(g^{-1}\left(V\right)\) in such a way that \(x\left(V_1\right)\) is contained in \(x\left(V_2\right)\) whenever \(V_1\) is contained in \(V_2\). This defines simultaneously the set \(\wh{X}\) and the function \(\wh{g}\). Given an open set \(V\) of \(Y\) and any connected component \(U\) of \(g^{-1}\left(V\right)\), define the open set \(U/V\) of \(\wh{X}\) to be the set of those points of \(\wh{X}\) for which \(x\left(V\right) = U\).

This construction gives a covariant functor going from the category of spreads to the category of complete spreads, which we call \emph{the Fox completion}. Restricted to spreads over a fixed space \(Y\), with morphisms covering the identity on \(Y\), this functor is left adjoint to the forgetful functor from complete spreads into spreads that just forgets to say the adjective `complete.' One need not understand the Fox completion in this categorical way to understand the properties used here. We use the adjunction only to note that the Fox completion is unique up to unique isomorphism, which justifies the written hat notation.

 Any filling \(F\) of \(S\) by  \(M\) is a complete spread over the building \(B_S\) of \(S\). The interior \(M^\circ\) of \(M\) is a dense and locally connected subspace of \(M\), making \(F\) the Fox completion of the restriction spread from \(M^\circ\) to \(B_S\). The kit  \(K_F\) assembled by \(F\) is included in \(M^\circ\) as a dense subspace and covers the identity over \(B_S\), so functoriality of the Fox completion gives \emph{the gluing \(g_F\) of \(\wh{K_F}\)}. This is the unique map \(g_F : \wh{K_F} \to M\) such that a dense and locally connected inclusion \(K_F \hookrightarrow \wh{K_F}\) forms a commutative square with \(F\) and \(g_F\).\footnote{Use this dense and locally connected inclusion to define \emph{a piece} and \emph{a part} of a complete kit in line with \S 3.} 


\begin{lem} 	\label{lem:assemblage}
	Let \(F\) be any filling of \(S\). Then the complete kit \(\wh{K_F}\) assembled by \(F\) is a compact space whose pieces are each an orientable and smooth 3-manifold with corners and whose parts are each a finite covering space of degree \(w\left(R\right)\) over the Fox completion of a room \(R\) of \(S\). Furthermore, the gluing \(g_F\) of \(\wh{K_F}\) is a surjective immersion such that its restriction to the interior of \(\wh{K_F}\) is a diffeomorphism with dense image equal to \(K_F\); hence, \(g_F\) is a quotient map that only identifies points along the boundary of \(\wh{K_F}\).
\end{lem}
\begin{proof}
Compactness of \(\wh{K_F}\) follows by an argument of \citeauthor{Fox1957} since the domain of \(F\) is assumed compact. The building of \(S\) is regular and the fibers of \(F\) are finite, so \(\wh{K_F}\) is Hausdorff by earlier discussion. That \(\wh{K_F}\) is second-countable follows from the finite degree covering space picture.  The gluing \(g_F\) is a diffeomorphism with dense image when restricted to the interior of \(\wh{K_F}\) since \(K_F\) is included as the interior of \(\wh{K_F}\), and the inclusion forms a commutative square with \(F\) and \(g_F\). Any part of \(\wh{K_F}\) is a covering space over the Fox completion of some room because \(\wh{K_F}\) is Hausdorff and the Fox completion yields a proper local homeomorphism with degree following from \cref{lem:wrapping numbers}. The Fox completion of any room admits the structure of an orientable and smooth 3-manifold with corners, which one may transport via the covering map obtained to bestow that structure upon \(\wh{K_F}\). The immersion property follows because two out of three sides of the commutative triangle to which \(g_F\) belongs are immersions.
\end{proof}


Generated is the picture that the domain of each filling is some assemblage of finitely many compact covering spaces, adhered together along their boundaries. The next few statements show that this gluing depends only on the local model and wrapping numbers around a given point. See \cref{fig:local pictures in preimage,fig:building} for definitions and portrayals of the local models referenced.

\begin{lem} 	\label{lem:preimage number}
	Let \(F\) be any filling of \(S\) and \(g_F\) the gluing of the complete kit \(\wh{K_F}\) assembled by \(F\). Then the function \(\# g_F^{-1}\) sending a point \(p\) in the domain of \(F\) to the cardinality of its fiber \(g_F^{-1}\left(p\right)\) equals
	\begin{align*}
		\# g_F^{-1}\left( p\right) = \begin{cases} 
			1 & \text{if } p \text{ is a preimage room point or a preimage boundary single point.} \\
			2 & \text{if } p \text{ is a preimage interior single point or a preimage boundary double point.} \\
			4 & \text{if } p \text{ is a preimage interior double point or a preimage boundary triple point.}\\
			8 & \text{if } p \text{ is a preimage interior triple point}.
		\end{cases}
	\end{align*}
\end{lem}
\begin{proof}
	Any point \(p\) in the domain of \(F\) has a local model neighborhood \(U\) that is homeomorphic to some model from \cref{fig:local pictures in preimage}. The intersection \(U \cap K_F\) has a number of connected components equal to the claimed value of \(\# g_F^{-1}\left( p\right)\). Any point in \(g_F^{-1}\left(p\right)\) is a function that assigns to \(U\) one of the connected components of \(U \cap K_F\) in a way that satisfies a nesting property. This nesting property implies that \(\# g_F^{-1}\left( p\right)\) is equal to the claimed value.
\end{proof}

\begin{lem}	\label{lem:filling to building}
	Let \(F\) be any filling of \(S\) and \(p\) any point in the building \(B_S\) of \(S\). If \(x\) is the minimum wrapping number of all chambers containing \(p\), then the fiber \(F^{-1}\left(p\right)\) above \(p\) consists of 
	\begin{align*}
		F^{-1}\left(p\right) = \begin{cases} 
			x \text{ preimage room points} & \text{if } p \text{ is a building room point.} \\
			x \text{ preimage interior plus } 1 \text{ preimage boundary single point} & \text{if } p \text{ is a building single point.} \\
			x \text{ preimage interior plus } 2 \text{ preimage boundary double points} & \text{if } p \text{ is a building double point.}\\
			x \text{ preimage interior plus } 3 \text{ preimage boundary triple points} & \text{if } p \text{ is a building triple point.}
		\end{cases}
	\end{align*}
\end{lem}
\begin{proof}
	The statement for room points follows because \(F\) is a covering map of degree equal to the only wrapping number nearby. In the remaining cases, look at the chamber which contributes the unique maximum wrapping number of all chambers containing \(p\) to see that the cardinality of \(F^{-1}\left(p\right)\) equals \(x + k\), where \(k\) equals one, two, or three depending on whether \(p\) is a single, double, or triple point respectively.  Since the local models in \cref{fig:local pictures in preimage} must be transported into the local models in \cref{fig:building} via any filling, and since \(F\) when restricted to the boundary 2-sphere is injective along its single points, exactly \(x\) points of \(F^{-1}\left(p\right)\) are interior points. The remaining \(k\) points are determined by the image of the boundary 2-sphere.
\end{proof}

\begin{cor}	\label{cor:count}
	Let \(F\) be any filling of \(S\) and \(g_F\) the gluing of the complete kit assembled by \(F\). Let \(F_g\) equal the composition  \(F \circ g_F\), regarded as a map to the building \(B_S\) of \(S\). If \(x\) is the minimum wrapping number of all chambers containing a point \(p\) in \(B_S\), then the function \(\#F_g^{-1}\) sending \(p\) to the cardinality of its fiber \(F_g^{-1}\left(p\right)\) equals	
	\begin{align*}
		\# F_g^{-1}(p) = \begin{cases} 
			x & \text{if } p \text{ is a building room point.} \\
			2x + 1 & \text{if } p \text{ is a building single point.} \\
			4x + 4 & \text{if } p \text{ is a building double point.}\\
			8x + 12 & \text{if } p \text{ is a building triple point}.
		\end{cases}
	\end{align*}
\end{cor}

The Fox completion thus distinguishes, with multiplicity given by the wrapping number, the local connected components of the rooms around any point. Define \emph{a local room of \(S\)} to be a connected component of the intersection of a sufficiently small ball in 3-space with the complement of \(S\) inside its own building. The local rooms at each point \(p\) on \(S\) are partially ordered by inclusion.  Define \emph{the set of local rooms at \(p\)} to be the finite set \(\mathfrak{R}_p\) of maximal chains of this partial ordering.  Given a filling \(F\) of \(S\), with \(F_g\) as just defined in \cref{cor:count}, there is, for each local room \(r\) in \(\mathfrak{R}_p\), a canonical subset \(r^F\) of the fiber \(F_g^{-1}\left(p\right)\) above \(p\).  Call this canonical subset \(r^F\) \emph{the fiber above the local room \(r\)}, and define \(r^F\) to be the intersection of the closure of the inverse image of any representative of \(r\) with \(F_g^{-1}\left(p\right)\). The collection of all fibers above the local rooms at \(p\) forms a partition of \(F_g^{-1}\left(p\right)\). If a local room \(r\) has wrapping number \(w\left(r\right)\), then the cardinality \(\#r^F\) of the fiber above the local room is equal to \(w\left(r\right)\). Two local rooms at \(p\) are said to be \emph{adjacent} if the closures of any of their representatives meet along a wall of \(S\).  Adjacency endows \(\mathfrak{R}_p\) with a relation \(\rightarrowtriangle\) where \(r \rightarrowtriangle r'\) if \(r\) is adjacent to \(r'\) and \( w\left(r\right) + 1 = w\left(r'\right)\).

\begin{prop}	\label{prop:injection}
	Let \(F\) be any filling of \(S\) by \(M\) and \(g_F\) the gluing of the complete kit assembled by \(F\). If two local rooms \(r\) and \(r'\) at a point \(p\) on \(S\) satisfy \(r \rightarrowtriangle r'\),  then there is a unique injection \(r^F \hookrightarrow r'^F\) between the fibers above the local rooms that belongs to a commutative triangle of injections.
	\[\begin{tikzcd}
	{r^F} &&&& {r'^F} \\
	\\
	\\
	&& M
	\arrow[hook, from=1-1, to=1-5]
	\arrow["{g_F}"', hook, from=1-1, to=4-3]
	\arrow["{g_F}", hook, from=1-5, to=4-3]
\end{tikzcd}\]
\end{prop} 
\begin{proof}
	To prove existence, one combines the picture of the Fox completion simply distinguishing the local rooms around any point on \(S\) with the preceding lemmas, while uniqueness follows categorically.
\end{proof}

\begin{defn}[The gluing category of a kit] 	\label{def:glue category}
	Let \(K\) be  a kit of any \(S\) and \(p\) a point on \(S\). Define \emph{the gluing category \(\mathscr{G}_p\left(K\right)\) of \(K\) at \(p\)} to be the category whose objects are the analogously defined \emph{fibers above the local rooms at \(p\) in the complete kit of \(K\)}. The morphisms are injections between the fibers above the local rooms that cover the natural wrapping number relation \(\rightarrowtriangle\)  on the local rooms at \(p\). In general \(\rightarrowtriangle\) is given by the inequality between the absolute values of the wrapping numbers of adjacent local rooms.
\end{defn}

Given any filling \(F\) of \(S\), the gluing geometry around all room points is understood since that is where the gluing \(g_F\) of the complete kit \(\wh{K_F}\) assembled by \(F\) is a diffeomorphism. Around each point \(p\) on \(S\) there is the finite set \(\mathfrak{R}_p\) of local rooms, and above  \(\mathfrak{R}_p\), in \(\wh{K_F}\), the finite sets of fibers above those local rooms. There are unique injections forming commutative diagrams between these fibers above the local rooms that cover the natural wrapping number relation \(\rightarrowtriangle\) on \(\mathfrak{R}_p\) below.  \cref{fig:instruction manual picture} portrays the shape that these diagrams make, as well as the cardinalities of the fibers above the local rooms. The previous proposition gives, for each point \(p\) on \(S\), a diagram \(\mathfrak{D}_F\left(p\right)\) in \(\mathscr{G}_p\left(K_F\right)\), which we call \emph{the assembly of \(F\) at \(p\)}.

\begin{thm}	\label{thm:push-out}
	Let \(F\) be any filling of any generically immersed 2-sphere \(S\) in 3-space  and \(\mathscr{G}_p\) the gluing category of the kit \(K_F\) assembled by \(F\) at a point \(p\) on \(S\). Then the assembly  of \(F\) at \(p\) is
	\begin{align*}
		\mathfrak{D}_F(p) = \begin{cases} 
			\text{a morphism in } \mathscr{G}_p & \text{if } p \text{ is a single point.} \\
			\text{a push-out square in } \mathscr{G}_p & \text{if } p \text{ is a double point.} \\
			\text{a push-out cube in } \mathscr{G}_p & \text{if } p \text{ is a triple point.}
		\end{cases}
	\end{align*}
\end{thm}
\begin{proof}
The previous proposition handles commutativity as well as the case when \(p\) is a single point of \(S\). By counting, a commuting square in \(\mathscr{G}_p\) is a push-out square exactly when the union of the images of its two terminal arrows contains the fiber above the terminal local room. If the terminal arrows were not surjective in some commuting square after a choice of adhesion, then there would exist a point in the fiber above the terminal local room that is on a corner of the assembled complete kit \(\wh{K_F}\). However, no other piece of \(\wh{K_F}\) is adhered there at that corner, thus creating a contradictory corner in \(M\). To conclude, a push-out cube in \(\mathscr{G}_p\) is characterized by all its faces being push-out squares.
\end{proof}

\section{Instructions for How to Assemble a Kit into a Smooth Manifold with Smooth Boundary}

The concept of local rooms in a local model neighborhood was just introduced. Later in this section, we introduce the concept of local walls. Geometrically, the coboundary of a local wall is a pair of adjacent local rooms. The constructions in this section are intrinsically defined from any generically immersed 2-sphere  \(S\) in 3-space. For the problem at hand, \emph{an instruction manual \(E_K\)} is defined here only when \(S\) is positive. For further research purposes, \emph{the assembly sheaf \(\mathscr{A}_S\)} is defined here for any \(S\). 

\begin{defn}[Instructions at a point] 	\label{def:allpushouts}
	Let \(K\) be a kit of any positive \(S\) and \(\mathscr{G}_p\) the gluing category of \(K\) at a point \(p\) on \(S\). Define \emph{the instructions of \(K\) at} \(p\)  to be the set \(E_K(p)\) consisting of
	\begin{align*}
		E_K(p) = \begin{cases} 
			\text{all morphisms in } \mathscr{G}_p & \text{if } p \text{ is a single point.} \\
			\text{all push-out squares in } \mathscr{G}_p & \text{if } p \text{ is a double point.} \\
			\text{all push-out cubes in } \mathscr{G}_p & \text{if } p \text{ is a triple point.}
		\end{cases}
	\end{align*}
	We call an element of the instructions of \(K\) at \(p\) \emph{an adhesion of \(K\) at \(p\)}.
\end{defn}

\begin{lem} 	\label{lem:stalknumbers}
	Let \(K\) be any kit of any positive \(S\). If \(x\) is the minimum wrapping number of all chambers containing a point \(p\) on \(S\), then the function \(\#E_K\) sending \(p\) to the cardinality of the instructions of \(K\) at \(p\) equals
	\begin{align*}
		\# E_K\left(p\right)  = \begin{cases} 
			\left(x+1\right)! & \text{if } p \text{ is a single point.} \\
			\left(\left(x+1\right)!\right)^2\left(x+2\right)! & \text{if } p \text{ is a double point.} \\
			\left(\left(x+1\right)!\right)^3\left(\left(x+2\right)!\right)^3\left(x+3\right)! & \text{if } p \text{ is a triple point.}
		\end{cases}
	\end{align*}
\end{lem}
\begin{proof}
	For any positive integer \(x\), let \(\mathcal{S}_x\) denote the symmetric group on \(x\) elements. Observe that if \(p\) is a single point, then \(E_K\left(p\right)\) is a torsor for \(\mathcal{S}_{x+1}\); if \(p\) is a double point, then \(E_K\left(p\right)\) is a torsor for \(\mathcal{S}_{x+1} \times \mathcal{S}_{x+1} \times \mathcal{S}_{x+2}\); and if \(p\) is a triple point, then \(E_K\left(p\right)\) is a torsor for \(\mathcal{S}_{x+1} \times \mathcal{S}_{x+1} \times \mathcal{S}_{x+1} \times \mathcal{S}_{x+2} \times \mathcal{S}_{x+2} \times \mathcal{S}_{x+2} \times \mathcal{S}_{x+3}\).
\end{proof}


Define \(E_K = \cup_p E_K\left(p\right)\) to be \emph{the instruction manual of \(K\)}, which by definition comes equipped with a canonical surjection \(\phi_K : E_K \to S\), which we call \emph{the instructions for \(K\)}. The topology on \(E_K\) can be understood geometrically via \cref{fig:instruction manual picture}. To see small open neighborhoods of an adhesion in \(E_K\), notice the six half-edges coming into any triple point intersecting the center of each of the six faces of a little push-out cube that links the triple point.  Push each face of the cube diagram off to the half-edge which intersects its center.  Do similarly for each edge injection arrow piercing the single locus of \(S\) transversely. This picture of a basic open set makes it immediate that \(\phi_K\) is a local homeomorphism onto \(S\). Before the topology of \(E_K\) is defined set-theoretically, we define the maximal atlas of local models for \(S\).

 \emph{A local model for \(S\)} is a smooth coordinate chart of 3-space \(\left(U, \varphi\right)\) such that the restriction \(\restr{\varphi}{U \cap S}\) to the intersection is a homeomorphism onto some model with defining equation from \cref{fig:local pictures}.  \emph{An atlas of local models for \(S\)} is a collection of local models for \(S\) whose union of \emph{local model neighborhoods} contain \(S\). \emph{A maximal atlas of local models for \(S\)} is an atlas of local models for \(S\) that is not properly contained within another. Since an atlas of local models for \(S\) is by definition a subset of the smooth structure on 3-space, there is a unique \emph{maximum model atlas for \(S\)}, which we denote by \(\mathfrak{M}_S\).
 
 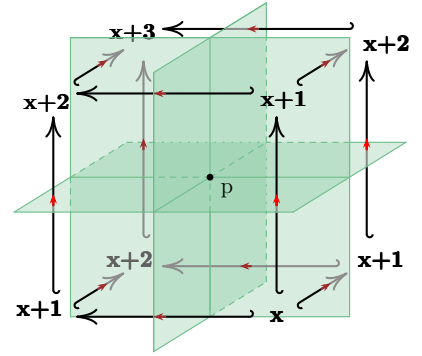
\begin{figure}[h!]
 	{\centering
 		\begin{minipage}{\textwidth}
 			\centering
 			
 			\begin{tabular}{@{}c@{\hspace{0.3in}}c@{\hspace{0.3in}}c@{}}
 				\textbf{Single Point Instructions} & \textbf{Double Point Instructions} & \textbf{Triple Point Instructions}
 				\\[1.5mm]
 				\(\vcenter{\hbox{\scalebox{\stalkScale}{%
 							\begin{tikzpicture}[stalkproj]
 								
 								\coordinate (ymzm) at (0.707107,-0.707107,-1.000000);
 								\coordinate (ypzm) at (-0.707107,0.707107,-1.000000);
 								\coordinate (ypzp) at (-0.707107,0.707107,1.000000);
 								\coordinate (ymzp) at (0.707107,-0.707107,1.000000);
 								
 								\fill[stalkface] (ymzm) -- (ypzm) -- (ypzp) -- (ymzp) -- cycle;
 								\draw[stalkedge] (ypzm) -- (ypzp);
 								
 								\coordinate (Npt) at (-0.509117,-0.509117,0.000000);
 								\coordinate (Np1pt) at (0.601041,0.601041,0.000000);
 								\coordinate (normalpt) at (0.070711,0.070711,0.000000);
 								\coordinate (Opt) at (0.000000,0.000000,0.000000);
 								\coordinate (xlab) at (-0.466690,-0.466690,0.000000);
 								\coordinate (xp1lab) at (0.662913,0.662913,0.000000);
 								\coordinate (arrowaux1) at (0.000000,0.000000,0.000000);
 								\coordinate (arrowaux2) at (0.300279,0.300279,0.000000);
 								
 								\begin{scope}[shift={(canvas cs:x=-1.41346887cm,y=-0.25274103cm)}, rotate=10.0975]
 									\draw[segvis] (-2.4pt,0pt) arc[start angle=180, end angle=90, radius=2.4pt];
 								\end{scope}
 								\draw[segvis, shorten <=11pt] (Npt) -- (arrowaux1); 
 								\draw[segocc] (arrowaux1) -- (arrowaux2); 
 								\draw[bigtip, segvis, shorten >=9pt] (arrowaux2) -- (Np1pt); 
 								
 								\boldstamp[rotate=10.1]{xlab}{x}{black}
 								\boldstamp[rotate=10.1]{xp1lab}{x+1}{black}
 								
 								\draw[stalknormal, shadowredbright] (Opt) -- (normalpt); 
 								
 								\draw[stalkedge] (ypzp) -- (ymzp) -- (ymzm) -- (ypzm);
 								
 								\fill[black] (Opt) circle (1.8pt);
 								\node[font=\Large, black, yshift=-14pt] at (Opt) {p};
 							\end{tikzpicture}
 				}}}\)
 				&
 				\(\vcenter{\hbox{\scalebox{\stalkScale}{%
 							\begin{tikzpicture}[stalkprojskew]
 								
 								
 								\coordinate (xmym) at (-1.000000,-1.000000,0.000000);
 								\coordinate (xmyp) at (-1.000000,1.000000,0.000000);
 								\coordinate (ymzm) at (0.000000,-1.000000,-1.000000);
 								\coordinate (ym) at (0.000000,-1.000000,0.000000);
 								\coordinate (ymzp) at (0.000000,-1.000000,1.000000);
 								\coordinate (ypzm) at (0.000000,1.000000,-1.000000);
 								\coordinate (yp) at (0.000000,1.000000,0.000000);
 								\coordinate (ypzp) at (0.000000,1.000000,1.000000);
 								\coordinate (xpym) at (1.000000,-1.000000,0.000000);
 								\coordinate (xpyp) at (1.000000,1.000000,0.000000);
 								\coordinate (O) at (0.000000,0.000000,0.000000);
 								\coordinate (aux1) at (0.000000,1.000000,-0.500000);
 								\coordinate (aux2) at (-0.894231,1.000000,0.000000);
 								
 								\coordinate (xC) at (-0.750000,0.000000,-0.750000);
 								\coordinate (x1aC) at (0.750000,0.000000,-0.750000);
 								\coordinate (x1bC) at (-0.750000,0.000000,0.750000);
 								\coordinate (x2C) at (0.750000,0.000000,0.750000);
 								
 								\fill[stalkface] (ym) -- (ymzp) -- (ypzp) -- (yp) -- cycle; 
 								\fill[stalkface] (ym) -- (ymzm) -- (ypzm) -- (yp) -- cycle; 
 								\fill[stalkface] (ym) -- (xpym) -- (xpyp) -- (yp) -- cycle; 
 								\fill[stalkface] (ym) -- (xmym) -- (xmyp) -- (yp) -- cycle; 
 								
 								\draw[stalkedge] (ymzp) -- (ypzp);
 								\draw[stalkedge] (ym) -- (yp);
 								\draw[stalkedge] (ymzm) -- (ypzm);
 								\draw[stalkedge] (ypzm) -- (aux1);
 								\draw[stalkedge] (xpym) -- (ym);
 								\draw[stalkedge] (xpyp) -- (yp);
 								\draw[stalkedge] (xmym) -- (xmyp);
 								\draw[stalkedge] (xmyp) -- (aux2);
 								\draw[stalkedge, dashed] (aux1) -- (yp);
 								\draw[stalkedge, dashed] (aux2) -- (yp);
 								
 								\fill[black] (O) circle (1.8pt);
 								\node[font=\Large, black, xshift=12pt, yshift=-5pt] at (O) {p};
 								
 								\coordinate (parrowaux1) at (-0.447116,0.000000,-0.750000);
 								\coordinate (parrowaux2) at (-0.000000,0.000000,-0.750000);
 								\draw[hookstart, segvis, shorten <=9pt] (xC) -- (parrowaux1); 
 								\draw[segocc] (parrowaux1) -- (parrowaux2); 
 								\draw[bigtip, segvis, shorten >=9pt] (parrowaux2) -- (x1aC); 
 								\coordinate (parrowaux3) at (-0.750000,0.000000,-0.250000);
 								\coordinate (parrowaux4) at (-0.750000,0.000000,-0.000000);
 								\draw[hookstart, segvis, shorten <=9pt] (xC) -- (parrowaux3); 
 								\draw[segocc] (parrowaux3) -- (parrowaux4); 
 								\draw[bigtip, segvis, shorten >=9pt] (parrowaux4) -- (x1bC); 
 								\coordinate (parrowaux5) at (0.750000,0.000000,-0.139786);
 								\coordinate (parrowaux6) at (0.750000,0.000000,-0.000000);
 								\draw[hookstart, segvis, shorten <=9pt] (x1aC) -- (parrowaux5); 
 								\draw[segocc] (parrowaux5) -- (parrowaux6); 
 								\draw[bigtip, segvis, shorten >=9pt] (parrowaux6) -- (x2C); 
 								\coordinate (parrowaux7) at (-0.447116,0.000000,0.750000);
 								\coordinate (parrowaux8) at (-0.000000,0.000000,0.750000);
 								\draw[hookstart, segvis, shorten <=9pt] (x1bC) -- (parrowaux7); 
 								\draw[segocc] (parrowaux7) -- (parrowaux8); 
 								\draw[bigtip, segvis, shorten >=9pt] (parrowaux8) -- (x2C); 
 								
 								\draw[stalkedge] (ym) -- (ymzp);
 								
 								\draw[stalkedge] (yp) -- (ypzp);
 								\draw[stalkedge] (ymzm) -- (ym);
 								\draw[stalkedge] (xmym) -- (ym);
 								\draw[stalkedge] (xpym) -- (xpyp);
 								
 								\coordinate (pmid_B) at (0.000000,0.000000,-0.750000);
 								\coordinate (pred_B) at (0.100000,0.000000,-0.750000);
 								\draw[stalknormal, red] (pmid_B) -- (pred_B);
 								\coordinate (pmid_L) at (-0.750000,0.000000,0.000000);
 								\coordinate (pred_L) at (-0.750000,0.000000,0.100000);
 								\draw[stalknormal, red] (pmid_L) -- (pred_L);
 								\coordinate (pmid_R) at (0.750000,0.000000,0.000000);
 								\coordinate (pred_R) at (0.750000,0.000000,0.100000);
 								\draw[stalknormal, red] (pmid_R) -- (pred_R);
 								\coordinate (pmid_T) at (0.000000,0.000000,0.750000);
 								\coordinate (pred_T) at (0.100000,0.000000,0.750000);
 								\draw[stalknormal, red] (pmid_T) -- (pred_T);
 								
 								\coordinate (lab_x) at (-0.750000,0.000000,-0.750000);
 								\boldstamp{lab_x}{x}{black}
 								\coordinate (lab_x+1a) at (0.850000,0.000000,-0.750000);
 								\boldstamp{lab_x+1a}{x+1}{black}
 								\coordinate (lab_x+1b) at (-0.850000,0.000000,0.750000);
 								\boldstamp{lab_x+1b}{x+1}{black}
 								\coordinate (lab_x+2) at (0.850000,0.000000,0.750000);
 								\boldstamp{lab_x+2}{x+2}{black}
 							\end{tikzpicture}
 				}}}\)
 				&
 				\(\vcenter{\hbox{\scalebox{\stalkScale}{%
 							\begin{tikzpicture}[stalkproj]
 								
 								
 								\coordinate (XmYm) at (-1.000000,-1.000000,0.000000);
 								\coordinate (XmZm) at (-1.000000,0.000000,-1.000000);
 								\coordinate (Xm) at (-1.000000,0.000000,0.000000);
 								\coordinate (XmZp) at (-1.000000,0.000000,1.000000);
 								\coordinate (XmYp) at (-1.000000,1.000000,0.000000);
 								\coordinate (YmZm) at (0.000000,-1.000000,-1.000000);
 								\coordinate (Ym) at (0.000000,-1.000000,0.000000);
 								\coordinate (YmZp) at (0.000000,-1.000000,1.000000);
 								\coordinate (Zm) at (0.000000,0.000000,-1.000000);
 								\coordinate (O) at (0.000000,0.000000,0.000000);
 								\coordinate (Zp) at (0.000000,0.000000,1.000000);
 								\coordinate (YpZm) at (0.000000,1.000000,-1.000000);
 								\coordinate (Yp) at (0.000000,1.000000,0.000000);
 								\coordinate (YpZp) at (0.000000,1.000000,1.000000);
 								\coordinate (XpYm) at (1.000000,-1.000000,0.000000);
 								\coordinate (XpZm) at (1.000000,0.000000,-1.000000);
 								\coordinate (Xp) at (1.000000,0.000000,0.000000);
 								\coordinate (XpZp) at (1.000000,0.000000,1.000000);
 								\coordinate (XpYp) at (1.000000,1.000000,0.000000);
 								\coordinate (O) at (0.000000,0.000000,0.000000);
 								\coordinate (aux1) at (0.000000,1.000000,0.750001);
 								\coordinate (aux2) at (0.000000,1.000000,-0.500000);
 								\coordinate (aux3) at (0.000000,0.000000,-0.250000);
 								\coordinate (aux4) at (-0.403847,0.000000,0.000000);
 								\coordinate (aux5) at (-1.000000,0.000000,-0.250000);
 								\coordinate (aux6) at (-0.403847,0.000000,-1.000000);
 								\coordinate (aux7) at (0.596155,1.000000,0.000000);
 								\coordinate (aux8) at (-0.807693,1.000000,0.000000);
 								
 								\coordinate (x_std) at (0.800000,-0.800000,-0.800000);
 								\coordinate (x1_X_std) at (-0.800000,-0.800000,-0.800000);
 								\coordinate (x_std) at (0.800000,-0.800000,-0.800000);
 								\coordinate (x1_Y_std) at (0.800000,0.800000,-0.800000);
 								\coordinate (x_std) at (0.800000,-0.800000,-0.800000);
 								\coordinate (x1_Z_std) at (0.800000,-0.800000,0.800000);
 								\coordinate (x1_X_std) at (-0.800000,-0.800000,-0.800000);
 								\coordinate (x2_XY_std) at (-0.800000,0.800000,-0.800000);
 								\coordinate (x1_X_std) at (-0.800000,-0.800000,-0.800000);
 								\coordinate (x2_XZ_std) at (-0.800000,-0.800000,0.800000);
 								\coordinate (x1_Y_sp) at (0.750000,0.850000,-0.850000);
 								\coordinate (x2_XY_sp) at (-0.850000,0.850000,-0.850000);
 								\coordinate (x1_Y_std) at (0.800000,0.800000,-0.800000);
 								\coordinate (x2_YZ_std) at (0.800000,0.800000,0.800000);
 								\coordinate (x1_Z_std) at (0.800000,-0.800000,0.800000);
 								\coordinate (x2_XZ_std) at (-0.800000,-0.800000,0.800000);
 								\coordinate (x1_Z_std) at (0.800000,-0.800000,0.800000);
 								\coordinate (x2_YZ_std) at (0.800000,0.800000,0.800000);
 								\coordinate (x2_XY_std) at (-0.800000,0.800000,-0.800000);
 								\coordinate (x3_std) at (-0.800000,0.800000,0.800000);
 								\coordinate (x2_XZ_std) at (-0.800000,-0.800000,0.800000);
 								\coordinate (x3_std) at (-0.800000,0.800000,0.800000);
 								\coordinate (x2_YZ_sp) at (0.850000,0.850000,0.850000);
 								\coordinate (x3_sp) at (-0.850000,0.850000,0.850000);
 								
 								\coordinate (lab_x) at (0.800000,-0.800000,-0.800000);
 								\coordinate (lab_x1_X) at (-0.900000,-0.800000,-0.750000);
 								\coordinate (lab_x1_Y) at (0.900000,0.800000,-0.800000);
 								\coordinate (lab_x1_Z) at (0.850000,-0.800000,0.750000);
 								\coordinate (lab_x2_XY) at (-0.900000,0.800000,-0.800000);
 								\coordinate (lab_x2_XZ) at (-0.850000,-0.800000,0.730000);
 								\coordinate (lab_x2_YZ) at (0.940000,0.800000,0.750000);
 								\coordinate (lab_x3) at (-0.880000,0.800000,0.830000);
 								
 								\fill[stalkface] (O) -- (Yp) -- (YpZp) -- (Zp) -- cycle; 
 								\fill[stalkface] (O) -- (Yp) -- (YpZm) -- (Zm) -- cycle; 
 								\fill[stalkface] (O) -- (Ym) -- (YmZp) -- (Zp) -- cycle; 
 								\fill[stalkface] (O) -- (Ym) -- (YmZm) -- (Zm) -- cycle; 
 								\fill[stalkface] (O) -- (Xp) -- (XpZp) -- (Zp) -- cycle; 
 								\fill[stalkface] (O) -- (Xp) -- (XpZm) -- (Zm) -- cycle; 
 								\fill[stalkface] (O) -- (Xm) -- (XmZp) -- (Zp) -- cycle; 
 								\fill[stalkface] (O) -- (Xm) -- (XmZm) -- (Zm) -- cycle; 
 								\fill[stalkface] (O) -- (Xp) -- (XpYp) -- (Yp) -- cycle; 
 								\fill[stalkface] (O) -- (Xp) -- (XpYm) -- (Ym) -- cycle; 
 								\fill[stalkface] (O) -- (Xm) -- (XmYp) -- (Yp) -- cycle; 
 								\fill[stalkface] (O) -- (Xm) -- (XmYm) -- (Ym) -- cycle; 
 								
 								\draw[stalkedge] (aux1) -- (YpZp);
 								\draw[stalkedge] (Zp) -- (YpZp);
 								\draw[stalkedge] (O) -- (Zp);
 								\draw[stalkedge] (Zm) -- (aux3);
 								\draw[stalkedge] (O) -- (Ym);
 								\draw[stalkedge] (Ym) -- (YmZp);
 								\draw[stalkedge] (Zp) -- (YmZp);
 								\draw[stalkedge] (YmZm) -- (Ym);
 								\draw[stalkedge] (Zm) -- (YmZm);
 								\draw[stalkedge] (Xp) -- (O);
 								\draw[stalkedge] (Xp) -- (XpZp);
 								\draw[stalkedge] (XpZp) -- (Zp);
 								\draw[stalkedge] (Xp) -- (XpZm);
 								\draw[stalkedge] (Zm) -- (XpZm);
 								\draw[stalkedge] (Xm) -- (XmZp);
 								\draw[stalkedge] (XmZp) -- (Zp);
 								\draw[stalkedge] (aux5) -- (XmZm);
 								\draw[stalkedge] (aux6) -- (XmZm);
 								\draw[stalkedge] (XpYp) -- (aux7);
 								\draw[stalkedge] (Xp) -- (XpYm);
 								\draw[stalkedge] (Xm) -- (XmYm);
 								\draw[stalkedge, dashed] (O) -- (Yp);
 								\draw[stalkedge, dashed] (Yp) -- (aux1);
 								\draw[stalkedge, dashed] (YpZm) -- (aux2);
 								\draw[stalkedge, dashed] (Zm) -- (YpZm);
 								\draw[stalkedge, dashed] (aux3) -- (O);
 								\draw[stalkedge, dashed] (O) -- (aux4);
 								\draw[stalkedge, dashed] (Xm) -- (aux5);
 								\draw[stalkedge, dashed] (Zm) -- (aux6);
 								\draw[stalkedge, dashed] (aux7) -- (Yp);
 								\draw[stalkedge, dashed] (Xm) -- (XmYp);
 								\draw[stalkedge, dashed] (XmYp) -- (aux8);
 								\draw[stalkedge, loosely dotted] (aux2) -- (Yp);
 								\draw[stalkedge, loosely dotted] (aux8) -- (Yp);
 								
 								\fill[black] (O) circle (1.8pt);
 								\node[font=\Large, black, xshift=9pt, yshift=-7pt] at (O) {p};
 								
 								\draw[{Hooks[left]}-{>[length=4mm,width=3mm]}, segocc3, shorten <=12pt, shorten >=12pt] (x1_Y_sp) -- (x2_XY_sp); 
 								\coordinate (cubearrowaux1) at (-0.800000,0.800000,-0.450000);
 								\coordinate (cubearrowaux2) at (-0.800000,0.800000,-0.000000);
 								\draw[hookstart, segocc3, shorten <=12pt] (x2_XY_std) -- (cubearrowaux1); 
 								\draw[segocc3] (cubearrowaux1) -- (cubearrowaux2); 
 								\draw[bigtip3, segocc3, shorten >=12pt] (cubearrowaux2) -- (x3_std); 
 								\coordinate (cubearrowaux3) at (-0.000000,-0.800000,-0.800000);
 								\coordinate (cubearrowaux4) at (-0.080770,-0.800000,-0.800000);
 								\draw[hookstart, segvis3, shorten <=12pt] (x_std) -- (cubearrowaux3); 
 								\draw[segocc3] (cubearrowaux3) -- (cubearrowaux4); 
 								\draw[bigtip3, segvis3, shorten >=12pt] (cubearrowaux4) -- (x1_X_std); 
 								\coordinate (cubearrowaux5) at (0.800000,0.000000,-0.800000);
 								\draw[hookstart, segvis3, shorten <=12pt] (x_std) -- (cubearrowaux5); 
 								\draw[bigtip3, segocc3, shorten >=12pt] (cubearrowaux5) -- (x1_Y_std); 
 								\coordinate (cubearrowaux6) at (0.800000,-0.800000,-0.050000);
 								\coordinate (cubearrowaux7) at (0.800000,-0.800000,-0.000000);
 								\draw[hookstart, segvis3, shorten <=12pt] (x_std) -- (cubearrowaux6); 
 								\draw[segocc3] (cubearrowaux6) -- (cubearrowaux7); 
 								\draw[bigtip3, segvis3, shorten >=12pt] (cubearrowaux7) -- (x1_Z_std); 
 								\coordinate (cubearrowaux8) at (-0.800000,0.000001,-0.800000);
 								\draw[hookstart, segvis3, shorten <=12pt] (x1_X_std) -- (cubearrowaux8); 
 								\draw[bigtip3, segocc3, shorten >=12pt] (cubearrowaux8) -- (x2_XY_std); 
 								\coordinate (cubearrowaux9) at (-0.800000,-0.800000,-0.050000);
 								\coordinate (cubearrowaux10) at (-0.800000,-0.800000,-0.000000);
 								\draw[hookstart, segvis3, shorten <=12pt] (x1_X_std) -- (cubearrowaux9); 
 								\draw[segocc3] (cubearrowaux9) -- (cubearrowaux10); 
 								\draw[bigtip3, segvis3, shorten >=12pt] (cubearrowaux10) -- (x2_XZ_std); 
 								\coordinate (cubearrowaux11) at (0.800000,0.800000,-0.123810);
 								\coordinate (cubearrowaux12) at (0.800000,0.800000,-0.000000);
 								\draw[hookstart, segvis3, shorten <=12pt] (x1_Y_std) -- (cubearrowaux11); 
 								\draw[segocc3] (cubearrowaux11) -- (cubearrowaux12); 
 								\draw[bigtip3, segvis3, shorten >=12pt] (cubearrowaux12) -- (x2_YZ_std); 
 								\coordinate (cubearrowaux13) at (-0.000000,-0.800000,0.800000);
 								\coordinate (cubearrowaux14) at (-0.080770,-0.800000,0.800000);
 								\draw[hookstart, segvis3, shorten <=12pt] (x1_Z_std) -- (cubearrowaux13); 
 								\draw[segocc3] (cubearrowaux13) -- (cubearrowaux14); 
 								\draw[bigtip3, segvis3, shorten >=12pt] (cubearrowaux14) -- (x2_XZ_std); 
 								\coordinate (cubearrowaux15) at (0.800000,0.000000,0.800000);
 								\draw[hookstart, segvis3, shorten <=12pt] (x1_Z_std) -- (cubearrowaux15); 
 								\draw[bigtip3, segocc3, shorten >=12pt] (cubearrowaux15) -- (x2_YZ_std); 
 								\coordinate (cubearrowaux16) at (-0.800000,0.000000,0.800000);
 								\draw[hookstart, segvis3, shorten <=12pt] (x2_XZ_std) -- (cubearrowaux16); 
 								\draw[bigtip3, segocc3, shorten >=12pt] (cubearrowaux16) -- (x3_std); 
 								\coordinate (cubearrowaux17) at (-0.000000,0.850000,0.850000);
 								\coordinate (cubearrowaux18) at (-0.242309,0.850000,0.850000);
 								\draw[hookstart, segvis3, shorten <=12pt] (x2_YZ_sp) -- (cubearrowaux17); 
 								\draw[segocc3] (cubearrowaux17) -- (cubearrowaux18); 
 								\draw[bigtip3, segvis3, shorten >=12pt] (cubearrowaux18) -- (x3_sp); 
 								
 								\coordinate (splitZm) at (0.000000,0.000000,-0.955000);
 								\draw[stalkedge] (splitZm) -- (aux3);
 								\draw[stalkedge] (Ym) -- (YmZp);
 								\draw[stalkedge] (Xp) -- (XpZp);
 								\draw[stalkedge] (Zp) -- (YpZp);
 								
 								\draw[stalkedge] (aux4) -- (Xm);
 								\draw[stalkedge] (Xp) -- (XpYp);
 								\draw[stalkedge] (XpYm) -- (Ym);
 								\draw[stalkedge] (XmYm) -- (Ym);
 								
 								\coordinate (mid_x_x1_X) at (0.000000,-0.800000,-0.800000);
 								\coordinate (red_x_x1_X) at (-0.080000,-0.800000,-0.800000);
 								\draw[stalknormal, shadowred] (mid_x_x1_X) -- (red_x_x1_X);
 								\coordinate (mid_x_x1_Y) at (0.800000,0.000000,-0.800000);
 								\coordinate (red_x_x1_Y) at (0.800000,0.160000,-0.800000);
 								\draw[stalknormal, shadowred] (mid_x_x1_Y) -- (red_x_x1_Y);
 								\coordinate (mid_x_x1_Z) at (0.800000,-0.800000,0.000000);
 								\coordinate (red_x_x1_Z) at (0.800000,-0.800000,0.080000);
 								\draw[stalknormal, red] (mid_x_x1_Z) -- (red_x_x1_Z);
 								\coordinate (mid_x1_X_x2_XY) at (-0.800000,0.000000,-0.800000);
 								\coordinate (red_x1_X_x2_XY) at (-0.800000,0.160000,-0.800000);
 								\draw[stalknormal, shadowred] (mid_x1_X_x2_XY) -- (red_x1_X_x2_XY);
 								\coordinate (mid_x1_X_x2_XZ) at (-0.800000,-0.800000,0.000000);
 								\coordinate (red_x1_X_x2_XZ) at (-0.800000,-0.800000,0.080000);
 								\draw[stalknormal, red] (mid_x1_X_x2_XZ) -- (red_x1_X_x2_XZ);
 								\coordinate (mid_x1_Y_x2_XY) at (-0.050000,0.850000,-0.850000);
 								\coordinate (red_x1_Y_x2_XY) at (-0.130000,0.850000,-0.850000);
 								\draw[stalknormal, shadowred] (mid_x1_Y_x2_XY) -- (red_x1_Y_x2_XY);
 								\coordinate (mid_x1_Y_x2_YZ) at (0.800000,0.800000,0.000000);
 								\coordinate (red_x1_Y_x2_YZ) at (0.800000,0.800000,0.080000);
 								\draw[stalknormal, red] (mid_x1_Y_x2_YZ) -- (red_x1_Y_x2_YZ);
 								\coordinate (mid_x1_Z_x2_XZ) at (0.000000,-0.800000,0.800000);
 								\coordinate (red_x1_Z_x2_XZ) at (-0.080000,-0.800000,0.800000);
 								\draw[stalknormal, shadowred] (mid_x1_Z_x2_XZ) -- (red_x1_Z_x2_XZ);
 								\coordinate (mid_x1_Z_x2_YZ) at (0.800000,0.000000,0.800000);
 								\coordinate (red_x1_Z_x2_YZ) at (0.800000,0.160000,0.800000);
 								\draw[stalknormal, shadowred] (mid_x1_Z_x2_YZ) -- (red_x1_Z_x2_YZ);
 								\coordinate (mid_x2_XY_x3) at (-0.800000,0.800000,0.000000);
 								\coordinate (red_x2_XY_x3) at (-0.800000,0.800000,0.080000);
 								\draw[stalknormal, shadowred] (mid_x2_XY_x3) -- (red_x2_XY_x3);
 								\coordinate (mid_x2_XZ_x3) at (-0.800000,0.000000,0.800000);
 								\coordinate (red_x2_XZ_x3) at (-0.800000,0.160000,0.800000);
 								\draw[stalknormal, shadowred] (mid_x2_XZ_x3) -- (red_x2_XZ_x3);
 								\coordinate (mid_x2_YZ_x3) at (0.000000,0.850000,0.850000);
 								\coordinate (red_x2_YZ_x3) at (-0.080000,0.850000,0.850000);
 								\draw[stalknormal, shadowred] (mid_x2_YZ_x3) -- (red_x2_YZ_x3);
 								
 								\boldstamp{lab_x}{x}{black}
 								\boldstamp{lab_x1_X}{x+1}{black}
 								\boldstamp{lab_x1_Y}{x+1}{black}
 								\boldstamp{lab_x1_Z}{x+1}{black}
 								\boldstamp{lab_x2_XY}{x+2}{black!62} 
 								\boldstamp{lab_x2_XZ}{x+2}{black}
 								\boldstamp{lab_x2_YZ}{x+2}{black}
 								\begin{scope}
 									\clip (lab_x3) ++(-25pt,-14pt) rectangle ++(50pt,11.79pt);
 									\boldstamp{lab_x3}{x+3}{black!62}
 								\end{scope}
 								\begin{scope}
 									\clip (lab_x3) ++(-25pt,-2.21pt) rectangle ++(50pt,18.21pt);
 									\boldstamp{lab_x3}{x+3}{black}
 								\end{scope}

 								\draw[stalkedge] (XmZp) -- (Zp);
 								\draw[stalkedge] (YmZm) -- (Ym);
 							\end{tikzpicture}
 				}}}\)
 			\end{tabular}
 			
 		\end{minipage}
 		\par}
 	\caption{An adhesion of a kit \(K\) at a point \(p\) on a positive \(S\). The symbol \(\bf{x+k}\) denotes the \(\left(x+k\right)\)-element fiber above the local room that \(\bf{x+k}\) is displayed within. Each arrow represents some choice of injection between these sets so as to describe an admissible adhesion of the points above \(p\) in the parts of the complete kit \(\wh{K}\). Any square diagram is a push-out diagram in the underlying gluing category \(\mathscr{G}_p\) of \(K\) at \(p\).}
 	\label{fig:instruction manual picture}
 \end{figure}
 
\begin{defn}[Topology of an instruction manual]  \label{def:instruction manual topology}
The topology of the instruction manual \(E_K\) for a kit \(K\) of any positive \(S\) is defined using the local model neighborhoods from the maximal model atlas \(\mathfrak{M}_S\) of \(S\). \emph{The topology on \(E_K\)} is the topology determined by the following definition of a basis of open sets. Let \(\wh{F_K} : \wh{K} \to B_S\) be the Fox completion of \(K\) spread over the building \(B_S\) of \(S\). For any local model neighborhood \(U\) from \(\mathfrak{M}_S\), let \(U^{F_K}\) denote the finite set of connected components of the inverse image \(\wh{F_K}^{-1}\left(U\right)\). For all points \(p\) inside the intersection \(U \cap S\), any adhesion \(a\) of \(K\) at \(p\) induces an equivalence relation \(\sim_a\) on \(U^{F_K}\) under which two components are related if \(a\) adheres them. For an adhesion \(a\) of \(K\) at \(p\) and a local model neighborhood \(U\) containing \(p\), define \emph{the \(U\)-neighborhood of \(a\) in \(E_K\)} to be the set
\begin{align*}
	\left\{b \in E_K : \phi_K\left(b\right) \in U \text{ and the quotient } U^{F_K} \to U^{F_K}/\sim_a \text{ factors through the quotient } U^{F_K} \to U^{F_K}/\sim_b\right\}.
\end{align*}
\end{defn}

\begin{thm} 	\label{thm:globalsection}
	Let \(K\) be a kit of any positive generically immersed 2-sphere \(S\) in 3-space. Then \(K\) is the kit assembled by a filling \(F\) of \(S\) if and only if there exists a global section of the instructions for \(K\).  
\end{thm}
\begin{proof}
	If \(K\) is the kit assembled by a filling \(F\) of \(S\), then the complete kit \(\wh{K_F}\) assembled by \(F\) is uniquely isomorphic to the complete kit \(\wh{K}\) of \(K\).  \cref{thm:push-out} shows that any filling induces an inclusion \(S \hookrightarrow E_K\) whose composite with the instructions \(\phi_K\) for \(K\) is the identity. This map is readily seen to be continuous and hence a global section of \(\phi_K\). For the converse, observe that a global section of \(\phi_K\) determines a quotient map from \(\wh{K}\) onto some space \(M\), which must now be shown to be a manifold. 
	
	That \(M\) is Hausdorff and second-countable follows readily from the geometry of the pieces of \(\wh{K}\) and by following the injections for how the parts are adhered together along their boundaries. That \(M\) is locally Euclidean follows by examining the link of any point \(p\) on \(S\), looking at its inverse image inside \(\wh{K}\), then seeing that a section of the instructions around \(p\) identifies the components of the inverse image of the link into spheres and domes. A map \(F\) from \(M\) to the building of \(S\) is obtained via the universal property of quotient maps. By construction, the pieces are adhered together in a way to imply that \(F\) is a local homeomorphism of a manifold \(M\) with boundary into 3-space. Transport the smooth structure on 3-space to \(M\) via this local homeomorphism of a manifold with boundary to tautologically prove that \(F\) is an immersion.
\end{proof}

\begin{lem} 	\label{lem:isomorphic torsors}
	Any covering map isomorphism \(h : K \to K'\) between kits of any positive  \(S\) induces a unique homeomorphism \(h_E : E_K \to E_{K'}\) of instruction manuals that belongs to a commutative triangle of local homeomorphisms.
	\[\begin{tikzcd}
		{E_K} &&&& {E_{K'}} \\
		\\
		\\
		&& S
		\arrow["{h_E}", from=1-1, to=1-5]
		\arrow["{\phi_K}"', from=1-1, to=4-3]
		\arrow["{\phi_{K'}}", from=1-5, to=4-3]
	\end{tikzcd}\]
\end{lem}
\begin{proof}
	Functoriality of the Fox completion gives a unique isomorphism between the complete kits \(\wh{K}\) and \(\wh{K'}\). All adhesions in \(E_K\) and \(E_{K'}\) are defined intrinsically in terms of the sets \(\wh{K}\) and \(\wh{K'}\), so the commuting homeomorphism \(h_E\) is produced canonically and in a way that respects the topologies.
\end{proof}

The problem of determining whether a positive \(S\) is fillable has been reduced to the problem of finding a global section of the instructions \(\phi_K\) for at least one kit \(K\) of \(S\).  The sheaf of finite groups over \(S\) that is constructed next is the algebraic tool with which we will solve this reduced problem. This sheaf of groups \(\mathscr{A}_S\) is defined so that the sheaf of sections of \(\phi_K\) is a torsor for \(\mathscr{A}_S\). Consequently, \(\phi_K\) admits a global section if and only if the class of this torsor is the preferred point. Call this obstructing cohomology class \emph{the assemblage class of \(K\)}.  Only finitely many assemblage classes of  \(S\) arise, because the set of inequivalent selections from the hardware store of \(S\) is a finite set, and the sheaves of sections of the instructions for equivalent selections of kit are proven to be isomorphic by \cref{lem:isomorphic torsors}. Extra structure described in \cref{prop:assemblysheaftorsor} implies that the first cohomology set of \(S\) with coefficients in its assembly sheaf is a finite set. 

\begin{defn}[The assembly sheaf] 	\label{def:assemblysheaf}
	Let \(S\) be the image of a generic codimension one immersion of a closed, connected, and orientable manifold into Euclidean space. \emph{The assembly sheaf of \(S\)} is the sheaf of groups  \(\mathscr{A}_S\) determined by the following definition on the maximal model atlas \(\mathfrak{M}_S\) of \(S\).  For a local model neighborhood \(U\)  from \(\mathfrak{M}_S\), define \emph{the set of local walls in \(U\)} to be the finite set \(\Sigma_U\) of connected components of the single points of \(S\) inside \(U\). Define an equivalence relation \(\vec{\sim}\) on \(\Sigma_U\) by declaring \(\sigma \ \vec{\sim} \ \sigma'\) whenever \emph{the inward-pointing direction  relative to the building which contains \(\sigma\)} and the inward-pointing direction relative to the building which contains \(\sigma'\), point into the same local room at \(U\). For any Cartesian product \(\prod G \) of the same group \(G\), let \(\dDelta\left(\prod G\right)\) denote the diagonal subgroup isomorphic to \(G\). For any positive integer \(x\), let \(\mathcal{S}_x\) denote the symmetric group on \(x\) elements. Define set maps \(\mathcal{S}_{\ccdot} : \Sigma_U \to \cup_x \left\{\mathcal{S}_x\right\}\) by \(\mathcal{S}_\sigma = \mathcal{S}_{x+1}\), where \(x\) is the minimum of the absolute values of the two wrapping numbers on either side of \(\sigma\).  Define \emph{the assembly group \(\mathscr{A}_S\left(U\right)\)} at  \(U\) to be the product
	\begin{align*}
		\mathscr{A}_S\left(U\right) =  
			 \Pprod{{\left[\sigma'\right] \in \Sigma_U/\vec{\sim}}} \DDelta \left(\prod_{\sigma \in \left[\sigma'\right]}  \mathcal{S}_{\sigma}\right).
	\end{align*}
	
	For any two local model neighborhoods \(U\) and \(V\) such that \(V\) is contained in \(U\), there is a canonical covariant inclusion of local walls. Each restriction map \(\rho_{UV} : \mathscr{A}_S\left(U\right) \to \mathscr{A}_S\left(V\right)\) equals the projection onto the factor of \(\mathscr{A}_S\left(U\right)\) in the image of the canonical inclusion. The axioms of cocycle, empty set, identity, and gluing hold readily via these canonical inclusions.
\end{defn}

 \begin{prop}  	\label{prop:assemblysheaftorsor}
 	The assembly sheaf \(\mathscr{A}_S\) of any positive generically immersed 2-sphere \(S\) in 3-space  is a constructible sheaf of finite products of symmetric groups such that the sheaf of sections of the instructions of any kit of \(S\) is a torsor for \(\mathscr{A}_S\).
 \end{prop}
 \begin{proof}
 	 The assembly group of any local model neighborhood \(U\) from the maximal model atlas \(\mathfrak{M}_S\) of \(S\) is isomorphic to 
 	 \begin{align*}
 	 	\mathscr{A}_S\left(U\right) \cong \begin{cases} 
 	 		\mathcal{S}_{x+1} & \text{if } U \text{ is a single point neighborhood,} \\
 	 		\mathcal{S}_{x+1} \times \mathcal{S}_{x+1} \times \mathcal{S}_{x+2} & \text{if } U \text{ is a double point neighborhood,} \\
 	 		\mathcal{S}_{x+1} \times \mathcal{S}_{x+1} \times \mathcal{S}_{x+1} \times \mathcal{S}_{x+2} \times \mathcal{S}_{x+2} \times \mathcal{S}_{x+2} \times \mathcal{S}_{x+3} & \text{if } U \text{ is a triple point neighborhood,}
 	 	\end{cases}
 	 \end{align*} 
 	where the value of \(x\) is equal to the minimum of the wrapping numbers of the local rooms at \(U\). Each group acts by permuting the images of injections of the assembly diagrams at the points of \(S\). If two local walls point towards the same local room, then deciding upon any one of those injections will uniquely determine the others by the universal property of the push-out. The assembly sheaf of \(S\) is \emph{constructible} in the sense that its value is locally constant along the strata of single, double, and triple points. In addition, there are homomorphisms from groups of less singular into more singular strata, given by the canonical inclusions.
 \end{proof}
 	
Deciding whether the image \(S\) of a generic immersion of the 2-sphere into 3-space is fillable has now been reduced to (i) a linear-time computation of the wrapping numbers, followed by (ii) a computation of the finitely many assemblage classes of \(S\) that live in the finite first cohomology set of \(S\) with coefficients in its assembly sheaf. To determine whether an extending domain is the 3-ball, use a 3-ball recognition algorithm.

 A generalized main theorem that characterizes compact fillings of a generic codimension one immersion, with connected image, of a closed manifold inside Euclidean space is left as an exercise for the reader. For disconnected domains, one desires to understand algebraically what is the meaning of iterating over different normal direction fields on the image to obtain some obstruction classes inside the first cohomology sets with different but related coefficient sheaves. Surgery is an approach to connect disconnected images. 
 
 The formula contained in the main theorem counting the number of inequivalent fillings now has a conceptual explanation. \cref{lem:isomorphic torsors} proves that the deck group of any kit acts freely on the set of global sections of its instructions. This action produces equivalences because one simply divides out at the end of a construction by the exact ambiguity incurred at the start upon choosing a representative from our finite moduli space of finite degree covering spaces. For the converse, any equivalence of fillings restricts to a diffeomorphism of their assembled kits, and any two such diffeomorphisms differ only by a deck transformation of some fixed reference kit.

\section{Gallery, References, and Tool Use}

The mathematics is a result of my thinking and lived experience used for a few years now towards understanding immersed 2-spheres inside 3-space. The figures are a result of working for a few weeks with A.I. tools. Such tools also gathered references for myself to write a proper history. The blue surfaces are generically immersed 2-spheres in 3-space. The red curves are the level sets of some Morse function. The yellow balls are double points. The orange balls are triple points. The green balls are saddles. 
	\begin{figure}[h]
		\centering
		\includegraphics[scale=0.16]{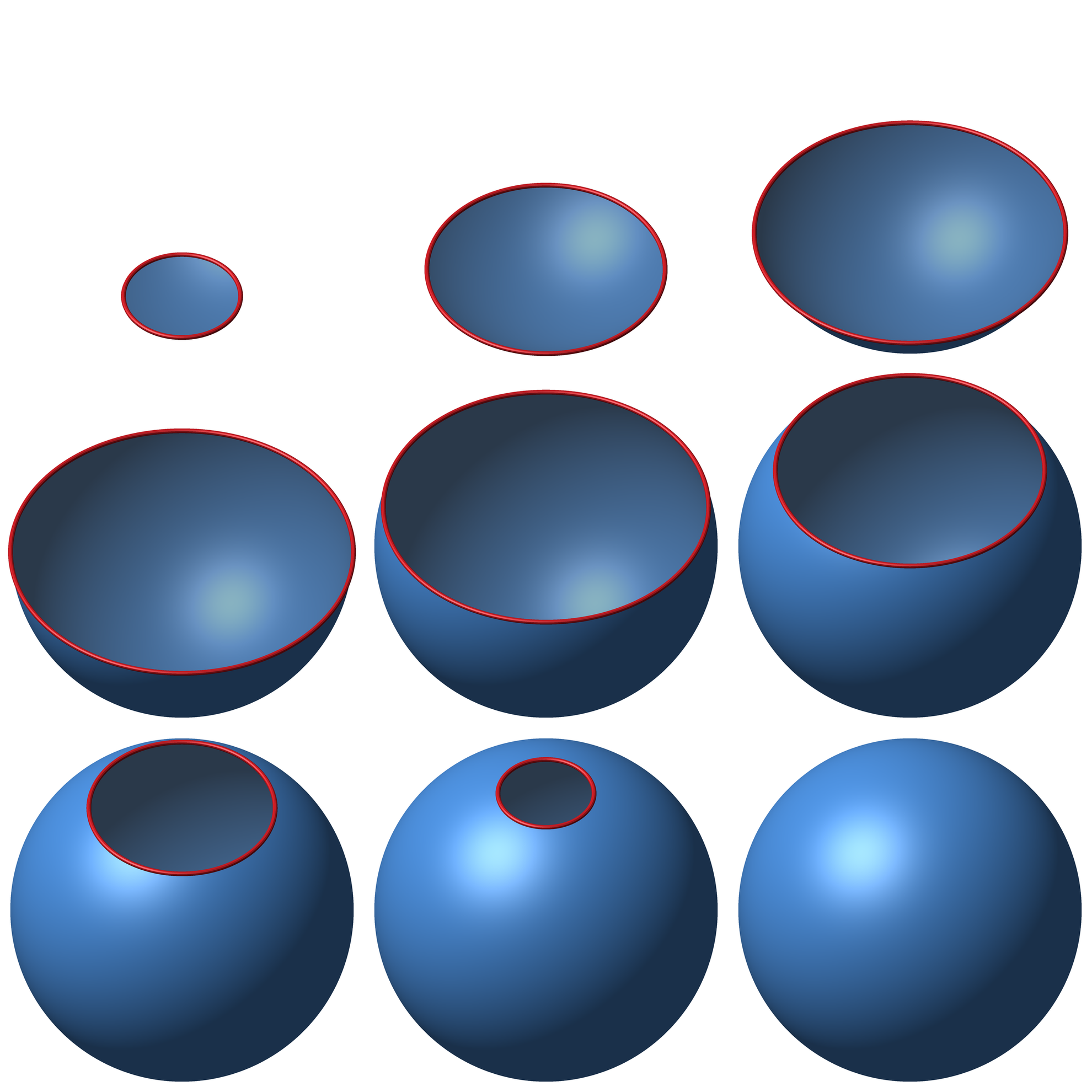}
		\caption{A level-set movie of the unit 2-sphere in 3-space coming from a height function.}
		\label{fig:sphere}
	\end{figure}
\begin{figure}
	\centering
	\includegraphics[scale=0.16]{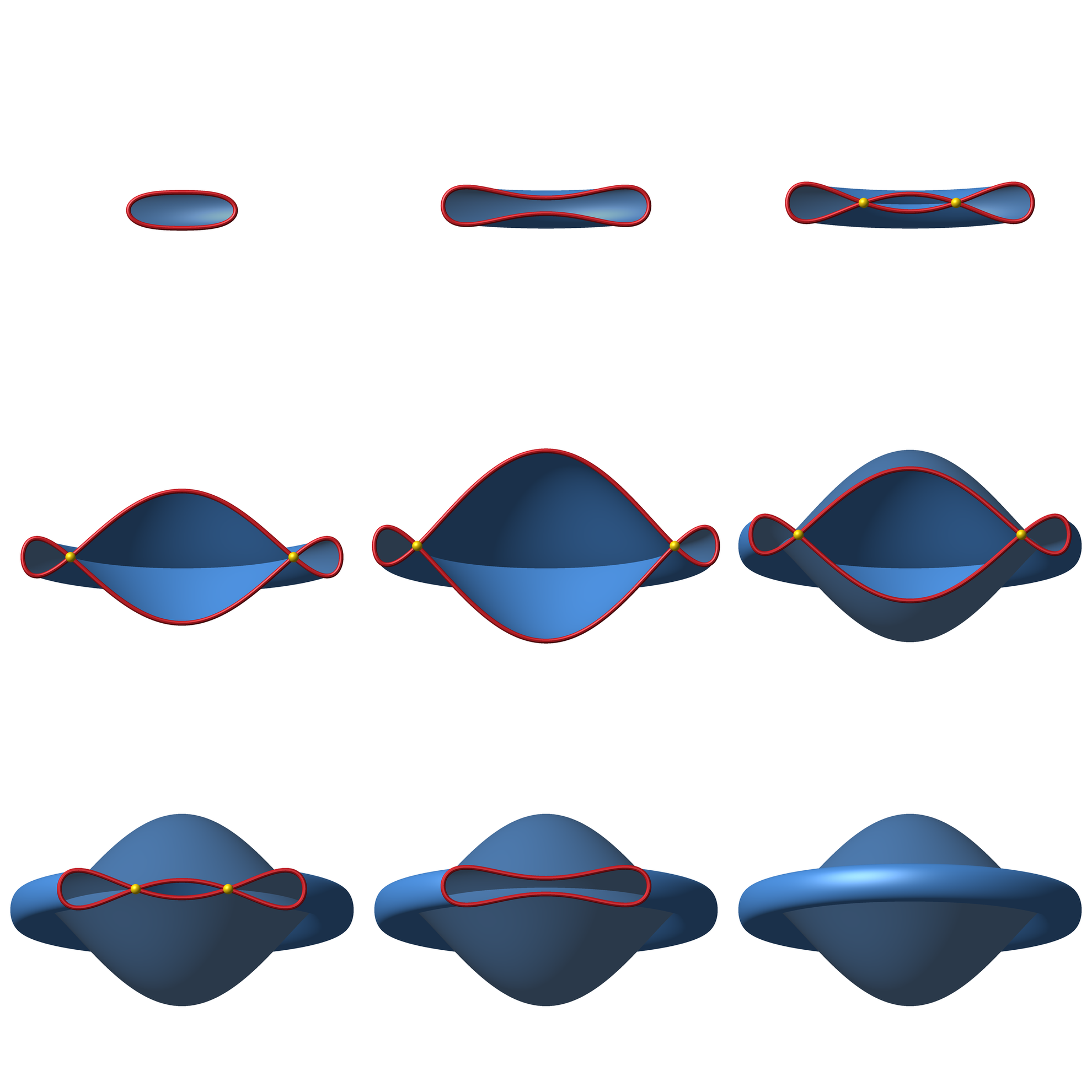}
	\caption{A disastrous finger move has destroyed the filling of this immersed sphere in 3-space.}
	\label{fig:dimple}
\end{figure}
\begin{figure}
	\centering
	\includegraphics[scale=0.16]{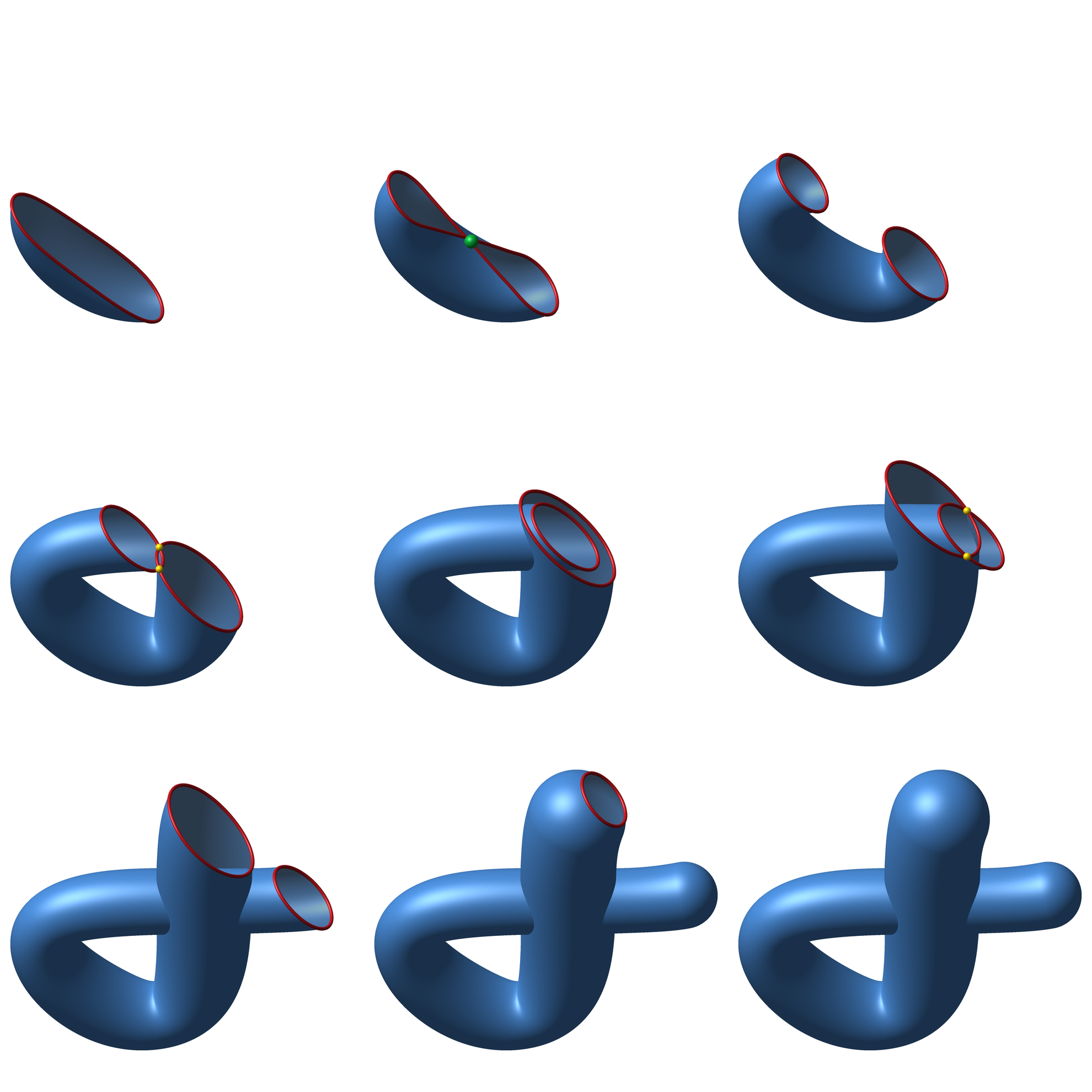}
	\caption{After four finger moves, the filling of this immersed sphere in 3-space has been preserved.}
	\label{fig:tube}
\end{figure}
\begin{figure}
	\centering
	\includegraphics[scale=0.20]{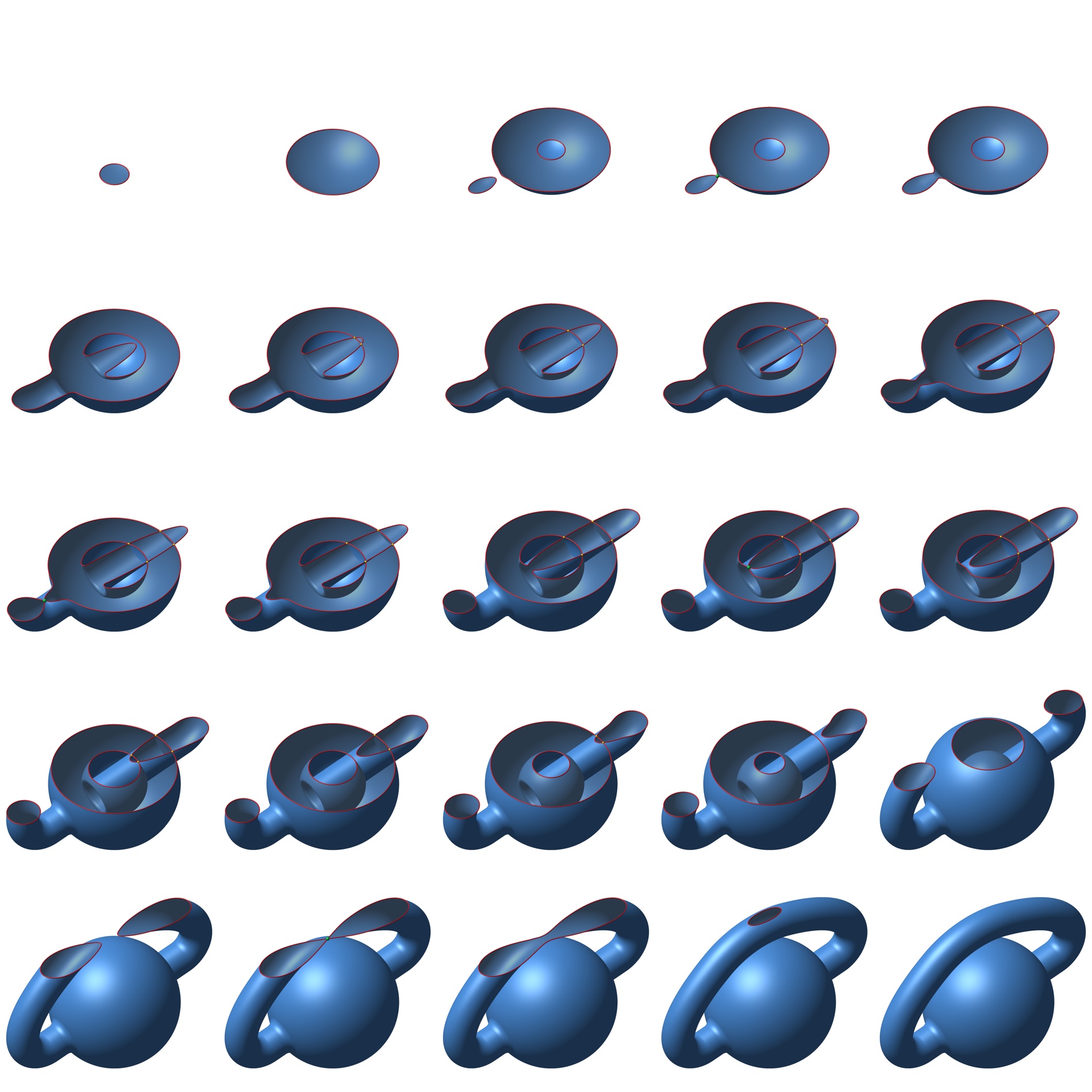}
	\caption{A kettlebell sphere immersion bounding an immersed \(S^2 \times S^1\) less an open ball.}
	\label{fig:s1xs2}
\end{figure}
\begin{figure}
	\centering
	\includegraphics[scale=0.20]{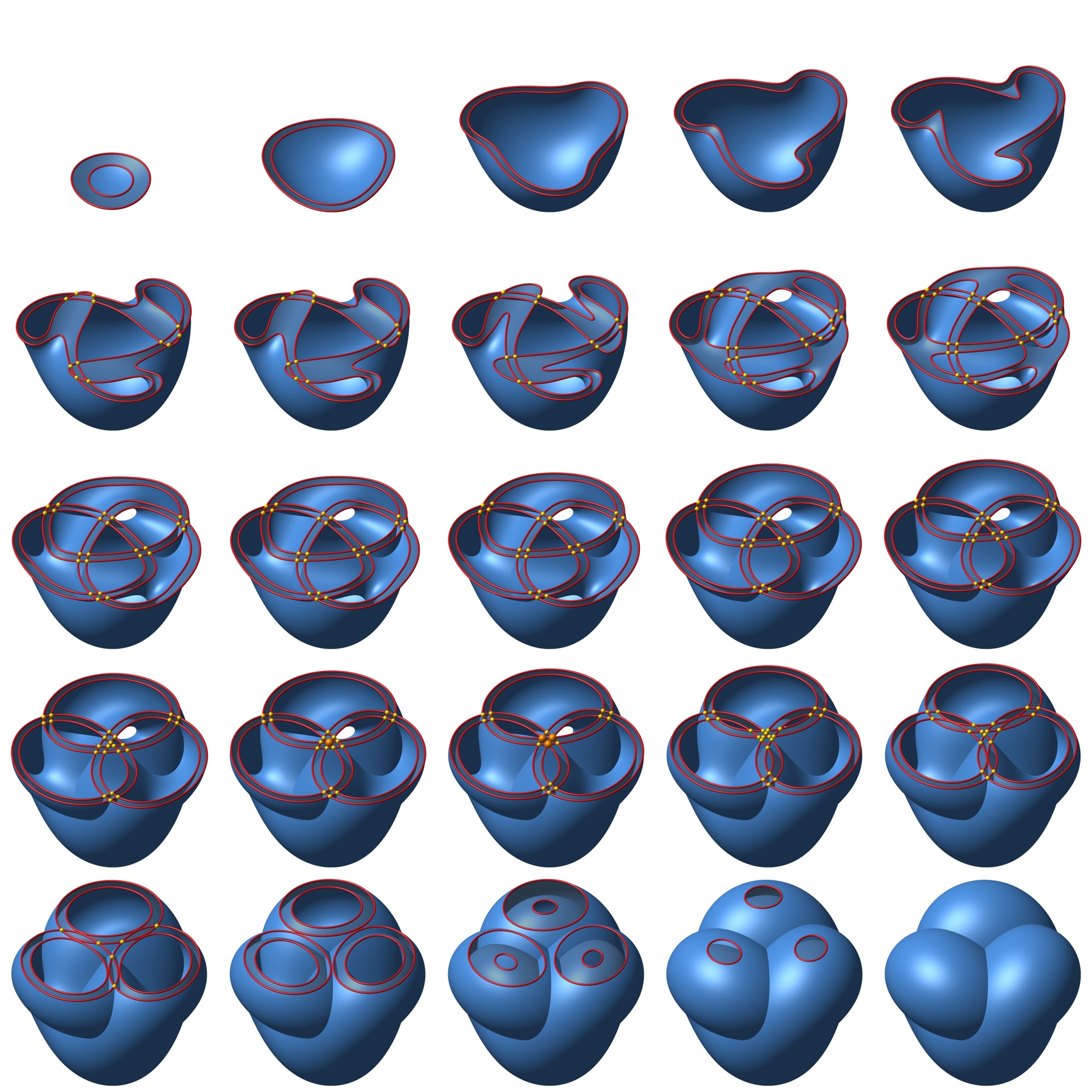}
	\caption{\citeauthor{Giller1982}'s 1982 surface is the boundary of a disk bundle on Boy's surface and bounds an immersed real projective 3-space less an open ball. This is because the 3-manifold it bounds is the disk bundle of the nontrivial line bundle over real projective 2-space.}
	\label{fig:giller}
\end{figure}
\begin{figure}
	\centering
	\includegraphics[scale=0.23]{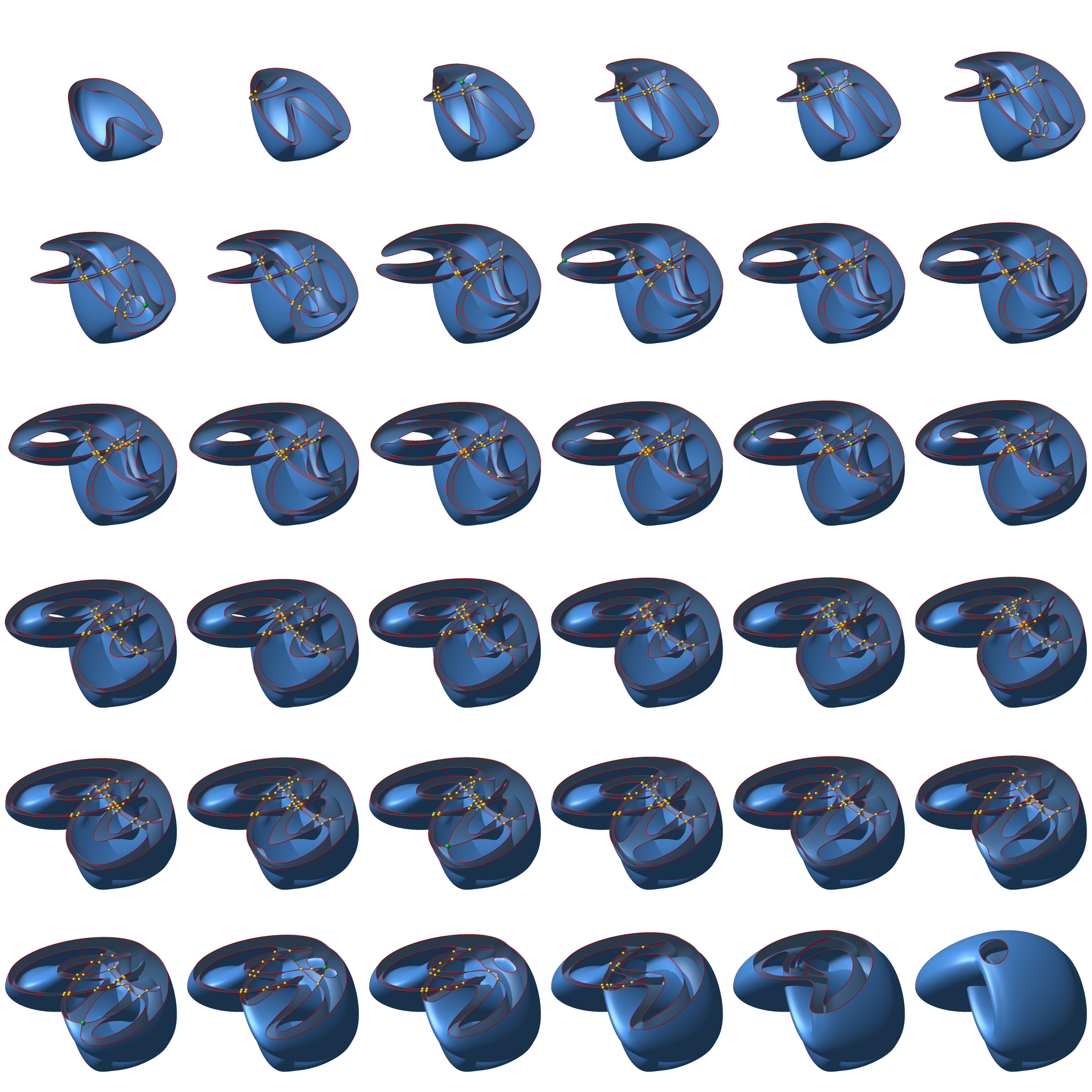}
	\caption{\citeauthor{Ferry1974}'s surface is an immersed sphere that bounds an immersed 3-torus less an open ball. The formula for this immersion was recorded by Steven \cite{Ferry1974}.}
	\label{fig:ferry}
\end{figure}
\begin{figure}
	\centering
	\includegraphics[scale=0.23]{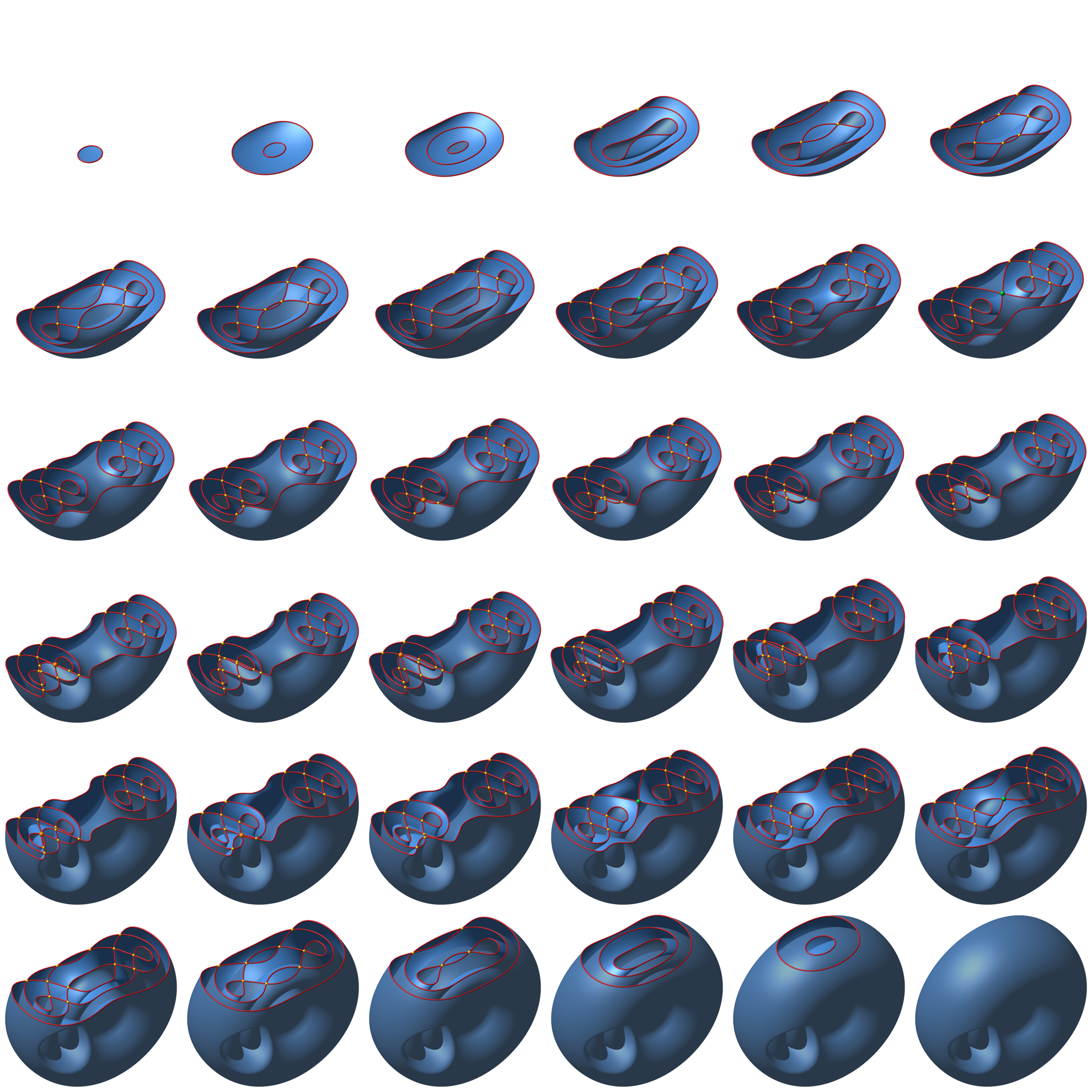}
	\caption{A positive sphere immersion that is not fillable, yet is one finger move away from an immersed sphere whose only fillings are by the 3-ball, in exactly two inequivalent ways. Positivity follows by inspection, while Sandy \citeauthor{blank1967}'s algorithm is applied to frame 22 to prove that no filling by a disk exists. Gauss--Bonnet implies that any planar extension of this immersed circle has Euler characteristic one. \cref{prop:transversal slice} is then applied to obstruct the existence of a filling.}
	\label{fig:milnorpush}
\end{figure}

\bibliographystyle{plainnat}
\bibliography{Dissertation}

\begin{thebibliography}{55}
\providecommand{\natexlab}[1]{#1}
\providecommand{\url}[1]{\texttt{#1}}
\expandafter\ifx\csname urlstyle\endcsname\relax
  \providecommand{\doi}[1]{doi: #1}\else
  \providecommand{\doi}{doi: \begingroup \urlstyle{rm}\Url}\fi

\bibitem[Adams(1966)]{AdamsJGroups}
J.~F. Adams.
\newblock On the groups {$J(X)$}---{I}{V}.
\newblock \emph{Topology}, 5\penalty0 (1):\penalty0 21--71, 1966.

\bibitem[Akbulut and King(1992)]{akbulutking1992}
Selman Akbulut and Henry King.
\newblock \emph{Topology of Real Algebraic Sets}, volume~25 of
  \emph{Mathematical Sciences Research Institute Publications}.
\newblock Springer-Verlag, New York, 1st edition, 1992.

\bibitem[Bailey(1975)]{Bailey1975}
Keith~D. Bailey.
\newblock Extending closed plane curves to immersions of the disk with $n$
  handles.
\newblock \emph{Transactions of the American Mathematical Society},
  206:\penalty0 1--24, 1975.

\bibitem[Bais and Zuddas(2023)]{BaisZuddas2023}
Valentina Bais and Daniele Zuddas.
\newblock On {S}tiefel's parallelizability of 3-manifolds.
\newblock \emph{Expositiones Mathematicae}, 41\penalty0 (1):\penalty0 238--243,
  March 2023.

\bibitem[Banchoff(1974)]{banchoff1974a}
Thomas~F. Banchoff.
\newblock Triple points and singularities of projections of smoothly immersed
  surfaces.
\newblock \emph{Proceedings of the American Mathematical Society}, 46\penalty0
  (3):\penalty0 402--406, December 1974.

\bibitem[Blank(1967)]{blank1967}
Samuel~Joel Blank.
\newblock \emph{Extending Immersions and Regular Homotopies in Codimension 1}.
\newblock PhD thesis, Brandeis University, Waltham, Massachusetts, 1967.

\bibitem[Buchbinder et~al.(2026)Buchbinder, Krishnamoorthy, and
  Vixie]{Buchbinder2026MinimalHomotopies3D}
Lia Buchbinder, Bala Krishnamoorthy, and Kevin~R. Vixie.
\newblock Minimal homotopies in three dimensions: A cable system approach.
\newblock \emph{arXiv preprint arXiv:2605.06565}, May 2026.

\bibitem[Budney and Gabai(2021)]{BudneyGabai2021}
Ryan Budney and David Gabai.
\newblock Knotted 3-balls in {$S^4$}.
\newblock \emph{arXiv preprint arXiv:1912.09029}, April 2021.

\bibitem[Bunge and Funk(2006)]{BungeFunk2006}
Marta Bunge and Jonathon Funk.
\newblock \emph{Singular Coverings of Toposes}, volume 1890 of \emph{Lecture
  Notes in Mathematics}.
\newblock Springer, Berlin, Heidelberg, 2006.

\bibitem[Carrara et~al.(2001)Carrara, Carter, and
  Saito]{carraracartersaito2001}
Vera Carrara, J.~Scott Carter, and Masahico Saito.
\newblock Singularities of the projections of surfaces in 4-space.
\newblock \emph{Pacific Journal of Mathematics}, 199\penalty0 (1):\penalty0
  21--40, May 2001.

\bibitem[Carter and Saito(1997)]{CarterSaito1997Seifert}
J.~Scott Carter and Masahico Saito.
\newblock A {S}eifert algorithm for knotted surfaces.
\newblock \emph{Topology}, 36\penalty0 (1):\penalty0 179--201, 1997.

\bibitem[Cervera et~al.(1991)Cervera, Mascar{\'o}, and
  Michor]{CerveraMascaroMichor1991}
Vicente Cervera, Francisca Mascar{\'o}, and Peter~W. Michor.
\newblock The action of the diffeomorphism group on the space of immersions.
\newblock \emph{Differential Geometry and its Applications}, 1\penalty0
  (4):\penalty0 391--401, 1991.

\bibitem[Ekholm(2001)]{Ekholm2001}
Tobias Ekholm.
\newblock Differential 3-knots in 5-space with and without self-intersections.
\newblock \emph{Topology}, 40\penalty0 (1):\penalty0 157--196, 2001.

\bibitem[Ferry(1974)]{Ferry1974}
Steven Ferry.
\newblock An immersion of \({T}^n-{D}^n\) into \({R}^n\).
\newblock \emph{L'Enseignement Math{\'e}matique}, 20\penalty0 (3--4):\penalty0
  177--178, 1974.

\bibitem[Fox(1957)]{Fox1957}
Ralph~H. Fox.
\newblock Covering spaces with singularities.
\newblock In Ralph~H. Fox, Donald~C. Spencer, and Albert~W. Tucker, editors,
  \emph{Algebraic Geometry and Topology: A Symposium in Honor of S. Lefschetz},
  volume~12 of \emph{Princeton Mathematical Series}, pages 243--257. Princeton
  University Press, Princeton, NJ, 1957.

\bibitem[Francis(1969)]{Francis1969}
George~K. Francis.
\newblock The folded ribbon theorem: A contribution to the study of immersed
  circles.
\newblock \emph{Transactions of the American Mathematical Society},
  141:\penalty0 271--303, 1969.

\bibitem[Francis(1970)]{Francis1970}
George~K. Francis.
\newblock Extensions to the disk of properly nested plane immersions of the
  circle.
\newblock \emph{Michigan Mathematical Journal}, 17:\penalty0 373--383, 1970.

\bibitem[Francis(1973)]{Francis1973}
George~K. Francis.
\newblock Spherical curves that bound immersed discs.
\newblock \emph{Proceedings of the American Mathematical Society}, 41\penalty0
  (1):\penalty0 87--93, 1973.

\bibitem[Francis(1978)]{Francis1978}
George~K. Francis.
\newblock Polymersions with nontrivial targets.
\newblock \emph{Illinois Journal of Mathematics}, 22\penalty0 (1):\penalty0
  161--170, 1978.

\bibitem[Frisch(2010)]{Frisch2010Sphere}
Dennis Frisch.
\newblock Extending immersions into the sphere.
\newblock \emph{arXiv preprint arXiv:1012.4923}, 2010.

\bibitem[Giller(1982)]{Giller1982}
Cole~A. Giller.
\newblock Towards a classical knot theory for surfaces in \(\mathbb{R}^4\).
\newblock \emph{Illinois Journal of Mathematics}, 26\penalty0 (4):\penalty0
  591--631, 1982.

\bibitem[Godement(1958)]{Godement1958}
Roger Godement.
\newblock \emph{Topologie Alg{\'e}brique et Th{\'e}orie des Faisceaux}.
\newblock Number 1252 in Actualit{\'e}s Scientifiques et Industrielles.
  Hermann, Paris, 1958.
\newblock Publications de l'Institut de Math{\'e}matique de l'Universit{\'e} de
  Strasbourg, vol.~13; see especially pp.~109--111.

\bibitem[Gramain(1970{\natexlab{a}})]{Gramain1970a}
Andr{\'e} Gramain.
\newblock Sur les immersions de codimension 1 qui sont des bords.
\newblock \emph{Annales Scientifiques de l'\'Ecole Normale Sup{\'e}rieure},
  3\penalty0 (2):\penalty0 111--184, 1970{\natexlab{a}}.

\bibitem[Gramain(1970{\natexlab{b}})]{Gramain1970b}
Andr{\'e} Gramain.
\newblock Bounding immersions of codimension 1 in the {E}uclidean space.
\newblock \emph{Bulletin of the American Mathematical Society}, 76\penalty0
  (2):\penalty0 361--365, 1970{\natexlab{b}}.

\bibitem[Gunning and Narasimhan(1967)]{GunningNarasimhan1967}
Robert~C. Gunning and Raghavan Narasimhan.
\newblock Immersion of open {R}iemann surfaces.
\newblock \emph{Mathematische Annalen}, 174\penalty0 (2):\penalty0 103--108,
  1967.

\bibitem[Hass(1998)]{Hass1998Algorithms}
Joel Hass.
\newblock Algorithms for knots and 3-manifolds.
\newblock \emph{Chaos, Solitons \& Fractals}, 9:\penalty0 569--581, 1998.

\bibitem[Hironaka(1964)]{Hironaka1964}
Heisuke Hironaka.
\newblock Resolution of singularities of an algebraic variety over a field of
  characteristic zero. {I}, {II}.
\newblock \emph{Annals of Mathematics}, 79:\penalty0 109--326, 1964.

\bibitem[Hirsch(1959)]{hirsch1959}
Morris~W. Hirsch.
\newblock Immersions of manifolds.
\newblock \emph{Transactions of the American Mathematical Society}, 93\penalty0
  (2):\penalty0 242--276, February 1959.

\bibitem[Kirby(1997)]{Kirby1997}
Robion~C. Kirby.
\newblock Problems in low-dimensional topology.
\newblock In William~H. Kazez, editor, \emph{Geometric Topology (Athens, GA,
  1993)}, volume~2 of \emph{AMS/IP Studies in Advanced Mathematics}, pages
  35--473. American Mathematical Society, Providence, RI, 1997.

\bibitem[Kusner(2005)]{kusner2005conformal}
Rob Kusner.
\newblock Conformal structures and necksizes of embedded constant mean
  curvature surfaces.
\newblock In David~A. Hoffman, editor, \emph{Global Theory of Minimal
  Surfaces}, volume~2 of \emph{Clay Mathematics Proceedings}, pages 585--596.
  American Mathematical Society, Providence, RI, 2005.

\bibitem[Li and Barbi{\v{c}}(2018)]{LiBarbic2018}
Yijing Li and Jernej Barbi{\v{c}}.
\newblock Immersion of self-intersecting solids and surfaces.
\newblock \emph{ACM Transactions on Graphics}, 37\penalty0 (4):\penalty0
  45:1--45:14, August 2018.

\bibitem[Marx(1965)]{Marx1965}
Morris~L. Marx.
\newblock Normal curves arising from light open mappings of the annulus.
\newblock \emph{Transactions of the American Mathematical Society},
  120\penalty0 (1):\penalty0 46--56, 1965.

\bibitem[Marx(1968)]{Marx1968}
Morris~L. Marx.
\newblock The branch point structure of extensions of interior boundaries.
\newblock \emph{Transactions of the American Mathematical Society},
  131\penalty0 (1):\penalty0 79--98, 1968.

\bibitem[Marx(1974)]{Marx1974}
Morris~L. Marx.
\newblock Extensions of normal immersions of {$S^1$} into {$R^2$}.
\newblock \emph{Transactions of the American Mathematical Society},
  187:\penalty0 309--326, 1974.

\bibitem[Marx and Verhey(1970)]{MarxVerhey1970}
Morris~L. Marx and Roger~F. Verhey.
\newblock Interior and polynomial extensions of immersed circles.
\newblock \emph{Proceedings of the American Mathematical Society}, 24:\penalty0
  41--49, 1970.

\bibitem[Masuda and Nagase(2001)]{MasudaNagase2001}
Keiichi Masuda and Teruo Nagase.
\newblock A criterion sphere for extensions of a nice immersion of the 2-sphere
  into the {E}uclidean 3-space.
\newblock \emph{Yokohama Mathematical Journal}, 49\penalty0 (1):\penalty0
  17--35, 2001.

\bibitem[Mather(2012)]{Mather2012}
John~N. Mather.
\newblock Notes on topological stability.
\newblock \emph{Bulletin of the American Mathematical Society}, 49\penalty0
  (4):\penalty0 475--506, 2012.

\bibitem[McIntyre(1997)]{McIntyre1997}
Margaret McIntyre.
\newblock Bounding immersed curves.
\newblock \emph{Topology and its Applications}, 77:\penalty0 205--220, 1997.

\bibitem[Pappas(1996)]{Pappas1996}
Christian Pappas.
\newblock Extensions of codimension one immersions.
\newblock \emph{Transactions of the American Mathematical Society},
  348\penalty0 (8):\penalty0 3065--3083, 1996.

\bibitem[Phillips(1967)]{Phillips1967}
Anthony Phillips.
\newblock Submersions of open manifolds.
\newblock \emph{Topology}, 6\penalty0 (2):\penalty0 171--206, April 1967.
\newblock Slightly expanded version of the author's Ph.D. thesis, Princeton
  University, 1966.

\bibitem[Po{\'e}naru(1968)]{Poenaru1968}
Valentin Po{\'e}naru.
\newblock Extension des immersions en codimension 1 (d'apr{\`e}s {S}. {B}lank).
\newblock In \emph{S{\'e}minaire Bourbaki}, number~10, pages 473--505.
  Soci{\'e}t{\'e} Math{\'e}matique de France, 1968.
\newblock Expos{\'e} no. 342.

\bibitem[Poincar\'e(1895)]{poincare1895}
Henri Poincar\'e.
\newblock Analysis situs.
\newblock \emph{Journal de l'\'Ecole Polytechnique (2)}, 1:\penalty0 1--121,
  1895.
\newblock \S 9.

\bibitem[Rubinstein(1995)]{Rubinstein1995}
J.~Hyam Rubinstein.
\newblock An algorithm to recognize the {$3$}-sphere.
\newblock In S.~D. Chatterji, editor, \emph{Proceedings of the {I}nternational
  {C}ongress of {M}athematicians, {V}ol.\ 1, 2 ({Z}\"{u}rich, 1994)}, pages
  601--611, Basel, 1995. Birkh\"{a}user.

\bibitem[Sassen et~al.(2024)Sassen, Schumacher, Rumpf, and Crane]{Sassen2024}
Josua Sassen, Henrik Schumacher, Martin Rumpf, and Keenan Crane.
\newblock Repulsive shells.
\newblock \emph{ACM Transactions on Graphics}, 43\penalty0 (4):\penalty0 1--22,
  2024.

\bibitem[Smale(1959)]{smale1959}
Stephen Smale.
\newblock The classification of immersions of spheres in {E}uclidean spaces.
\newblock \emph{Annals of Mathematics}, 69\penalty0 (2):\penalty0 327--344,
  March 1959.

\bibitem[Spring(2005)]{spring2005}
David Spring.
\newblock The golden age of immersion theory in topology: 1959--1973. {A}
  mathematical survey from a historical perspective.
\newblock \emph{Bulletin of the American Mathematical Society}, 42\penalty0
  (2):\penalty0 163--180, 2005.

\bibitem[Stiefel(1935)]{Stiefel1935}
Eduard Stiefel.
\newblock Richtungsfelder und {F}ernparallelismus in $n$-dimensionalen
  {M}annigfaltigkeiten.
\newblock \emph{Commentarii Mathematici Helvetici}, 8:\penalty0 305--353, 1935.
\newblock Also issued as Dissertation ETH Z{\"u}rich, Nr. 844.

\bibitem[Sullivan(2023)]{Sullivan2023}
Dennis Sullivan.
\newblock Any oriented non-closed connected $4$-manifold can be spread
  holomorphically over the complex projective plane minus a point.
\newblock \emph{Pure Appl. Math. Q.}, 19\penalty0 (6):\penalty0 2915--2918,
  2023.

\bibitem[Thom(1969)]{Thom1969}
Ren{\'e} Thom.
\newblock Ensembles et morphismes stratifi{\'e}s.
\newblock \emph{Bulletin of the American Mathematical Society}, 75\penalty0
  (2):\penalty0 240--284, 1969.

\bibitem[Thompson(1994)]{Thompson1994}
Abigail Thompson.
\newblock Thin position and the recognition problem for {${S}^3$}.
\newblock \emph{Mathematical Research Letters}, 1\penalty0 (5):\penalty0
  613--630, 1994.

\bibitem[Titus(1960)]{Titus1960}
Charles~J. Titus.
\newblock A theory of normal curves and some applications.
\newblock \emph{Pacific Journal of Mathematics}, 10\penalty0 (3):\penalty0
  1083--1096, 1960.

\bibitem[Titus(1961)]{Titus1961}
Charles~J. Titus.
\newblock The combinatorial topology of analytic functions on the boundary of a
  disk.
\newblock \emph{Acta Mathematica}, 106:\penalty0 45--64, 1961.

\bibitem[Whitehead(1961)]{Whitehead1961ImmersionOpen3Manifold}
J.~H.~C. Whitehead.
\newblock The immersion of an open 3-manifold in {E}uclidean 3-space.
\newblock \emph{Proceedings of the London Mathematical Society}, 11\penalty0
  (1):\penalty0 81--90, 1961.

\bibitem[Whitney(1965)]{Whitney1965Tangents}
Hassler Whitney.
\newblock Tangents to an analytic variety.
\newblock \emph{Annals of Mathematics}, 81\penalty0 (3):\penalty0 496--549,
  1965.

\bibitem[Zhao(2019)]{zhao2019}
Bojun Zhao.
\newblock Extensions of immersions of surfaces into $\mathbb{R}^{3}$, 2019.
\newblock Preprint; arXiv:1906.04294.

\end{thebibliography}
\end{document}